\documentclass{amsart}

\usepackage{amsmath}
\usepackage{amsthm}
\usepackage{amssymb}
\usepackage{amsrefs}
\usepackage{graphicx}
\usepackage{color}

\newcommand{ \Z } [0] { \mathbf{Z} }
\newcommand{ \Q } [0] { \mathbf{Q} }
\newcommand{ \R } [0] { \mathbf{R} }
\newcommand{ \C } [0] { \mathbf{C} }
\newcommand{ \defeq } [0] {\stackrel{\textup{\tiny{def}}}{=}}
\newcommand{ \directedisom } [0] {\stackrel{\sim}{\longrightarrow}}
\DeclareMathOperator{\gal}{Gal}
\newcommand{\Qbar} [0] { \overline{\mathbf{Q}} }
\newcommand{\absgalgroup} [0] { \gal ( \Qbar / \Q ) }
\newcommand{ \sphere } [0] { \mathbb{S} }
\newcommand{ \torus } [0] { \mathbb{T} }

\newcommand{ \numcycles } [0] {\mathcal{O}}
\newcommand{ \edges } [0] { \mathcal{E} }
\newcommand{ \vertices } [0] { \mathcal{V} }

\newcommand{ \black } [0] { \bullet }
\newcommand{ \white } [0] { \circ }
\newcommand{ \perm } [0] { \sigma }
\newcommand{ \blackperm } [0] { \perm_{\black} }
\newcommand{ \whiteperm } [0] { \perm_{\white} }
\newcommand{ \opsymbol } [0] {\curlyvee}

\newcommand{ \Belyi } [0] {Bely\u{\i} }

\newcommand{ \twiddle }[0]{\textasciitilde}
\newcommand{ \caret }[0]{\textasciicircum}
\newcommand{ \codefont } [1] {\texttt{#1}}
\newcommand{ \codebox } [1] {\framebox{\parbox[]{\textwidth}{\codefont{\begin{flushleft}#1\end{flushleft}}}}}

\newcommand{\comment}[1]{}

\newcommand{ \Assertion } [0] {\emph{Assertion:} }
\newcommand{ \Assertions } [0] {\emph{Assertions:} }
\newcommand{ \proofbox }[1]{\framebox{\footnotesize{#1}}}

\theoremstyle{plain}
\newtheorem{lem}{Lemma}[section]
\newtheorem{prp}[lem]{Proposition}
\newtheorem{thm}[lem]{Theorem}
\newtheorem*{thmnonum}{Theorem}
\newtheorem{cor}[lem]{Corollary}
\newtheorem{orbittransfer}[lem]{Orbit Transfer Lemma}
\newtheorem{genustheorem}[lem]{Reroute Theorem}
\newtheorem{transitivitytheorem}[lem]{Transitivity Theorem}
\newtheorem*{introdfn}{Definition}
\newtheorem*{introgenustheorem}{Reroute Theorem}
\newtheorem*{introtransitivitytheorem}{Transitivity Theorem}
\newtheorem*{introgenusclassification}{Genus Theorem}
\newtheorem*{introcor}{Corollary}
\newtheorem*{introprp}{Proposition}

\theoremstyle{definition}
\newtheorem{dfn}[lem]{Definition}

\theoremstyle{remark}
\newtheorem*{rmk}{Remark}
\newtheorem*{deletionquestionI}{Deletion Question 1}
\newtheorem*{deletionquestionII}{Deletion Question 2}
\newtheorem{exm}[lem]{Example}

\begin{document}

\title[Conjugation of Transitive Permutation Pairs and Dessins d'Enfants]{Conjugation of Transitive Permutation Pairs \\and Dessins d'Enfants}


\author{Sean Rostami}

\address{
\begin{flushleft}
Syracuse University\\
Mathematics Department\\
Carnegie Building 215\\
Syracuse, NY 13244
\end{flushleft}
}

\email{sjrostam@syr.edu}
\email{sean.rostami@gmail.com}



\begin{abstract}
Let $ E $ be a finite set and $ S_E $ its symmetric group. Given $ \perm, \perm^{\prime} \in S_E $ that together generate a transitive subgroup, for which $ s \in S_E $ is it true that $ s \perm s^{-1}, \perm^{\prime} $ also generate a transitive subgroup? Such transitive permutation pairs encode \emph{dessin d'enfants}, important graph-theoretic objects which are also known to have great arithmetic significance. The absolute Galois group $ \absgalgroup $ acts on dessins d'enfants and permutes them in a very mysterious way. Two dessins d'enfants that share certain elementary combinatorial features are related by conjugations as above, and dessins d'enfants in the same $ \absgalgroup $-orbit share these features and more, so it seems worthwhile to have a good answer to the above question. I classify, relative to $ \perm, \perm^{\prime} $, exactly those transpositions $ s $ for which the new pair is guaranteed to be transitive. I also provide examples of the ``exceptional'' $ s $ which show the range of possible behavior and prove that the above question for the exceptional cases is equivalent to a natural question about deletion in graphs that may have a good answer in this more structured world of topological graphs. Finally, I classify transpositions $ s $ according to how they change the genus of the surface underlying the dessin d'enfant of $ \perm, \perm^{\prime} $. Some of the tools, like the Reroute Operation/Theorem, may have use beyond Dessins d'Enfants.
\end{abstract}

\maketitle

\tableofcontents

\section*{Introduction}

The field of Rational Numbers is denoted $ \Q $ and the field of Complex Numbers is denoted $ \C $. The algebraic closure of $ \Q $ in $ \C $ is denoted $ \Qbar $. The compactification of $ \C $, the Riemann Sphere, is denoted $ \widehat{\C} $. The (topological) sphere is denoted $ \sphere $ and the (topological) torus $ \torus $. For $ E $ a finite set, $ \vert E \vert $ denotes its cardinality and $ S_E $ denotes its symmetric group. For $ \perm \in S_E $, an $ \perm $-orbit is an orbit under the action on $ E $ by the cyclic subgroup generated by $ \perm $, and $ \numcycles ( \perm ) $ denotes the total number of $ \perm $-orbits. For $ \perm \in S_E $, a $ \perm $-cycle is simply a cycle in the disjoint cycle decomposition of $ \perm $. For $ \perm_1, \perm_2 \in S_E $, the pair $ ( \perm_1, \perm_2 ) $ is transitive iff the subgroup of $ S_E $ generated by $ \perm_1, \perm_2 $ is transitive, i.e. iff for any $ a, b \in E $ there is a word $ w $ in $ \perm_1, \perm_2 $ such that $ w ( a ) = b $. For $ \perm, s \in S_E $, abbreviate $ \perm^s \defeq s \cdot \perm \cdot s^{-1} $.

\subsection*{Motivation}

The key reason why transitive permutation pairs are important is that they encode, in a way that facilitates proof and computation, dessins d'enfants, which are important objects in both Number Theory and Topological Graph Theory. Related concepts are cellular embeddings, maps, rotation systems, ribbon graphs, etc.

A \emph{dessin d'enfant} is a triple $ ( D, X, \iota ) $ consisting of a finite bicolored\footnote{``Bicolored'' means that one of two colors is assigned to each vertex and every edge is incident to one vertex of each color. This differs from the notion of ``bipartite'' only in that a choice of color for each vertex is fixed.} graph $ D $, a connected oriented compact surface $ X $ without boundary, and an embedding $ \iota : D \hookrightarrow X $ such that the complement $ X \setminus \iota ( D ) $ is homeomorphic to a union of open discs. For the last requirement, it is necessary but not sufficient that $ D $ is a connected graph. The symbols $ \white $ and $ \black $ will be used to indicate the colors of the vertices, and a vertex will be referred to as either a $ \white $-vertex or a $ \black $-vertex accordingly. Usually the embedding $ \iota $ will be omitted from the notation.

These objects were known, in various slightly different forms, for quite a long time. A newer reason to consider such an object can be found in the following theorem of \Belyi \cite{belyi} and others: \emph{A compact Riemann surface $ S $ is defined over $ \Qbar $ if and only if there is a holomorphic function $ f : S \rightarrow \widehat{\C} $ ramified over at most three values.} If $ ( S, f ) $ is such a \Belyi pair, with the ramification values normalized to be $ 0, 1, \infty \in \widehat{\C} $, then a dessin d'enfant $ ( D, X ) $ is obtained as follows: define $ X $ to be the mere topological surface of $ S $ and define $ D $ to be the preimage in $ X $ under $ f $ of $ [ 0, 1 ] $, where preimages of $ 0 $ are colored by $ \white $, preimages of $ 1 $ are colored by $ \black $, and preimages of $ ( 0, 1 ) $ are edges. Conversely, given a dessin d'enfant, one may construct a \Belyi pair (cf. \S4.2 of \cite{GG}). These two constructions yield, modulo certain natural equivalence relations on each side, a bijection. The number-theoretic significance of this is due to an observation by Grothendieck \cite{sketch}: \emph{The absolute Galois group $ \absgalgroup $ permutes, very mysteriously, the set of dessins d'enfants $ ( D, X ) $ via its permutation of the equivalent objects $ ( S, f ) $.} It is perhaps worth emphasizing that a dessin d'enfant is, superficially, a purely topological object but endows its surface, in particular, with a complex structure. To appreciate this, consider the diversity of complex structures, the elliptic curves, on $ \torus $.

I now describe the role played by permutations. Suppose $ ( D, X ) $ is a dessin d'enfant, and denote by $ \edges $ the set of edges of the graph $ D $. The orientation of $ X $ cyclically orders the edges incident to each $ \white $-vertex, which defines a disjoint cycle decomposition, i.e. a permutation $ \whiteperm \in S_{\edges} $. Similarly, $ ( D, X ) $ defines a permutation $ \blackperm \in S_{\edges} $. The pair $ ( \whiteperm, \blackperm ) $ is called the \emph{monodromy pair} of $ ( D, X ) $, and connectedness of $ D $ implies that the pair is transitive. It is perhaps surprising that, conversely, if $ ( \whiteperm, \blackperm ) $ is a transitive pair then a dessin d'enfant $ ( D, X ) $ can be constructed whose monodromy pair is $ ( \whiteperm, \blackperm ) $. Equivalence classes of these transitive pairs are also in bijection with equivalence classes of dessins d'enfants. The equivalence relation for permutation pairs, simultaneous conjugation, is easy to describe and will be immediately important: $ ( \whiteperm, \blackperm ) $ is equivalent to $ ( \whiteperm^{\prime}, \blackperm^{\prime} ) $ if and only if there is $ s \in S_{\edges} $ such that $ \whiteperm^{\prime} = \whiteperm^s $ and $ \blackperm^{\prime} = \blackperm^s $. Thus, the theory of Dessins d'Enfants is equivalent to the theory of transitive permutation pairs.

The operation of ``non-simultaneous'' conjugation, i.e. the maps $ ( \whiteperm, \blackperm ) \mapsto ( \whiteperm^s, \blackperm ) $ for various $ s \in S_{\edges} $ seems to be important. One reason is that if two dessins d'enfants share certain graph-theoretic data then their monodromy pairs are related by such a conjugation. One of the most elementary facts about the $ \absgalgroup $-action on dessins d'enfants is that two dessins d'enfants in the same $ \absgalgroup $-orbit are indeed related in this way. Another equally elementary fact is that the surfaces of two dessins d'enfants in the same $ \absgalgroup $-orbit have the same genus. So, it seems important ultimately to understand the subset\footnote{Any such subset is a union of double cosets: For any $ C \in C ( \blackperm ) \backslash S_{\edges} / C ( \whiteperm ) $, where $ C ( \perm ) $ is the centralizer of $ \perm $, the effect of $ ( \whiteperm, \blackperm ) \mapsto ( \whiteperm^s, \blackperm ) $ is the same for all $ s \in C $, since pairs are considered modulo simultaneous conjugation.} of $ s \in S_{\edges} $ for which $ ( \whiteperm^s, \blackperm ) $ is again transitive and to understand which among those preserve genus. Such conjugations can also be seen in the action by $ \absgalgroup $ on the \emph{Grothendieck-Teichm\"{u}ller Group}, cf. Proposition 1.6 in \cite{ihara}; it is possible to make precise the connection between these conjugations of the Grothendieck-Teichm\"{u}ller Group and conjugations of permutations pairs.

\subsection*{Results}

\begin{center}
\emph{For locations of anything mentioned here, see the next subsection \emph{Outline}.}
\end{center}

For $ ( D, X ) $ a dessin d'enfant with edges $ \edges $ and monodromy pair $ ( \whiteperm, \blackperm ) $, it is well-known that the connected components of $ X \setminus D $, called \emph{faces}, are in natural correspondence with $ \whiteperm \blackperm $-orbits. Since the $ \white $-vertices and $ \black $-vertices correspond to $ \whiteperm $-orbits and $ \blackperm $-orbits, the Euler Characteristic $ \chi_X $ of $ X $ can be computed directly: $ \chi_X = \numcycles ( \whiteperm ) + \numcycles ( \blackperm ) - \vert \edges \vert + \numcycles ( \whiteperm \blackperm ) $. It is very useful to generalize this, in the most direct way possible, to all pairs. Let $ E $ be a finite set and $ S_E $ its symmetric group. For arbitrary $ \whiteperm, \blackperm \in S_E $, define the \emph{synthetic Euler characteristic} of $ ( \whiteperm, \blackperm ) $ to be 
\begin{equation*}
\chi_{( \whiteperm, \blackperm )} \defeq \numcycles ( \whiteperm ) + \numcycles ( \blackperm ) - \vert E \vert + \numcycles ( \whiteperm \blackperm )
\end{equation*}

Inspired by the well-known formula ``$ \chi = 2 - 2 g $'', define the \emph{synthetic genus} of $ ( \whiteperm, \blackperm ) $ to be 
\begin{equation*}
g_{( \whiteperm, \blackperm )} \defeq 1 - \chi ( \whiteperm, \blackperm ) / 2
\end{equation*}

There is an operation $ \opsymbol $ that seems to be well-adapted to the question, which is also a variation on and generalization of the operation of ``edge sliding'' from Topological Graph Theory, cf. \S3.3.3 of \cite{GT}.

\begin{introdfn}
For each distinct pair $ a, b \in E $, there is a \emph{reroute} operation $ \opsymbol $ on $ S_E \times S_E $. The idea of $ \opsymbol $ is to ``unplug'' edge $ a $ from its $ \white $-vertex and reconnect it to the $ \white $-vertex of $ b $. Any non-simultaneous conjugation $ ( \whiteperm, \blackperm ) \mapsto ( \whiteperm^s, \blackperm ) $ is essentially achieved by repeated application of $ \opsymbol $ relative to various edges which are easy to read from $ s $.
\end{introdfn}

The following Definition and Theorem, and its proof, are the foundation of the main conclusions.

\begin{introdfn}
Relative to $ ( a, b ) $, call a pair $ ( \whiteperm, \blackperm ) $ in $ S_E $ 
\begin{itemize}
\setlength{\itemsep}{5pt}
\item[-] \emph{Type U (Unoriented)} iff $ a, \whiteperm ( a ), b $ are distinct and no two of them represent the same $ \whiteperm \blackperm $-orbit.

\item[-] \emph{Type N (Negatively Oriented)} iff $ a, \whiteperm ( a ), b $ represent the same $ \whiteperm \blackperm $-orbit and the cycle containing them is of the form\footnote{Here, $ \whiteperm ( a ) \neq b $ is required but $ \whiteperm ( a ) = a $ is allowed. The right way to talk about this, valuable elsewhere, is via the notion of \emph{arc}, cf. Definition \ref{DFNarcs}. But this omitted from the Introduction.} $ ( \ldots a \ldots b \ldots \whiteperm ( a ) \ldots ) $.

\item[-] \emph{Type P (Positively Oriented)} iff it is neither\footnote{``Type P'' can be described directly, cf. Definition \ref{DFNsubtypes}, but this is omitted from the Introduction.} Type U nor Type N.
\end{itemize}
\end{introdfn}

\begin{introgenustheorem}
Let $ ( \whiteperm, \blackperm ) $ be an arbitrary pair in $ S_E $, not necessarily transitive. Let $ ( \whiteperm^{\opsymbol}, \blackperm^{\opsymbol} ) $ be the reroute of $ ( \whiteperm, \blackperm ) $ relative to $ ( a, b ) $. Let $ g $ be the synthetic genus of $ ( \whiteperm, \blackperm ) $ and $ g^{\opsymbol} $ that of $ ( \whiteperm^{\opsymbol}, \blackperm^{\opsymbol} ) $.

\begin{enumerate}
\setlength{\itemsep}{5pt}
\item If $ ( \whiteperm, \blackperm ) $ is Type U relative to $ ( a, b ) $ then $ g^{\opsymbol} = g + 1 $.

\item If $ ( \whiteperm, \blackperm ) $ is Type N relative to $ ( a, b ) $ then $ g^{\opsymbol} = g - 1 $.

\item If $ ( \whiteperm, \blackperm ) $ is Type P relative to $ ( a, b ) $ then $ g^{\opsymbol} = g $.
\end{enumerate}
\end{introgenustheorem}

By studying what happens after repeated application of the reroute operation $ \opsymbol $, one can conclude the following answer to the original question:

\begin{introdfn}
Relative to $ ( a, b ) $, call a pair $ ( \whiteperm, \blackperm ) $ in $ S_E $ \emph{Exceptional} iff $ a, \whiteperm ( a ), b, \whiteperm ( b ) $ are distributed into $ \whiteperm \blackperm $-cycles in one of the following ways:
\begin{itemize}
\setlength{\itemsep}{5pt}

\item[-]  $ ( \ldots \whiteperm ( a ) \ldots \whiteperm ( b ) \ldots b \ldots a \ldots ) $ with $ \whiteperm ( a ) \neq a $

\item[-] $ ( \ldots \whiteperm ( b ) \ldots \whiteperm ( a ) \ldots a \ldots b \ldots ) $ with $ \whiteperm ( b ) \neq b $

\item[-] $ ( \ldots a \ldots b \ldots ), ( \ldots \whiteperm ( a ) \ldots \whiteperm ( b ) \ldots ) $

\item[-] $ ( \ldots a \ldots \whiteperm ( a ) \ldots ), ( \ldots b \ldots \whiteperm ( b ) \ldots ) $
\end{itemize}
The last of these is called \emph{Wild Exceptional}, for reasons that are explained later.
\end{introdfn}

\begin{introtransitivitytheorem}
Let $ ( D, X ) $ be a dessin d'enfant with edges $ \edges $ and monodromy pair $ ( \whiteperm, \blackperm ) $. Fix distinct $ a, b \in \edges $ and let $ t \in S_{\edges} $ be the transposition exchanging $ a $ and $ b $.

\begin{enumerate}
\setlength{\itemsep}{5pt}

\item If $ ( \whiteperm, \blackperm ) $ is not Exceptional relative to $ ( a, b ) $ then $ ( \whiteperm^t, \blackperm ) $ is transitive.

\item If $ ( \whiteperm, \blackperm ) $ is Exceptional then examples show that $ ( \whiteperm^t, \blackperm ) $ may or may not be transitive, depending on the truth of certain connectivity properties of $ D $.
\end{enumerate}
\end{introtransitivitytheorem}

After using the Reroute Theorem to analyze the Exceptional cases, the Transitivity Theorem implies:

\begin{introcor}
Let everything be as in the Transitivity Theorem.

If $ \numcycles ( \whiteperm^t \blackperm ) < \numcycles ( \whiteperm \blackperm ) $ then $ ( \whiteperm^t, \blackperm ) $ is transitive.
\end{introcor}

By imposing additional hypotheses, some converses are gained. For example:

\begin{introcor}
Let everything be as in the Transitivity Theorem, but assume that $ X = \sphere $.

If $ ( \whiteperm, \blackperm ) $ is not Wild Exceptional relative to $ ( a, b ) $ then: $ ( \whiteperm^t, \blackperm ) $ is transitive if and only if $ \numcycles ( \whiteperm^t \blackperm ) \leq \numcycles ( \whiteperm \blackperm ) $.
\end{introcor}

A strong statement can be given in the admittedly narrow class of trees:

\begin{introprp}
Let everything be as in the Transitivity Theorem, but assume $ D $ is a tree.

For all $ s \in S_{\edges} $, $ ( \whiteperm^s, \blackperm ) $ is transitive if and only if $ \numcycles ( \whiteperm^s \blackperm ) = 1 $.
\end{introprp}

Finally, transpositions are classified according to how they change synthetic genus.

\begin{introgenusclassification}
Let $ ( \whiteperm, \blackperm ) $ be an arbitrary pair in $ S_E $, not necessarily transitive. Using descriptions similar in flavor to those of ``Type'' or ``Exceptional'', transpositions $ t \in S_E $ are classified according to whether $ ( \whiteperm^t, \blackperm ) $ has higher, lower, or equal synthetic genus than $ ( \whiteperm, \blackperm ) $. Regardless, the synthetic genus may change by at most $ 1 $.
\end{introgenusclassification}

Overall, the approach is somewhat messy and the reader may reasonably ask if there is a significantly more elegant approach. I think there is not, due to the specificity and opaqueness of (a) the notion of ``Type'', (b) the Exceptional cases, and (c) the classification according to synthetic genus.

In the future, I hope to understand \emph{general} conjugations as completely as those by transpositions. On the other hand, I think the facts here are sufficient to allow work on a genuinely Galois-theoretic question, restricted to the case of quadratic extensions.

\subsection*{Outline}

In \S\ref{Snotation}, I set some notation and recall a few standard facts from the subjects concerned. I also make explicit some conventions that may not be standard.

In \S\ref{Ssyntheticgenus}, I formalize some notions that are likely variations on things that are very well-known. One, a dessin d'famille, is a natural generalization of a dessin d'enfant which will be very useful (\ref{DFNdessindnichee}). Additionally, I provide a few basic tools to go with these notions, like their relationship to permutations (\ref{DFNmonodromyofnichees}, \ref{PRPconstructionofmodels}) and a similarly generalized notion of \emph{genus} (\ref{DFNsyntheticgenus}). As a bonus, a well-known fact about dessins d'enfants whose precise statement and proof does not seem to appear in the literature is generalized and proved (\ref{THMcyclesandboundary}).

In \S\ref{Sdeletion}, I give a nice description, in terms of the monodromy pair, of those edges of a dessin d'enfant which border only one face instead of two (\ref{LEMfaceincidence}). For a dessin d'enfant in $ \sphere $, it is equivalent to say that deletion of the edge results in a disconnected graph, but for dessins d'enfants in surfaces of higher genus, disconnection is merely a sufficient condition. I do not understand at this time how to characterize, in terms of the monodromy pair, those edges whose deletion results in a disconnected graph. An analogous question appears in \S\ref{SStameexceptional}, and a good answer to it would significantly improve the Transitivity Theorem.

In \S\ref{Sarcs}, which is very short, I define the slightly unusual concept of \emph{arc} (\ref{DFNarcs}). Given an element $ \perm $ of a group acting on a set, an arc is essentially a half-open interval in a $ \perm $-orbit, after arranging the orbit as a circuit with $ \perm ( x ) $ following $ x $ for every $ x $ in the orbit. Arcs are used many times in the rest of the paper.

In \S\ref{Sfundamentaloperation}, I define the reroute operations $ \opsymbol $ on $ S_E \times S_E $ (\ref{DFNfundamentaloperation}). Every choice of distinct $ a, b \in E $ yields a different operation, and for any $ s \in S_E $ it is possible to choose such pairs in $ E $ so that $ ( \whiteperm^s, \blackperm ) $ is the same as performing in succession the reroute operations relative to the chosen pairs (\ref{PRPconjugationviaoperations}). By using the concept of \emph{arc}, $ S_E \times S_E $ can be perfectly partitioned (\ref{DFNtypes}) so as to predict exactly how $ \opsymbol $ will change (\ref{THMgenus}) the synthetic genus of a pair.

In \S\ref{Siteration}, I use both the statements and the proofs from \S\ref{Sfundamentaloperation} to study the repeated application of the $ \opsymbol $ operation. The results are not conceptual, and are presented essentially as a database to be exploited heavily in \S\ref{Sclassificationbytransitivity} and \S\ref{Sclassificationbygenus}. The concept of \emph{arc} is again valuable here, allowing a very annoying amount of seemingly special cases to be unified. The \emph{exceptional} classes of permutation pairs, those for which the conclusion of the Transitivity Theorem is not certain, are defined here (\ref{DFNtameexceptional1A}, \ref{DFNtameexceptional1B}, \ref{DFNwildexceptional}, \ref{DFNtameexceptional2}).

In \S\ref{Sclassificationbytransitivity}, the Transitivity Theorem is stated and proved (\ref{THMtransitivity}). Examples are given which illustrate the range of behavior that the exceptional pairs may exhibit. Finally, it is proved that the conclusion of the Transitivity Theorem in the exceptional cases is equivalent to a certain connectivity property which is perhaps closer to ``pure'' Graph Theory than most other things in this paper (\ref{PRPgeneralconnvswalks}, \ref{PRPtameexctransiffwalks}).

In \S\ref{Sclassificationbygenus}, I give an explicit description of permutation pairs according to how genus will change after conjugating by a transposition (\ref{PRPgenusraising}, \ref{PRPgenuslowering}). The results are again not conceptual, and are mostly just a consolidation of the database from \S\ref{Siteration}. The concept of \emph{arc} is valuable here too.

In the Appendix, some \texttt{MAGMA} functions are provided. Due to the complexity in \S\ref{Sfundamentaloperation}, \S\ref{Siteration}, \S\ref{Sclassificationbytransitivity} it seemed appropriate to check the conclusions by computer in a reasonably large symmetric group. These functions were used to do this.

Many examples and pictures are provided throughout the paper. In the spirit of Dessins d'Enfants, and taking into account the familiarity that today's children have with computers, all pictures were drawn by hand using very rudimentary paint software.

\subsection*{Acknowledgements}

Most material here was generated from the summer of 2016 to the winter of 2016/2017. However, my introduction to the subject, the decision to pursue this question, and some important early progress that informed the overall direction, occurred while I was a postdoc at University of Wisconsin and while I was a fellow at the Mathematical Sciences Research Institute (DMS-1440140) in Fall 2014.

I thank Nathan Clement and Ted Dewey for many good conversations about the subject while we were at UW. I thank Brian Hwang for the same, at MSRI and at Cornell. I thank Tom Haines, for his interest in the project and for encouraging me to finish it. I thank Jack Graver and Mark Watkins, for their friendliness to me as a new member of Syracuse University, and for good conversations about Graph Theory. Finally, I thank the computing staff at UW, especially Sara Nagreen and John Heim, for their assistance with \texttt{MAGMA} even after my position at UW ended.

\section{Notation and Conventions} \label{Snotation}

The cardinality of a set $ E $ is denoted $ \vert E \vert $. The sphere is denoted $ \sphere $ and the torus is denoted $ \torus $. In examples/pictures below, $ \sphere $ is always presented as the plane, with the reader expected to imagine the point at infinity, and the orientation is always ``counterclockwise''. In examples/pictures below, $ \torus $ is always oriented by ``right-hand-rule from the outside''. For $ X $ a topological space, $ \pi_0 ( X ) $ denotes the set of connected components of $ X $. For a group $ \Gamma $ acting on a set $ E $, a subset $ F \subset E $ is \emph{$ \Gamma $-stable} iff $ g \cdot x \in F $ for all $ x \in F $; when $ \Gamma $ is cyclic and generated by $ \gamma $, such a subset is called \emph{$ \gamma $-stable} instead. For finite sequences $ ( x_1, x_2, \ldots, x_n ) $, \emph{rotation} is the operation $ ( x_1, x_2, \ldots, x_n ) \mapsto ( x_2, \ldots, x_n, x_1 ) $  and \emph{reversal} is the operation $ ( x_1, x_2, \ldots, x_n ) \mapsto ( x_n, \ldots, x_2, x_1 ) $.

\subsection{Permutations}

For $ E $ a finite set, $ S_E $ denotes its symmetric group. For $ \perm, s \in S_E $, the conjugate $ s \cdot \perm \cdot s^{-1} $ is abbreviated to $ \perm^s $.

A \emph{$ \perm $-cycle} is a cycle in the disjoint cycle decomposition of $ \perm $. \emph{Trivial cycles (fixed points) are always considered to be legitimate cycles, so the reader must be careful about the sense in which a permutation is considered to be ``a cycle''.} Cycle notation is used in the customary way: the cycle $ ( a, b, c ) $ sends $ a $ to $ b $ etc. The operation in $ S_E $ is ``functional'', so applying $ \perm_1 \perm_2 $ to $ e \in E $ results in $ \perm_1 ( \perm_2 ( e ) ) $.

For $ \perm \in S_E $, an \emph{$ \perm $-orbit} is the same as an $ \langle \perm \rangle $-orbit, where $ \langle \perm \rangle \subset S_E $ is the subgroup generated by $ \perm $. For $ \perm \in S_E $, the quantity of $ \perm $-orbits (equivalently, the quantity of $ \perm $-cycles) is denoted $ \numcycles ( \perm ) $.

For $ x, y \in S_E $, the pair $ ( x, y ) $ is \emph{transitive} iff the generated subgroup $ \langle x, y \rangle \subset S_E $ is transitive on $ E $.

\subsection{Graphs}

\begin{center}
\emph{For more details of everything in this subsection, consult the very excellent book \cite{GT}.}
\end{center}

By abuse of terminology, the term ``graph'' will always mean what is more commonly called a ``multigraph'': it is allowed that there are multiple edges incident to the same pair of vertices. A graph is \emph{nondegenerate}\footnote{I am not aware of any standard terminology for this restriction, although it also appears in some key literature, e.g. \cite{HR}.} iff every vertex is incident to at least one edge, and \emph{degenerate} otherwise. For $ G $ a graph and $ e $ an edge, $ G \setminus e $ denotes the subgraph obtained by deleting the edge $ e $: $ G \setminus e $ has the same vertices as $ G $ and all edges of $ G $ except $ e $. \emph{Note that $ G \setminus e $ may be degenerate even if $ G $ was nondegenerate.}

For $ G $ a graph and $ x, y $ vertices, a \emph{walk} from $ x $ to $ y $ means the customary thing: a sequence $ v_0, e_1, v_1, \ldots, e_n, v_n $ with $ v_i $ vertices, $ e_i $ edges such that $ e_i $ is incident to $ v_{i-1} $ and $ v_i $ for all $ 0 < i \leq n $, $ v_0 = x $, and $ v_n = y $. A graph is \emph{connected} iff there is a walk from $ x $ to $ y $ for all vertices $ x, y $. The set of connected components of a graph $ G $ is denoted $ \pi_0 ( G ) $.


Let $ \white $ and $ \black $ be formal symbols, fixed throughout the paper. For a graph $ G $ with vertices $ \vertices $, a \emph{coloring} is a function $ \vertices \rightarrow \{ \white, \black \} $. If a coloring is fixed then $ v \in \vertices $ is called a $ \white $-vertex (resp. $ \black $-vertex) iff the image of $ v $ under $ \vertices \rightarrow \{ \white, \black \} $ is $ \white $ (resp. $ \black $).

A \emph{bicolored graph} is a pair $ ( G, f ) $ where $ G $ is a graph and $ f $ is a coloring such that every edge is incident to both a $ \white $-vertex and a $ \black $-vertex. \emph{This differs from the notion of ``bipartite'' only in that a choice of color for each vertex is fixed.} For $ e $ an edge, its $ \white $-vertex will be denoted $ \white_e $ and its $ \black $-vertex $ \black_e $. Note that bicolored graphs have no loops. Throughout the rest of the paper, the coloring function will be suppressed from the notation without exception.

\subsection{Dessins d'Enfants} \label{SSdessins}

\begin{center}
\emph{For more details of everything in this subsection, consult the very excellent book \cite{GG}.}
\end{center}

A \emph{dessin d'enfant} is a triple $ ( D, X, \iota ) $ with $ X $ a connected oriented compact surface without boundary, $ D $ a finite bicolored nondegenerate\footnote{Since $ D $ is necessarily connected, this extra condition really only excludes one trivial case: a single vertex in $ \sphere $. Nonetheless, \emph{something} must be assumed, and nondegeneracy seems the best expression.} graph, $ \iota $ an embedding $ D \hookrightarrow X $, and it is required that $ X \setminus \iota ( D ) $ is homeomorphic to a union of open discs, each of which is called a \emph{face}. Necessarily, $ D $ is a connected graph. Usually the embedding $ \iota $ will be omitted from the notation. Dessins d'enfants $ ( D_1, X_1 ) $ and $ ( D_2, X_2 ) $ are \emph{isomorphic} iff there is an orientation-preserving homeomorphism $ X_1 \rightarrow X_2 $ which induces a graph isomorphism $ D_1 \directedisom D_2 $.

Associated with a dessin d'enfant $ ( D, X ) $ is a pair $ ( \whiteperm, \blackperm ) $, called the \emph{monodromy pair}\footnote{This pair indeed defines a representation of a fundamental group -- see \S4.3.1 of \cite{GG}.}, in $ S_{\edges} $, where $ \edges $ are the edges of $ D $. Necessarily, the pair $ ( \whiteperm, \blackperm ) $ is transitive. The function $ ( D, X ) \mapsto ( \whiteperm, \blackperm ) $ induces a \emph{bijection} between isomorphism classes of dessins d'enfants and equivalence classes of transitive pairs in $ S_{\edges} \times S_{\edges} $ modulo the equivalence relation of ``simultaneous conjugation'', i.e. $ ( \whiteperm^{\prime}, \blackperm^{\prime} ) $ is equivalent to $ ( \whiteperm, \blackperm ) $ iff there is $ s \in S_{\edges} $ such that $ \whiteperm^{\prime} = \whiteperm^s $ and $ \blackperm^{\prime} = \blackperm^s $. By construction of the function $ ( D, X ) \mapsto ( \whiteperm, \blackperm ) $, the $ \white $-vertices of $ D $ are in natural bijection with $ \whiteperm $-orbits and the $ \black $-vertices with $ \blackperm $-orbits. It is also true, though less obvious, that the faces of $ ( D, X ) $ are in natural bijection with $ \whiteperm \blackperm $-orbits. If $ ( D, X ) $ is a dessin d'enfant with $ \edges $ the edges of $ D $ then it follows that the Euler characteristic $ \chi $ of $ X $ can be computed by the following formula: $ \chi = \numcycles ( \whiteperm ) + \numcycles ( \blackperm ) - \vert \edges \vert + \numcycles ( \whiteperm \blackperm ) $. An important feature of the bijection is that if $ F $ is a face and $ O $ is the corresponding $ \whiteperm \blackperm $-orbit then the edges bordering $ F $ are $ O \cup \blackperm ( O ) $. A precise statement and proof of the correspondence between faces and $ \whiteperm \blackperm $-orbits seems not to appear in print, and a generalization of it will be needed anyway, so a proof is included in \S\ref{Ssyntheticgenus} of this article, using the content of \cite{HR}.

\section{Dessins d'Familles and Genus} \label{Ssyntheticgenus}

\begin{center}
\emph{Most of this section is, on some essential level, well-known and not new at all. However, some things do not appear in print and other things are not tailored to the goals here. So, \S\ref{Ssyntheticgenus} is used to set some terminology and record some basic facts.}
\end{center}

It will be necessary to work with something more general than a dessin d'enfant:

\begin{dfn}[Dessin d'Famille] \label{DFNdessindnichee}
A \emph{dessin d'famille} is a triple $ ( G, X, \jmath ) $ where $ X $ is a connected oriented compact surface without boundary, $ G $ is a bicolored graph, and $ \jmath : G \hookrightarrow X $ is an embedding. It is \emph{nondegenerate} iff $ G $ is nondegenerate. Usually the embedding $ \jmath $ will be omitted from the notation.
\end{dfn}

Note that it is not assumed that $ X \setminus G $ is homeomorphic to a finite union of open discs, nor even that $ G $ is connected. Nonetheless, the surface $ X $ allows one to extract from $ G $ something like the monodromy pair of a dessin d'enfant:

\begin{dfn}[Monodromy] \label{DFNmonodromyofnichees}
For a dessin d'famille $ ( G, X ) $ with $ \edges $ the edges of $ G $, the \emph{monodromy pair} of $ ( G, X ) $ is the pair $ ( \whiteperm, \blackperm ) $ in $ S_{\edges} $ where $ \whiteperm $ is the permutation expressing the cyclic ordering of the edges incident to each $ \white $-vertex according to the orientation of $ X $ and $ \blackperm $ is the analogous permutation relative to the $ \black $-vertices.

\emph{If desired, this can be made rigorous using the Neighborhood Theorem 3.1 in \cite{HR}.}
\end{dfn}

An easy but important fact is the equivalence of transitivity and connectedness:

\begin{lem}[Transitivity/Connectedness] \label{LEMtransiffconn}
If $ ( G, X ) $ is a dessin d'famille and $ G $ is connected then its monodromy pair $ ( \whiteperm, \blackperm ) $ is transitive. If $ G $ is nondegenerate then the converse is true.
\end{lem}

The idea here is nearly identical to that for dessins d'enfants.

\begin{proof}
For $ a, b \in \edges $, Definition \ref{DFNmonodromyofnichees} implies the following: $ a $ and $ b $ share their $ \white $-vertex (resp. $ \black $-vertex) if and only if $ a $ and $ b $ represent the same $ \whiteperm $-orbit (resp. $ \blackperm $-orbit). Let $ a, b \in \edges $ be arbitrary. Let $ x $ be a vertex of $ a $ and $ y $ a vertex of $ b $. By hypothesis, there is a walk in $ G $ from $ x $ to $ y $. By applying the initial observation inductively along this walk, there is $ w \in \langle \whiteperm, \blackperm \rangle $ such that $ w ( a ) = b $. Now consider the converse statement. If $ a, b \in \edges $ represent the same $ \whiteperm $-orbit or $ \blackperm $-orbit then the initial observation implies that $ a $ and $ b $ are contained in the same connected component. This implies that the edges of any connected component of $ G $ are stabilized by both $ \whiteperm $ and $ \blackperm $. By assumption of nondegeneracy, any component must contain at least one edge, so transitivity implies that there cannot be more than one component.
\end{proof}

There is a subtle but important difference between a connected dessin d'famille and a dessin d'enfant. If $ ( G, X, \jmath ) $ is a dessin d'famille with edges $ \edges $ and $ G $ is connected then, by Lemma \ref{LEMtransiffconn}, its monodromy pair $ ( \whiteperm, \blackperm ) $ is transitive. As mentioned in \S\ref{Snotation}, a dessin d'enfant can be constructed from this transitive pair $ ( \whiteperm, \blackperm ) $. The subtlety is that the topological surface of this dessin d'enfant can be different from $ X $, although the underlying graph is the same and its embedding has much in common with $ \jmath $. I like to say that a connected dessin d'famille is ``a dessin d'enfant in the wrong surface''. The following example should make the idea clear:

\begin{exm}
Consider the dessin d'famille $ ( G, \torus ) $ in the following picture:
\begin{center}
\includegraphics[scale=0.4]{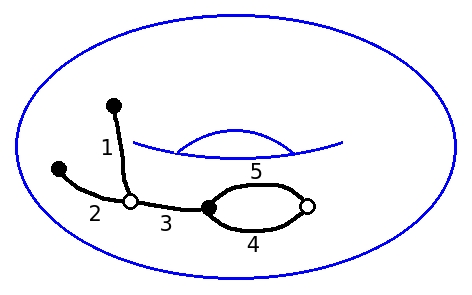}
\end{center}
This $ G $ is connected but $ ( G, \torus ) $ is not a dessin d'enfant, since the complement $ \torus \setminus G $ has a connected component which is not simply connected. The monodromy pair $ ( \whiteperm, \blackperm ) $ of $ ( G, \torus ) $ has disjoint cycle decompositions $ \whiteperm = ( 1, 2, 3 ) \cdot ( 4, 5 ) $ and $ \blackperm = ( 1 ) \cdot ( 2 ) \cdot ( 3, 4, 5 ) $. The pair $ ( \whiteperm, \blackperm ) $ is certainly transitive, and if one were to construct the corresponding dessin d'enfant one would get the same graph embedded into $ \sphere $. Because the monodromy pair is the same, the ``configuration'' of the graph would be the same in $ \sphere $ as it is in $ \torus $.
\end{exm}

\begin{dfn}[Models]
For a finite set $ E $ and arbitrary pair $ ( \whiteperm, \blackperm ) $ in $ S_E $, a dessin d'famille $ ( G, X ) $ is a \emph{model} for $ ( \whiteperm, \blackperm ) $ iff $ E $ is the edge set of $ G $ and $ ( \whiteperm, \blackperm ) $ is the monodromy pair of $ ( G, X ) $.
\end{dfn}

It is an easy formality that models exist for all pairs:

\begin{prp}[Models Exist] \label{PRPconstructionofmodels}
If $ ( \whiteperm, \blackperm ) $ is an arbitrary pair in $ S_E $ then a nondegenerate model $ ( G, X ) $ exists.
\end{prp}

In the rest of the paper, this Proposition will be used without explicit reference.

\begin{proof}
The restrictions of $ \whiteperm, \blackperm $ to each $ \langle \whiteperm, \blackperm \rangle $-orbit $ O_i \subset E $ form a transitive pair in $ S_{O_i} $ and so there is a unique dessin d'enfant $ ( D_i, X_i ) $ with edges $ O_i $ whose monodromy pair are these restrictions. Form the connected sum\footnote{Since these are \emph{oriented} manifolds, the gluing map between the punctures should be orientation-\emph{reversing}. This guarantees that the resulting surface is oriented and the orientation is consistent with the orientations of all $ X_i $, which is critical in order to get a \emph{model}.} $ ( G, X ) $ of all $ ( D_i, X_i ) $ using discs whose closures are contained entirely within the interiors of the faces of $ ( D_i, X_i ) $. It is immediate from the construction that $ ( G, X ) $ is a model for $ ( \whiteperm, \blackperm ) $. Since dessins d'enfants are always nondegenerate, it is clear that this model is nondegenerate.
\end{proof}

Recall from \S\ref{SSdessins} that if $ ( D, X ) $ is a dessin d'enfant then the number of components of $ X \setminus D $ is equal to $ \numcycles ( \whiteperm \blackperm ) $. This can be generalized very naturally to dessins d'familles, although to do so rigorously requires the machinery of \cite{HR}. However, it will be helpful to state a vague version first:

\begin{thmnonum}[preliminary version of Theorem \ref{THMcyclesandboundary}]
Let $ ( G, X ) $ be a nondegenerate dessin d'famille with edges $ \edges $ and monodromy pair $ ( \whiteperm, \blackperm ) $.

The complement $ X \setminus G $ is a disjoint union of connected components which, as open subsets of $ X $, are surfaces without boundary. Intuitively, each of these connected components can be ``completed'' to a surface with boundary, collectively forming a (likely disconnected) surface with boundary $ \overline{X \setminus G} $.

\Assertion The connected components of the manifold boundary $ \partial \overline{X \setminus G} $ are in bijection with the set $ \edges / \whiteperm \blackperm $ of $ \whiteperm \blackperm $-orbits and surject onto the connected components of the graph $ G $:
\begin{equation*}
\edges / \whiteperm \blackperm \directedisom \pi_0 ( \partial \overline{X \setminus G} ) \twoheadrightarrow \pi_0 ( G )
\end{equation*}
\end{thmnonum}

This is a generalization because if $ ( D, X ) $ is a dessin d'enfant then $ \vert \pi_0 ( D ) \vert = 1 $ and the components of $ X \setminus D $ are homeomorphic to discs, so $ \pi_0 ( \partial \overline{X \setminus D} ) $ is in canonical bijection with $ \pi_0 ( X \setminus D ) $.

\begin{exm} \label{EXMdessindnichee}
The following depicts a disconnected dessin d'famille $ ( G, \torus ) $:
\begin{center}
\includegraphics[scale=0.4]{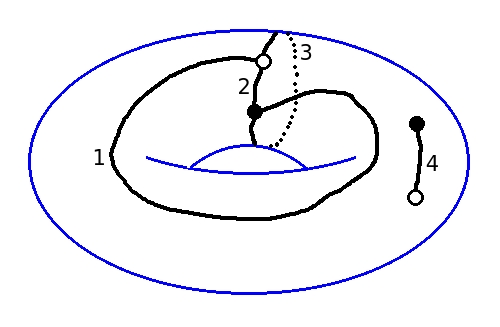}
\end{center}
Its monodromy pair is $ \whiteperm = ( 1, 2, 3 ) \cdot ( 4 ) $ and $ \blackperm = \whiteperm $, so $ \whiteperm \blackperm = ( 1, 3, 2 ) \cdot ( 4 ) $. The complement $ \torus \setminus G $ is homeomorphic to a cylinder (or annulus), and the two cycles of $ \whiteperm \blackperm $ correspond to the two boundary circles of the cylinder.
\end{exm}

Despite the intuitive nature of the claim, it is surprisingly difficult to prove. To justify the inclusion of such a proof here, note the following: \emph{It seems that even the well-known version for Dessins d'Enfants has never been proved rigorously in print.}

I now review \cite{HR}, state a precise version of the Theorem, and prove it. Let 
\begin{equation*}
\theta : \overline{X \setminus G} \longrightarrow X
\end{equation*}
be the ``completion'' of $ X \setminus G $ in the sense of Scissors Theorem 2.3 in \cite{HR}. The space $ \overline{X \setminus G} $ is a compact surface with boundary, likely disconnected, and $ \theta $ is a continuous surjection with various properties. Among those properties are the fact that $ \theta $ sends $ \partial \overline{X \setminus G} $ onto $ G $ and restricts to a homeomorphism between the interior of $ \overline{X \setminus G} $ and $ X \setminus G $. For more details, consult Theorem 2.3 in \cite{HR}.

Since manifolds with boundary are involved, it is necessary to talk about both planes and half-planes, and \cite{HR} uses $ \C $ as plane and denotes by $ \C_{+} $ the closed upper half-plane, so that $ \R = \partial \C_{+} $. The main ingredient needed to construct $ \overline{X \setminus G} $ is a certain set $ \Lambda $ of half-plane maps\footnote{Some of these $ \lambda $ may not be embeddings. This possibility is allowed so that vertices with valence $ 1 $ can be treated. For details, consult Step 1 in the proof of Theorem 2.3 in \cite{HR}.} $ \lambda : \C_{+} \rightarrow X $. There is a natural equivalence relation on the disjoint union of $ X \setminus G $ with $ \C_{+} \times \Lambda $, and $ \overline{X \setminus G} $ is the quotient space. I denote by $ \lambda [ z ] $ the point of $ \overline{X \setminus G} $ represented by $ ( z, \lambda ) \in \C_{+} \times \Lambda $, and by $ \lambda_i [ \R ] $ the set of $ \lambda [ x ] $ for all $ x \in \R $. Roughly speaking, the open half-planes $ \lambda [ \C_{+} \setminus \R ] $ overlap to form a collar of $ G $ in $ X $ and the open intervals $ \lambda [ \R ] $ attached to that collar overlap to form $ \partial \overline{X \setminus G} $.

I also attach a few keywords to notions from the important Neighborhood Theorem 3.1 in \cite{HR}. For a point $ x \in G $, not necessarily a vertex, a \emph{standard neighborhood} of $ x $ is any of the topological embeddings $ h : \C \rightarrow X $ guaranteed by Theorem 3.1 of \cite{HR}. The image $ h ( \C ) \subset X $ is necessarily open, $ h ( 0 ) = x $, and $ x $ is called the \emph{center} of $ h $. By construction, the preimage $ h^{-1} ( G ) $ is the set of all $ r e^{2 \pi i k / n} $ for all Real $ r \geq 0 $ and all $ k \in \Z $, for $ n $ the valence\footnote{If $ x $ is not a vertex then the valence is defined by \cite{HR} to be $ 2 $. This is done in order to recognize that a small disc around such $ x $ is separated by $ G $ into two components, half-discs.} of $ x $. This subset $ h^{-1} ( G ) $ is called the \emph{star} of $ h $ and, for fixed $ k \in \Z $, the set of all $ r e^{2 \pi i k / n} $ for all $ r > 0 $ is called a \emph{ray} of $ h $. By construction, no edge of $ G $ is fully contained in $ h ( \C ) $, and no vertex of $ G $ is contained in $ h ( \C ) $ except possibly $ x $, so for each ray $ r $ of $ h $ there is $ e \in \edges $ such that $ h ( r ) \subset e $. A connected component of the complement $ \C \setminus h^{-1} ( G ) $ is called a \emph{cone} of $ h $. Given a fixed orientation of $ \C $, the cones of $ h $ are cyclically ordered and, for each ray $ r $, there is a cone \emph{preceding} $ r $ and a cone \emph{following} $ r $ (which are the same cone iff $ n = 1 $). I will frequently use the following fact, from Step 6 of the proof of Theorem 2.3 in \cite{HR}: If $ h $ is a standard neighborhood and $ C $ is a cone of $ h $ then there exists $ \gamma : \C_{+} \rightarrow \overline{C} $, called \emph{suitable} below, such that $ h \circ \gamma \in \Lambda $.

There is a natural surjection, essentially just a restriction of $ \theta $, that will be important:

\begin{dfn}
The function
\begin{equation*}
\partial\mathbf{\theta} : \pi_0 ( \partial \overline{X \setminus G} ) \longrightarrow \pi_0 ( G )
\end{equation*}
sends $ J $ to the connected component of $ G $ containing the (necessarily connected) subset $ \theta ( J ) \subset G $. It is immediate from Theorem 3.2(c) in \cite{HR} that $ \partial\mathbf{\theta} $ is surjective.
\end{dfn}

Using the orientation of $ X $, another important natural function can be defined. A bit more work is necessary to define it, and this will precede the formal definition.

Let $ e \in \edges $ be arbitrary. Let $ h : \C \rightarrow X $ be a standard neighborhood of the vertex $ \black_e $. Let $ r \subset \C $ be the ray of $ h $ for which $ h ( r ) \subset e $, and let $ C \subset \C $ be the cone following $ r $. As in Step 6 of the proof of Theorem 2.3 in \cite{HR}, setting $ \lambda \defeq h \circ \gamma $ for suitable $ \gamma : \C_{+} \rightarrow \overline{C} $ yields a point $ \lambda [ 0 ] \in \partial \overline{X \setminus G} $. According to the equivalence relation defining $ \overline{X \setminus G} $, the connected component $ J \subset \partial \overline{X \setminus G} $ containing $ \lambda [ 0 ] $ is independent of $ h $ and $ \gamma $.

\begin{dfn}
The function 
\begin{equation*}
\mathbf{r}_{\black} : \edges \longrightarrow \pi_0 ( \partial \overline{X \setminus G} )
\end{equation*}
sends $ e $ to the connected component $ J $ described in the previous paragraph.
\end{dfn}

Note that $ \partial\mathbf{\theta} \circ \mathbf{r}_{\black} : \edges \rightarrow \pi_0 ( G ) $ is the obvious function.

\begin{rmk}
Roughly speaking, $ \blackperm $ rotates $ e $ to $ \blackperm ( e ) $ according to the orientation of $ X $ and, in doing so, ``sweeps out'' a small cone in $ X \setminus G $ bounded by $ e $ and $ \blackperm ( e ) $. This cone lifts to a small half-plane in $ \overline{X \setminus G} $, whose boundary is inside $ \mathbf{r}_{\black} ( e ) $.
\end{rmk}

It will also be convenient to have the counterpart to $ \mathbf{r}_{\black} $, the function
\begin{equation*}
\mathbf{r}_{\white} : \edges \longrightarrow \pi_0 ( \partial \overline{X \setminus G} )
\end{equation*}
defined by repeating the construction of $ \mathbf{r}_{\black} $ with $ \white $ and $ \black $ exchanged.

\begin{thm} \label{THMcyclesandboundary}
Denote by $ \edges / \whiteperm \blackperm $ the set of $ \whiteperm \blackperm $-orbits.

\Assertions

\begin{enumerate}
\setlength{\itemsep}{5pt}

\item \label{THMcyclesandboundary1} $ \mathbf{r}_{\black} $ is surjective.

\item \label{THMcyclesandboundary2} $ \mathbf{r}_{\black} $ is constant on each $ \whiteperm \blackperm $-orbit.

\item \label{THMcyclesandboundary3} $ \mathbf{r}_{\black} $ restricts to a bijection $ \mathbf{r}_{\black} : \edges / \whiteperm \blackperm \directedisom \pi_0 ( \partial \overline{X \setminus G} ) $.

\item \label{THMcyclesandboundary4} For $ J = \mathbf{r}_{\black} ( e ) $, the sequence of edges occurring in the combinatorial boundary associated by \cite{HR} to $ J $ is, modulo rotation and reversal, $ e, \blackperm ( e ), \whiteperm \blackperm ( e ), \ldots $.
\end{enumerate}

In particular, the set of edges occurring in $ \theta ( J ) \subset G $ is $ O \cup \blackperm ( O ) $ for $ O $ the $ \whiteperm \blackperm $-orbit such that $ \mathbf{r}_{\black} ( O ) = J $.

All assertions are still true after exchanging $ \white $ and $ \black $.
\end{thm}

\begin{proof}
\proofbox{assertion (\ref{THMcyclesandboundary1})} Let $ J \in \pi_0 ( \partial \overline{X \setminus G} ) $ be arbitrary and let $ x \in J $ be a point. By Step 7 of the proof of Theorem 2.3 in \cite{HR}, there is a standard neighborhood $ h : \C \rightarrow X $ and a cone $ C $ of $ h $ such that precomposing $ h $ with suitable $ \C_{+} \rightarrow \overline{C} $ yields $ \lambda \in \Lambda $ such that $ \lambda [ 0 ] = x $. Let $ r \subset \C $ be the ray of $ h $ for which $ C $ follows $ r $, and let $ e \in \edges $ be such that $ h ( r ) \subset e $. It follows from the definition of $ \mathbf{r}_{\black} $ that $ \mathbf{r}_{\black} ( e ) = J $, so $ \mathbf{r}_{\black} $ is surjective.

\proofbox{assertion (\ref{THMcyclesandboundary2})} Let $ e \in \edges $ be arbitrary, and set $ \epsilon \defeq \blackperm ( e ) $. For each point of $ \overline{\epsilon} $, choose a standard neighborhood of that point, and select from this open cover a finite subcover $ h_1, h_2, \ldots, h_n $. By choice, $ h_i ( 0 ) \neq h_j ( 0 ) $ for all $ i \neq j $. Since standard neighborhoods do not contain any vertices of $ G $ except possibly their centers, there must be $ i, j $ such that $ h_i ( 0 ) = \black_{\epsilon} $ and $ h_j ( 0 ) = \white_{\epsilon} $. After renumbering if necessary, I can assume that $ h_1 ( 0 ) = \black_{\epsilon} $ and $ h_n ( 0 ) = \white_{\epsilon} $. Fix a parameterization of $ \overline{\epsilon} $ by $ [ 0, 1 ] $ such that $ 0 \mapsto \black_{\epsilon} $ and $ 1 \mapsto \white_{\epsilon} $, and transport to $ \overline{\epsilon} $ the usual total order on $ [ 0, 1 ] $. After renumbering if necessary, I can assume that $ h_1 ( 0 ) < h_2 ( 0 ) < \cdots < h_n ( 0 ) $.

I now choose a special cone $ c_i $ for each $ h_i $. Let $ r_1 \subset \C $ be the ray of $ h_1 $ such that $ h_1 ( r_1 ) \subset \epsilon $ and let $ c_1 \subset \C $ be the cone of $ h_1 $ preceding $ r_1 $. Similarly, let $ r_n \subset \C $ be the ray of $ h_n $ such that $ h_n ( r_n ) \subset \epsilon $ and let $ c_n \subset \C $ be the cone of $ h_n $ following $ r_n $. Now, suppose $ 0 < i < n $. Since $ h_i $ is standard and $ h_i ( 0 ) \in \epsilon $, the star of $ h_i $ is simply $ \R $ and $ h_i ( \R ) \subset \epsilon $. The total order of $ \epsilon $ therefore orders the two rays of $ h_i $, the positive and negative axes of $ \R $, and $ c_i $ is defined to be the cone following whichever ray of $ h_i $ is ``first'' according to this order. In other words, $ c_i $ is the cone following whichever ray $ r $ satisfies $ h_{i-1} ( 0 ) < h_i ( x ) < h_i ( 0 ) $ for all $ x \in r $. For each $ i $, precomposing $ h_i $ with a suitable $ \C_{+} \rightarrow \overline{c}_i $ yields $ \lambda_i \in \Lambda $.

I claim that the union of $ \lambda_i [ \R ] \subset \partial \overline{X \setminus G} $ for all $ i $ is a \emph{connected} subset of $ \partial \overline{X \setminus G} $. Since $ \lambda_1 [ 0 ] \in \mathbf{r}_{\black} ( e ) $ by definition of $ \mathbf{r}_{\black} $ and choice of cone $ c_1 $, and since $ \lambda_n [ 0 ] \in \mathbf{r}_{\white} ( \epsilon ) $ by definition of $ \mathbf{r}_{\white} $ and choice of cone $ c_n $, this shows that $ \mathbf{r}_{\white} ( \blackperm ( e ) ) = \mathbf{r}_{\black} ( e ) $. Exchanging colors and applying again shows that $ \mathbf{r}_{\black} ( \whiteperm \blackperm ( e ) ) = \mathbf{r}_{\white} ( \blackperm ( e ) ) $. Therefore, $ \mathbf{r}_{\black} ( \whiteperm \blackperm ( e ) ) = \mathbf{r}_{\black} ( e ) $, i.e. $ \mathbf{r}_{\black} $ is constant on each $ \whiteperm \blackperm $-orbit.

Suppose that $ h_i ( \C ) \cap h_j ( \C ) \cap \epsilon \neq \emptyset $ for some $ i, j $. I can assume that $ i < j $, so $ h_i ( 0 ) < h_j ( 0 ) $ according to the total order of $ \overline{\epsilon} $. Necessarily, there is $ x \in h_i ( \C ) \cap h_j ( \C ) \cap \epsilon $ such that $ h_i ( 0 ) < x < h_j ( 0 ) $. Let $ R_i $ and $ R_j $ be the rays of $ h_i $ and $ h_j $ such that $ x \in h_i ( R_i ) \cap h_j ( R_j ) $. Let $ C_i $ be the cone of $ h_i $ preceding $ R_i $, and $ C_j $ the cone of $ h_j $ following $ R_j $. If it is shown that $ C_i = c_i $ and $ C_j = c_j $ then, by Proposition \ref{PRPintersectingcones} below, the paths $ \lambda_i [ \R ] $ and $ \lambda_j [ \R ] $ will share a point. Certainly $ 1 \leq i < n $. If $ 1 < i < n $ then it is immediate from the choices of rays $ r_i, R_i $ that $ R_i \neq r_i $, but $ h_i $ has only two rays and so the cone $ C_i $ of $ h_i $ preceding $ R_i $ is the same as the cone $ c_i $ following $ r_i $. If $ i = 1 $ then $ R_1 = r_1 $ by choice and so $ C_1 = c_1 $, since each cone is the one preceding its ray. In all cases, $ C_i = c_i $. Similarly, $ 1 < j \leq n $ and if $ 1 < j < n $ then $ R_j = r_j $ and so $ C_j = c_j $, since each cone is the one following its ray. If $ j = n $ then $ R_n = r_n $ by choice and $ C_n = c_n $ since each cone is the one following its ray. In all cases, $ C_j = c_j $.

Consider the interval $ \lambda_1 [ \R ] $. There must be $ i \neq 1 $ such that $ h_i ( \C ) \cap h_1 ( \C ) \cap \epsilon \neq \emptyset $, since otherwise $ \overline{\epsilon} $ would be the disjoint union of two open sets (i.e. disconnected). By the previous paragraph, $ \lambda_1 ( \R ) \cap \lambda_i ( \R ) \neq \emptyset $ and therefore $ \lambda_1 [ \R ] \cup \lambda_i [ \R ] $ is a connected subset of $ \partial \overline{X \setminus G} $. For the same reason, there must be a $ j \neq 1, i $ such that either $ h_j ( \C ) \cap h_1 ( \C ) \cap \epsilon \neq \emptyset $ or $ h_j ( \C ) \cap h_i ( \C ) \cap \epsilon \neq \emptyset $, and therefore $ \lambda_1 [ \R ] \cup \lambda_i [ \R ] \cup \lambda_j [ \R ] $ is connected. Continuing in this way, $ \lambda_1 [ \R ] \cup \lambda_2 [ \R ] \cup \cdots \cup \lambda_n [ \R ] $ is connected.

\proofbox{assertion (\ref{THMcyclesandboundary4})} Recall from \S5 of \cite{HR} that the \emph{combinatorial boundary} associated to any $ J \in \pi_0 ( \partial \overline{X \setminus G} ) $ is a certain closed walk $ P_J $ in $ G $ such that $ \theta ( J ) = P_J $. Fix $ e \in \edges $ and set $ J \defeq \mathbf{r}_{\black} ( e ) $. As in \cite{HR}, the circle $ J $ is partitioned by finitely many points $ x_i $ into finitely many arcs $ s_i $ so that each $ \theta ( x_i ) $ is a vertex of $ G $ and each $ \theta ( s_i ) $ is an edge of $ G $. Since $ G $ is bicolored, there are at least two $ x_i $ and at least two $ s_i $. It is clear from the definition of $ \lambda_i $ above that the image by $ \theta $ of the open interval $ \lambda_1 [ \R ] \cup \lambda_2 [ \R ] \cup \cdots \cup \lambda_n [ \R ] \subset J $ intersects the edges $ e, \blackperm ( e ), \whiteperm \blackperm ( e ) $. By construction of the combinatorial boundary $ P_J $, the sequence (modulo rotation and reversal) $ \edges_J $ of edges occurring in $ P_J $ contains $ e, \blackperm ( e ), \whiteperm \blackperm ( e ) $ as a subsequence. Repeating the argument proves the last assertion in the Theorem. It remains to prove that $ \mathbf{r}_{\black} $ is injective on $ \whiteperm \blackperm $-orbits.

\proofbox{assertion (\ref{THMcyclesandboundary3})} It follows from Theorem 2.3(d) in \cite{HR} and the construction of the combinatorial boundary in \cite{HR} that each $ e \in \edges $ appears twice among the combinatorial boundaries: either once in both $ \edges_J $ and $ \edges_{J^{\prime}} $ for some distinct $ J, J^{\prime} $ or twice in $ \edges_J $ for unique $ J $. In other words, $ \frac{1}{2} \sum_{J} \vert \edges_J \vert = \vert \edges \vert $. It follows from (\ref{THMcyclesandboundary4}) that $ \vert O \vert = \frac{1}{2} \vert \edges_J \vert $, for $ O $ the $ \whiteperm \blackperm $-orbit such that $ \mathbf{r}_{\black} ( O ) = J $. Since $ \mathbf{r}_{\black} $ is already known to be surjective by (\ref{THMcyclesandboundary1}), failure of injectivity would imply $ \sum_{O} \vert O \vert > \sum_{J} \frac{1}{2} \vert \edges_J \vert $, a contradiction due to $ \vert \edges \vert = \sum_{O} \vert O \vert $.
\end{proof}

\begin{cor} \label{CORcountingboundary}
$ \numcycles ( \whiteperm \blackperm ) \geq \vert \pi_0 ( G ) \vert $.
\end{cor}

\begin{proof}
By Theorem \ref{THMcyclesandboundary}, $ \partial\mathbf{\theta} \circ \mathbf{r}_{\black} $ is a surjection $ \edges / \whiteperm \blackperm \twoheadrightarrow \pi_0 ( G ) $. 
\end{proof}

\begin{prp} \label{PRPintersectingcones}
Let $ h_1, h_2 : \C \hookrightarrow X $ be standard neighborhoods. Let $ r_1 $ and $ r_2 $ be rays of $ h_1 $ and $ h_2 $. Suppose $ h_1 ( r_1 ) \cap h_2 ( r_2 ) \neq \emptyset $, let $ x $ be in the intersection, and let $ x_i \in r_i $ be such that $ h_i ( x_i ) = x $. \Assertion If $ x $ is ``between'' $ h_1 ( 0 ) $ and $ h_2 ( 0 ) $, i.e. there is no element of $ r_1 $ between $ 0 $ and $ x_1 $ sent by $ h_1 $ to $ h_2 ( 0 ) $ and vice-versa, then the cone $ C_1 $ of $ h_1 $ preceding $ r_1 $ intersects the cone $ C_2 $ of $ h_2 $ following $ r_2 $.

More precisely, there is a closed half disc $ d : \C_{+} \hookrightarrow X $ such that
\begin{itemize}
\item[-] $ x = d ( 0 ) $
\item[-] $ d ( \R ) \subset h_1 ( r_1 ) \cap h_2 ( r_2 ) $
\item[-] $ d ( \C_{+} \setminus \R ) \subset h_1 ( C_1 ) \cap h_2 ( C_2 ) $
\end{itemize}

In particular, $ \lambda_1 [ x_1 ] = \lambda_2 [ x_2 ] $ in $ \overline{X \setminus G} $, where $ \lambda_i $ is the precomposition of $ h_i $ with a suitable $ \C_{+} \rightarrow \overline{C}_i $.
\end{prp}

\begin{proof}
The last claim follows from the first by definition of the equivalence relation used to construct $ \overline{X \setminus G} $, Step 2 of the proof of Theorem 2.3 in \cite{HR}.

Let $ D_1 \subset \C $ be an open disc centered at $ x_1 $, so small that $ D_1 $ does not intersect any other ray of $ h_1 $ and $ h_1 ( D_1 ) \subset h_2 ( \C ) $. Since $ h_1 ( \C ) \subset X $ is open, there is an open disc $ D_2 \subset \C $ centered at $ x_2 $ such that $ h_2 ( D_2 ) \subset h_1 ( D_1 ) $. Necessarily, $ D_i \setminus r_i $ consists of two connected components, one of which is contained in the cone $ C_i $. Call that component $ H_i $, and note that there is a homeomorphism $ \C_{+} \simeq \overline{H}_i $ which sends $ 0 $ to $ x_i $ and restricts to $ \R \simeq \partial H_i $. Thus, the goal is to show that $ h_2 ( H_2 ) \subset h_1 ( H_1 ) $, since then precomposing $ h_2 $ with the homeomorphism yields a closed half disc $ \C_{+} \hookrightarrow X $ with the desired properties.

Let $ I_2 \subset D_2 $ be a closed segment (i.e. homeomorphic to $ [ 0, 1 ] $) whose interior contains $ x_2 $ and such that $ I_2 \cap r_2 = \{ x_2 \} $. Set $ I \defeq h_2 ( I_2 ) $. Let $ I_1 \subset D_1 $ be the segment such that $ h_1 ( I_1 ) = I $. Note that $ x_1 $ is contained in the interior of $ I_1 $ and that $ I_1 \cap r_1 = \{ x_1 \} $.

For any triangle $ T_1 \subset \C $ with edge $ I_1 $ and opposite vertex $ 0 $, the orientation of its boundary induced by $ X $ via $ h_1 $ orders the two connected components of $ I_1 \setminus \{ x_1 \} $. By definition of ``preceding'', the component contained in $ C_1 $ is that which is considered ``first'' by this ordering. Similarly, any triangle $ T_2 \subset \C $ with edge $ I_2 $ and opposite vertex $ 0 $ orders the two connected components of $ I_2 \setminus \{ x_2 \} $ and, by definition of ``following'', the component contained in $ C_2 $  is that which is considered ``second''. Thus, the goal is to show that there are $ T_1, T_2 $ such that $ h_1 ( T_1 ) $ and $ h_2 ( T_2 ) $ induce \emph{opposite} orientations on $ I $.

Let $ Q \subset D_2 $ be a convex quadrilateral such that $ I_2 $ is one diagonal of $ Q $ and the other two vertices of $ Q $ are on the ray $ r_2 $. Let $ v_2 \in Q \cap r_2 $ be the vertex between $ 0 $ and $ x_2 $ on the ray $ r_2 $, and let $ t_2 \subset Q $ be the triangle, half of $ Q $, with edge $ I_2 $ and opposite vertex $ v_2 $. It is clear that there is a triangle $ T_2 \subset \C $ with edge $ I_2 $ and opposite vertex $ 0 $ such that $ t_2 \subset T_2 $, and therefore that $ T_2 $ and $ t_2 $ induce the same ordering of $ \pi_0 ( I_2 \setminus \{ x_2 \} ) $.

Now, let $ t_1 \subset \C $ be such that $ h_2^{-1} ( h_1 ( t_1 ) ) $ is the triangle in $ Q $ complementary to $ t_2 $. So, $ t_1 $ is a triangle with edge $ I_1 $ and opposite vertex $ v_1 \in r_1 $. I claim that $ v_1 $ is between $ x_1 $ and $ 0 $ on the ray $ r_1 $, which finishes the proof: it is clear that there is a triangle $ T_1 $ with edge $ I_1 $ and opposite vertex $ 0 $ such that $ t_1 \subset T_1 $, therefore $ T_1 $ and $ t_1 $ induce the same ordering of $ \pi_0 ( I_1 \setminus \{ x_1 \} ) $, thus the goal is to prove that $ h ( t_1 ) $ and $ h ( t_2 ) $ induce opposite orientations of $ I $, which is clear since $ h ( t_1 ), h ( t_2 ) \subset X $ are simplices intersecting only along $ I $.

Exactly one of $ v_1 $ and $ h_1^{-1} ( h_2 ( v_2 ) ) $ is between $ 0 $ and $ x_1 $. Suppose for contradiction that $ h_1^{-1} ( h_2 ( v_2 ) ) $ is, instead of $ v_1 $. Let $ E_1 \simeq [ 0, 1 ] $ be the closed segment on $ \overline{r}_1 $ from $ 0 $ to $ x_1 $. Let $ E_2 \simeq [ 0, 1 ] $ be the closed segment on $ \overline{r}_2 $ from $ 0 $ to $ x_2 $. By definition of ``standard'', there is an edge $ \epsilon $ of $ G $ such that $ h_1 ( r_1 ), h_2 ( r_2 ) \subset \overline{\epsilon} $. The restrictions $ h_i : E_i \rightarrow \overline{\epsilon} $ are continuous injections $ [ 0, 1 ] \hookrightarrow \R $, therefore monotone by a corollary of the \emph{Intermediate Value Theorem}. Since $ h_1 ( x_1 ) = h_2 ( x_2 ) $, the images intersect. By the contradiction hypothesis, the images share at least two points. An easy argument then shows that either $ h_1 ( E_1 ) \subset h_2 ( E_2 ) $ or $ h_1 ( E_1 ) \supset h_2 ( E_2 ) $, so either $ h_1 ( 0 ) \in h_2 ( E_2 ) $ or $ h_2 ( 0 ) \in h_1 ( E_1 ) $, violating the hypothesis.
\end{proof}

A nice fact about trees, hinting at more general statements, can now be proved without too much work:

\begin{prp}[Tree Case] \label{PRPtreecase}
Let $ ( T, \sphere ) $ be a tree dessin d'enfant with edges $ \edges $ and monodromy pair $ ( \whiteperm, \blackperm ) $. \Assertion For arbitrary $ s \in S_{\edges} $, $ ( \whiteperm^s, \blackperm ) $ is a transitive pair if and only if $ \numcycles ( \whiteperm^s \blackperm ) = 1 $.
\end{prp}

\begin{proof}
Assume that $ ( \whiteperm^s, \blackperm ) $ is transitive, so it corresponds to a dessin d'enfant. Let $ g^s $ be the genus of this dessin d'enfant. In particular, $ g^s \geq 0 $. Since the new dessin d'enfant has the same quantity of vertices and the same quantity of edges as $ T $, and since the genus of $ \sphere $ is $ 0 $, the formula (cf. \S\ref{SSdessins}) for Euler Characteristic implies that $ \numcycles ( \whiteperm^s \blackperm ) \leq \numcycles ( \whiteperm \blackperm ) $. Since $ T $ is a tree, $ \numcycles ( \whiteperm \blackperm ) = 1 $, which forces $ \numcycles ( \whiteperm^s \blackperm ) = 1 $. For the converse, I prove the contrapositive. Let $ ( G, X ) $ be a nondegenerate model for $ ( \whiteperm^s, \blackperm ) $. By hypothesis and Lemma \ref{LEMtransiffconn}, $ G $ is disconnected. By Corollary \ref{CORcountingboundary}, $ \numcycles ( \whiteperm^s \blackperm ) \geq 2 $.
\end{proof}


It will be very useful to have generalizations, for dessins d'familles, of Euler characteristic and genus. The following is such, and simply extends without modification the well-known formula from transitive permutation pairs to all permutation pairs:

\begin{dfn}[Synthetic Genus] \label{DFNsyntheticgenus}
For a finite set $ E $ and arbitrary pair $ ( \whiteperm, \blackperm ) $ in $ S_E $, the \emph{synthetic Euler characteristic} of $ ( \whiteperm, \blackperm ) $ is 
\begin{equation*}
\chi_{( \whiteperm, \blackperm )} \defeq \numcycles ( \whiteperm ) + \numcycles ( \blackperm ) - \vert E \vert + \numcycles ( \whiteperm \blackperm )
\end{equation*}
Accordingly, its \emph{synthetic genus} is\footnote{Of course, this is inspired by the well-known formula $ \chi = 2 - 2 g $.} 
\begin{equation*}
g_{( \whiteperm, \blackperm )} \defeq 1 - \frac{ \chi_{( \whiteperm, \blackperm )} }{ 2 }
\end{equation*}
The synthetic Euler characteristic and synthetic genus of a dessin d'famille are defined to be those of its monodromy pair.
\end{dfn}

Clearly, if $ ( G, X ) $ is a dessin d'enfant then its synthetic genus is the genus of $ X $. More generally, if $ ( G, X ) $ is a connected dessin d'famille then its synthetic genus is the genus of the surface, not necessarily $ X $, into which $ G $ embeds as a dessin d'enfant (via its monodromy pair).

\begin{exm} \label{EXMsyntheticgenus}
Let $ ( \whiteperm, \blackperm ) $ be from Example \ref{EXMdessindnichee}. By the information there, the synthetic Euler characteristic is $ ( 2 + 2 ) - 4 + 2 = 2 $. Thus, this ``genuinely toral'' dessin d'famille has a \emph{spherical} synthetic Euler characteristic. The important conclusion to draw here is that connectivity and synthetic genus are complementary.
\end{exm}

The synthetic Euler characteristic of a dessin d'famille $ ( G, X, \jmath ) $ is the sum of those for $ ( G_i, X, \jmath \vert_{G_i} ) $, where $ G_i $ are the connected components of $ G $. The statement can be translated into one about synthetic genus: for example, if $ ( G, X ) $ has synthetic genus $ g $ and two connected components $ ( G_i, X ) $ whose synthetic genuses are $ g_1 $ and $ g_2 $ then $ g_1 + g_2 = g + 1 $. In particular, synthetic Euler characteristic is always even and synthetic genus is always integral.

\section{Sequences and Arcs} \label{Sarcs}

\begin{center}
\emph{Throughout this section, $ E $ is a finite set and $ S_E $ is its symmetric group.}
\end{center}

The following terminology will be convenient in the remainder of the paper:

\begin{dfn}[Sequences]
For $ \perm \in S_E $, a \emph{$ \perm $-sequence} is a finite nonempty sequence $ x_0, x_1, \ldots, x_n $ in $ E $ satisfying $ \perm ( x_i ) = x_{i+1} $ for all $ 0 \leq i < n $.

For $ x, y \in E $, a \emph{$ \perm $-sequence from $ x $ to $ y $} is simply a $ \perm $-sequence $ x_0, x_1, \ldots, x_n $ such that $ x_0 = x $ and $ x_n = y $. \emph{It is allowed that $ x = y $ even if $ n > 0 $.}
\end{dfn}

Obviously, if there is a $ \perm $-sequence from $ x $ to $ y $ then $ x, y $ represent the same $ \perm $-orbit.

\begin{exm} \label{EXMsequences}
Define $ \perm \in S_4 $ by disjoint cycle decomposition: $ \perm \defeq ( 1, 2, 3 ) \cdot ( 4 ) $. The sequences $ 1, 2, 3, 1, 2 $ and $ 4, 4, 4 $ are both $ \perm $-sequences. The minimal $ \perm $-sequence from $ 1 $ to $ 3 $ is $ 1, 2, 3 $, while the minimal $ \perm $-sequence from $ 1 $ to $ 1 $ is the singleton sequence $ 1 $.
\end{exm}

The following notion will be used heavily in all remaining sections:

\begin{dfn}[Arcs] \label{DFNarcs}
Let $ \perm \in S_E $ be arbitrary and $ x, y \in E $ represent the same $ \perm $-orbit. \emph{It is allowed that $ x = y $.} Let $ x_0, \ldots, x_n $ be the unique $ \perm $-sequence from $ x $ to $ y $ of minimal length, which is a singleton iff $ x = y $. For $ x \neq y $, the \emph{$ \perm $-arc from $ x $ to $ y $} is defined to be the subsequence $ x_1, \ldots, x_n $. For $ x = y $, the \emph{$ \perm $-arc from $ x $ to $ y $} is defined to be the empty sequence.
\end{dfn}

A key feature of minimal sequences, and therefore also arcs, is that $ x_i \neq x, y $ for all $ 0 < i < n $, although sometimes there are no such $ i $.

\begin{exm} \label{EXMarcs}
Let $ \perm $ be as in Example \ref{EXMsequences}. The $ \perm $-arc from $ 1 $ to $ 3 $ is $ 2, 3 $ and the $ \perm $-arc from $ 4 $ to $ 4 $ is the empty sequence.
\end{exm}

The notions of ``sequence'' and ``arc'' will be used exclusively for the case that $ \perm = \whiteperm \blackperm $ for some pair $ ( \whiteperm, \blackperm ) $.

It will be useful in \S\ref{Sclassificationbytransitivity} to note that if $ ( G, X ) $ is a dessin d'famille with edges $ \edges $ and monodromy pair $ ( \whiteperm, \blackperm ) $ then any $ \whiteperm \blackperm $-sequence defines a \emph{walk} in the graph $ G $, as follows.

Suppose $ x, y \in \edges $ are in the same $ \whiteperm \blackperm $-orbit and let $ x_0, x_1, \ldots, x_n \in \edges $ be a $ \whiteperm \blackperm $-sequence from $ x $ to $ y $. Consider the extended sequence
\begin{equation} \label{EQNwalk}
x_0, \blackperm ( x_0 ), x_1, \blackperm ( x_1 ), \ldots, \blackperm ( x_{n-1} ), x_n \in \edges
\end{equation}

Any pair of consecutive edges in sequence (\ref{EQNwalk}) shares a well-defined vertex: $ x_0 $ and $ \blackperm ( x_0 ) $ are both incident to the same $ \black $-vertex, $ \blackperm ( x_0 ) $ and $ x_1 = \whiteperm ( \blackperm ( x_0 ) ) $ are both incident to the same $ \white $-vertex, etc. Thus, sequence (\ref{EQNwalk}) defines a walk from the $ \white $-vertex of $ x $ to the $ \black $-vertex of $ y $.

By omitting from (\ref{EQNwalk}) the first edge or the last edge or both, one similarly obtains walks from either vertex of $ x $ to either vertex of $ y $:
\begin{itemize}
\setlength{\itemsep}{5pt}

\item[-] The subsequence $ \blackperm ( x_0 ), x_1, \blackperm ( x_1 ), \ldots, \blackperm ( x_{n-1} ), x_n $ defines a walk from the $ \black $-vertex of $ x $ to the $ \black $-vertex of $ y $.

\item[-] The subsequence $ x_0, \blackperm ( x_0 ), x_1, \blackperm ( x_1 ), \ldots, \blackperm ( x_{n-1} ) $ defines a walk from the $ \white $-vertex of $ x $ to the $ \white $-vertex of $ y $.

\item[-] The subsequence $ \blackperm ( x_0 ), x_1, \blackperm ( x_1 ), \ldots, \blackperm ( x_{n-1} ) $ defines a walk from the $ \black $-vertex of $ x $ to the $ \white $-vertex of $ y $.
\end{itemize}

Note also that the edges in (\ref{EQNwalk}) are all contained in a single boundary component of $ X \setminus G $. Of course, this is related to ``boundary walk'', cf. \S3.1.4 of \cite{GT}.

\begin{exm} \label{EXMwalksfromarcs}
Consider the following dessin d'enfant:
\begin{center}
\includegraphics[scale=0.4]{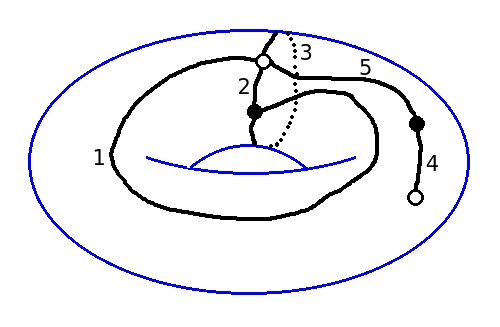}
\end{center}
Its monodromy pair $ ( \whiteperm, \blackperm ) $ in $ S_5 $ has the following disjoint cycle decompositions: $ \whiteperm = ( 1, 2, 5, 3 ) \cdot ( 4 ) $ and $ \blackperm = ( 1, 2, 3 ) \cdot ( 4, 5 ) $. From this, $ \whiteperm \blackperm = ( 1, 5, 4, 3, 2 ) $. The minimal $ \whiteperm \blackperm $-sequence from $ 1 $ to $ 4 $ is therefore $ 1, 5, 4 $. The extended sequence (\ref{EQNwalk}) is $ 1, 2, 5, 4, 4 $, which yields the following walk: $ \white, 1, \black, 2, \white, 5, \black, 4, \white, 4, \black $. It is a good idea to visualize this walk in the picture.
\end{exm}

\section{Incidence and Deletion} \label{Sdeletion}

\begin{center}
\emph{Throughout this section, $ ( D, X ) $ represents an arbitrary dessin d'enfant with edges $ \edges $ and monodromy pair $ ( \whiteperm, \blackperm ) $.}
\end{center}

Recall from \S\ref{SSdessins} that faces of $ ( D, X ) $ correspond naturally to $ \whiteperm \blackperm $-orbits.

\begin{lem}[Face Incidence] \label{LEMfaceincidence}
If $ e \in \edges $ then the faces of $ ( D, X ) $ bordered by $ e $ correspond to those $ \whiteperm \blackperm $-orbits containing $ e $ and $ \whiteperm ( e ) $. In particular, $ e $ borders only one face of $ ( D, X ) $ if and only if $ e $ and $ \whiteperm ( e ) $ represent the same $ \whiteperm \blackperm $-orbit.
\end{lem}

\begin{proof}
If $ F $ is a face of $ ( D, X ) $ and the corresponding cycle of $ \whiteperm \blackperm $ is $ c = ( x_1, \ldots, x_n ) $ then, by \S\ref{SSdessins}, the edges bordering $ F $, with multiplicity, are $ x_1, \blackperm ( x_1 ), \ldots, x_n, \blackperm ( x_n ) $. Thus, $ e $ borders a face $ F $ if and only if its corresponding $ \whiteperm \blackperm $-cycle $ c $ contains an edge $ x $ such that either $ x = e $ or $ \blackperm ( x ) = e $. Since a $ \whiteperm \blackperm $-cycle contains $ x $ if and only if it contains $ \whiteperm \blackperm ( x ) $, and since $ \whiteperm ( \blackperm ( x ) ) = \whiteperm ( e ) $, the first statement is proved. The second statement is immediate from the first.
\end{proof}

\begin{rmk}
Although Lemma \ref{LEMfaceincidence} seems to play a very minor role in the rest of the paper, it was actually the observation that led me towards all the other things.
\end{rmk}

It is well-known that, in the spherical case, deletion of an edge disconnects a graph if and only if the edge borders only one face instead of two. Therefore, I record the following:

\begin{cor}[Edge Deletion] \label{CORedgedeletion}
For $ e \in \edges $, if $ D \setminus e $ is disconnected then $ e $ and $ \whiteperm ( e ) $ represent the same $ \whiteperm \blackperm $-orbit. If $ X = \sphere $ then the converse is true.
\end{cor}

If $ ( D, X ) $ is a dessin d'enfant and $ X \neq \sphere $ then it happens frequently that $ e $ borders only one face and yet $ D \setminus e $ is connected:


\begin{exm}
The following is the trivial dessin d'enfant on $ \torus $:
\begin{center}
\includegraphics[scale=0.4]{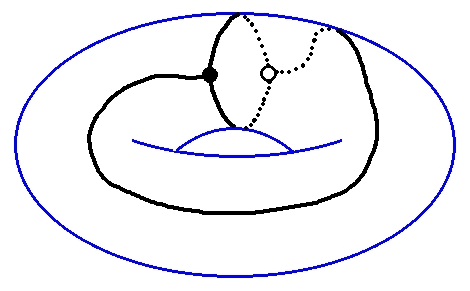}
\end{center}
This dessin d'enfant has only one face, hence every edge borders only one face, but the graph remains connected after deleting any one of them. Of course, the problem is that circuits in $ \sphere $ separate (\emph{Jordan Curve Theorem}), while circuits in $ \torus $ may not.
\end{exm}

It is unclear to me at this time how to characterize these ``disconnecting'' edges when $ X $ is general in a similarly clean way as Corollary \ref{CORedgedeletion}:

\begin{deletionquestionI}
Is there a ``good'' characterization, in terms of the monodromy pair $ ( \whiteperm, \blackperm ) $, of those $ e \in \edges $ such that $ D \setminus e $ is disconnected?
\end{deletionquestionI}

Such a characterization would be valuable for the classification given in \S\ref{Sclassificationbytransitivity} below, especially \S\ref{SSwildexceptional} and \S\ref{SStameexceptional}.

Nonetheless, it is not too difficult to understand deletion at the level of monodromy if the question of connectivity is ignored. Let $ ( G, X ) $ be a dessin d'famille with edges $ \edges $ and monodromy pair $ ( \whiteperm, \blackperm ) $, and let $ e \in \edges $ be arbitrary. Let $ ( \whiteperm^{\prime}, \blackperm^{\prime} ) $ be the monodromy pair of the dessin d'famille $ ( G \setminus e, X ) $. Disjoint cycle decompositions for $ \whiteperm^{\prime} $ and $ \blackperm^{\prime} $ are obtained from those of $ \whiteperm $ and $ \blackperm $ by deleting $ e $ in the obvious way. Because of the important role $ \whiteperm \blackperm $ plays, I give explicit descriptions of $ \whiteperm^{\prime} \blackperm^{\prime} $ also, next. Special treatment, which is annoying but not difficult, is needed if $ \whiteperm ( e ) = e $ or $ \blackperm ( e ) = e $, so assume for convenience that $ \whiteperm ( e ), \blackperm ( e ) \neq e $.

Suppose that $ e $ borders only one face. By Lemma \ref{LEMfaceincidence}, it is equivalent to suppose that $ e $ and $ \whiteperm ( e ) $ represent the same $ \whiteperm \blackperm $-orbit. Let $ x_0, x_1, \ldots, x_m \in \edges $ be the minimal $ \whiteperm \blackperm $-sequence from $ e $ to $ \whiteperm ( e ) $. Since $ \whiteperm ( e ) \neq e $, $ m \geq 1 $. Since also $ \blackperm ( e ) \neq e $, $ m \geq 2 $. It is easy to verify that $ c_0 \defeq ( x_1, \ldots, x_{m-1} ) $ is a $ \whiteperm^{\prime} \blackperm^{\prime} $-cycle. Similarly, let $ y_0, y_1, \ldots, y_n \in \edges $ be the minimal $ \whiteperm \blackperm $-sequence from $ \whiteperm ( e ) $ to $ e $. Since $ \whiteperm ( e ) \neq e $, $ n \geq 1 $. It is easy to verify that $ c_1 \defeq ( y_0, \ldots, y_{n-1} ) $ is a $ \whiteperm^{\prime} \blackperm^{\prime} $-cycle. Any $ \whiteperm \blackperm $-cycle that does not contain $ e $ is also a $ \whiteperm^{\prime} \blackperm^{\prime} $-cycle and $ \whiteperm^{\prime} \blackperm^{\prime} $ is the product of these cycles and $ c_0 $ and $ c_1 $. In particular, if $\chi $ and $ \chi^{\prime} $ are the synthetic Euler characteristics of $ ( \whiteperm, \blackperm ) $ and $ ( \whiteperm^{\prime}, \blackperm^{\prime} ) $ then $ \chi^{\prime} = \chi + 2 $.

Suppose instead that $ e $ borders two faces. By Lemma \ref{LEMfaceincidence}, it is equivalent to suppose that $ e $ and $ \whiteperm ( e ) $ represent different $ \whiteperm \blackperm $-orbits. Let $ x_0, \ldots, x_m \in \edges $ be the $ \whiteperm \blackperm $-sequence from $ e $ to $ e $ such that $ m - 1 $ is the size of the $ \whiteperm \blackperm $-orbit of $ e $ (in particular, $ m \geq 2 $ and $ \{ x_0, \ldots, x_{m-1} \} $ is the $ \whiteperm \blackperm $-orbit of $ e $). Let $ y_0, \ldots, y_n \in \edges $ be the analogous sequence from $ \whiteperm ( e ) $ to $ \whiteperm ( e ) $. It is easy to verify that $ c_1 \defeq ( x_1, \ldots, x_{m-1}, y_0, \ldots, y_{n-1} ) $ is a $ \whiteperm^{\prime} \blackperm^{\prime} $-cycle. Any $ \whiteperm \blackperm $-cycle that contains neither $ e $ nor $ \whiteperm ( e ) $ is also a $ \whiteperm^{\prime} \blackperm^{\prime} $-cycle and $ \whiteperm^{\prime} \blackperm^{\prime} $ is the product of these cycles and $ c_1 $. In particular, $ \chi^{\prime} = \chi $.

\begin{rmk}
Of course, the statements about Euler characteristics are essentially well-known, e.g. Theorem 3.3.5 in \cite{GT}.
\end{rmk}

It is very important to observe that, even in the ``generic'' situation, when $ \whiteperm ( e ), \blackperm ( e ) \neq e $, two very different situations can result in an increased synthetic Euler Characteristic:

\begin{exm} \label{EXMsyneulercharinctwoways}
Let $ ( \whiteperm, \blackperm ) $ be as in Example \ref{EXMwalksfromarcs}. The Euler characteristic is, as expected, $ 0 $. Deletion of edge $ 5 $ results in Example \ref{EXMdessindnichee}, a \emph{disconnected} \emph{toral} dessin d'famille with synthetic Euler characteristic $ 2 $. On the other hand, deletion of any one of edges $ 1, 2, 3 $ results in a connected dessin d'famille with Euler characteristic also $ 2 $.
\end{exm}

\section{The Reroute Operation} \label{Sfundamentaloperation}

\begin{center}
\emph{Throughout this section, $ E $ is a finite set, $ S_E $ is its symmetric group, distinct $ a, b \in E $ are fixed, and $ ( \whiteperm, \blackperm ) $ is an arbitrary pair in $ S_E $. Despite the notation, $ ( \whiteperm, \blackperm ) $ is not assumed to be transitive.}
\end{center}

\subsection{Definition and goal} \label{SSdefofoperation}

\begin{dfn}[Reroute $ \opsymbol $] \label{DFNfundamentaloperation}
Define $ E^{\opsymbol} \defeq \{ a_{\white}, a_{\black} \} \sqcup E \setminus \{ a \} $, where $ a_{\white} $ and $ a_{\black} $ are formal symbols.

Define $ \whiteperm^{\opsymbol} \in S_{E^{\opsymbol}} $ by modifying the disjoint cycle decomposition of $ \whiteperm $ as follows: \emph{Replace $ a $ by the symbol $ a_{\white} $ and insert the symbol $ a_{\black} $ immediately before $ b $ in the cycle of $ \whiteperm $ containing $ b $.}

Define $ \blackperm^{\opsymbol} \in S_{E^{\opsymbol}} $ by modifying the disjoint cycle decomposition of $ \blackperm $ as follows: \emph{Introduce the trivial cycle fixing $ a_{\white} $ and replace $ a $ by the symbol $ a_{\black} $.}

The pair $ ( \whiteperm^{\opsymbol}, \blackperm^{\opsymbol} ) $ is called the \emph{Reroute} of $ ( \whiteperm, \blackperm ) $ relative to $ ( a, b ) $.
\end{dfn}

The relevance of this operation $ \opsymbol $ to the main question is that to perform a conjugation on $ ( \whiteperm, \blackperm ) $ is essentially equivalent to performing a sequence of the operations $ \opsymbol $ for various choices of $ a, b $. An explicit statement of this for transpositions, which is the only case needed in this paper, occurs as Proposition \ref{PRPconjugationviaoperations} (the general case is not so difficult, but is too notationally cumbersome to justify its inclusion).

Although the most elegant definition of $ ( \whiteperm^{\opsymbol}, \blackperm^{\opsymbol} ) $ is that given in Definition \ref{DFNfundamentaloperation} above, it will be convenient to extract some simple facts in the form of a list:

\begin{lem}[Alternate \ref{DFNfundamentaloperation}]
Let the notation be as in Definition \ref{DFNfundamentaloperation} above.

\begin{enumerate}
\setlength{\itemsep}{5pt}

\item \label{op1} If $ \whiteperm ( a ) = a $ then $ \whiteperm^{\opsymbol} ( a_{\white} ) = a_{\white} $ but
\begin{itemize}
\item[] otherwise $ \whiteperm ( x ) = a $ implies $ \whiteperm^{\opsymbol} ( x ) = a_{\white} $, and $ \whiteperm^{\opsymbol} ( a_{\white} ) = \whiteperm ( a ) $.
\end{itemize}

\item \label{op2} If $ \whiteperm ( a ) = b $ then $ \whiteperm^{\opsymbol} ( a_{\white} ) = a_{\black} $ but
\begin{itemize}
\item[] otherwise $ \whiteperm ( x ) = b $ implies $ \whiteperm^{\opsymbol} ( x ) = a_{\black} $.
\end{itemize}

\item \label{op3} $ \whiteperm^{\opsymbol} ( a_{\black} ) = b $.

\item \label{op4} If $ x \neq a $ and $ \whiteperm ( x ) \neq a, b $ then $ \whiteperm^{\opsymbol} ( x ) = \whiteperm ( x ) $.

\item \label{op5} $ \blackperm^{\opsymbol} ( a_{\white} ) = a_{\white} $.

\item \label{op6} If $ \blackperm ( a ) = a $ then $ \blackperm^{\opsymbol} ( a_{\black} ) = a_{\black} $ but \begin{itemize}
\item[] otherwise $ \blackperm ( x ) = a $ implies $ \blackperm^{\opsymbol} ( x ) = a_{\black} $, and $ \blackperm^{\opsymbol} ( a_{\black} ) = \blackperm ( a ) $.
\end{itemize}

\item \label{op7} If $ x \neq a $ and $ \blackperm ( x ) \neq a $ then $ \blackperm^{\opsymbol} ( x ) = \blackperm ( x ) $.
\end{enumerate}
\end{lem}

\begin{proof}
This is clear from the definitions of $ \whiteperm^{\opsymbol}, \blackperm^{\opsymbol} $.
\end{proof}

The operation $ \opsymbol $ is the group-theoretic manifestation of the following picture:
\begin{center}
\begin{figure}[h]
\includegraphics[scale=0.4]{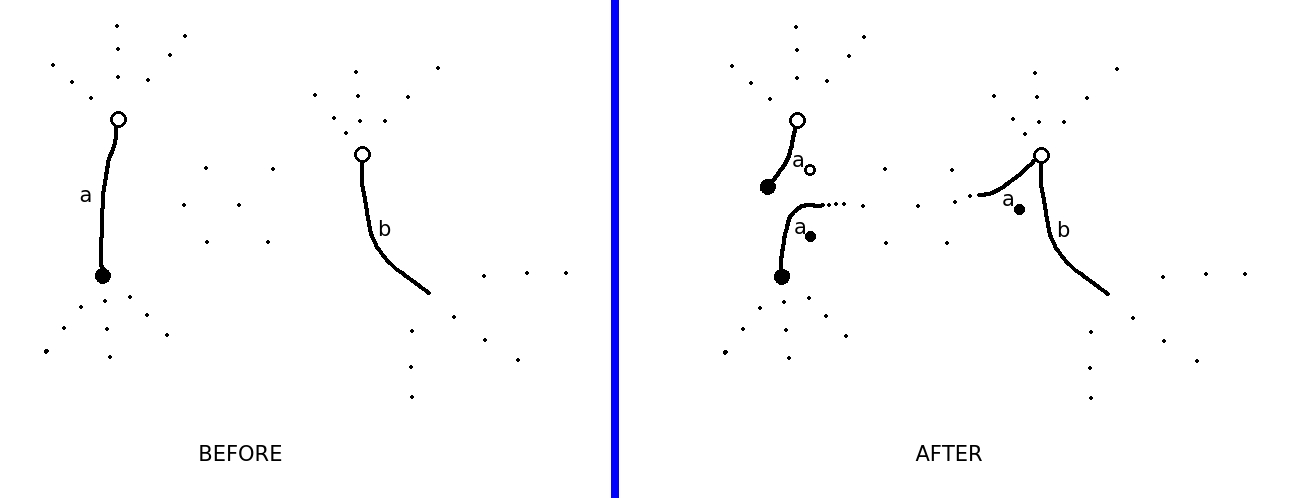}
\caption{\footnotesize{[Depiction of $ \opsymbol $] The purpose of the seemingly useless new edge $ a_{\white} $ is twofold: it guarantees that a degenerate model is never produced from a nondegenerate one, and it serves as a sort of ``bookmark'', recording where edge ``a'' used to be for future applications of the $ \opsymbol $ operation...}}
\label{FIGoperation}
\end{figure}
\end{center}

The goal of this section is to determine the exact relationship of $ \numcycles ( \whiteperm^{\opsymbol} \blackperm^{\opsymbol} ) $ to $ \numcycles ( \whiteperm \blackperm ) $. This relationship, the Reroute Theorem \ref{THMgenus} below, and its proof are the technical foundation of the paper.

One reason why this relationship is important is that it predicts \emph{genus}:

\begin{lem}[Genus Change] \label{LEMchardiff}
Let $ ( \whiteperm^{\opsymbol}, \blackperm^{\opsymbol} ) $ be the reroute of $ ( \whiteperm, \blackperm ) $ relative to $ ( a, b ) $. \Assertion For $ \chi $ and $ \chi^{\opsymbol} $ the synthetic Euler characteristics of $ ( \whiteperm, \blackperm ) $ and $ ( \whiteperm^{\opsymbol}, \blackperm^{\opsymbol} ) $,
\begin{equation*}
\chi - \chi^{\opsymbol} = \numcycles ( \whiteperm \blackperm ) - \numcycles ( \whiteperm^{\opsymbol} \blackperm^{\opsymbol} )
\end{equation*}
\end{lem}

Equivalently, if $ g $ is the synthetic genus of $ ( \whiteperm, \blackperm ) $ and $ g^{\opsymbol} $ is that of $ ( \whiteperm^{\opsymbol}, \blackperm^{\opsymbol} ) $ then
\begin{equation*}
g^{\opsymbol} - g = \frac{ \numcycles ( \whiteperm \blackperm ) - \numcycles ( \whiteperm^{\opsymbol} \blackperm^{\opsymbol} ) }{2}
\end{equation*}

\begin{proof}
This is obvious: $ \numcycles ( \whiteperm^{\opsymbol} ) = \numcycles ( \whiteperm ) $, $ \numcycles ( \blackperm^{\opsymbol} ) = \numcycles ( \blackperm ) + 1 $, $ \vert E^{\opsymbol} \vert = \vert E \vert + 1 $.
\end{proof}

\subsection{Comparison of orbits}

\begin{dfn}[Unbiased] \label{DFNunbiasedorbits}
An $ \whiteperm \blackperm $-orbit $ O $ is \emph{unbiased} iff $ O \cap \{ a, \whiteperm ( a ), b \} = \emptyset $.
\end{dfn}

The previous definition is justified by:

\begin{lem}[Unbiased Orbits] \label{LEMunbiasedorbitsalso}
If $ x \in E $, $ x \neq a $ and $ \whiteperm \blackperm ( x ) \notin \{ a, \whiteperm ( a ), b \} $ then $ x \in E^{\opsymbol} $ and $ \whiteperm^{\opsymbol} \blackperm^{\opsymbol} ( x ) = \whiteperm \blackperm ( x ) $. In particular, if $ O $ is an unbiased $ \whiteperm \blackperm $-orbit then $ O $ is also an $ \whiteperm^{\opsymbol} \blackperm^{\opsymbol} $-orbit.
\end{lem}

\begin{proof}
It is trivial that $ x \in E^{\opsymbol} $. It is immediate from the hypotheses that $ \blackperm ( x ) \neq a $, so the definition (\ref{op7}) of $ \blackperm^{\opsymbol} $ implies that $ \blackperm^{\opsymbol} ( x ) = \blackperm ( x ) $. The hypotheses further imply that $ \whiteperm ( \blackperm ( x ) ) \neq a, b $, so the definition (\ref{op4}) of $ \whiteperm^{\opsymbol} $ implies that $ \whiteperm^{\opsymbol} ( \blackperm ( x ) ) = \whiteperm ( \blackperm ( x ) ) $. Combining the two equalities yields the first claim. The second claim is immediate from the first.
\end{proof}

On the other hand,

\begin{dfn}[Biased] \label{DFNbiasedorbits}
An $ \whiteperm^{\opsymbol} \blackperm^{\opsymbol} $-orbit $ O^{\opsymbol} $ is \emph{biased} iff it is not also an unbiased $ \whiteperm \blackperm $-orbit (see Lemma \ref{LEMunbiasedorbitsalso}).
\end{dfn}

The previous definition is justified by:

\begin{lem}[Biased Orbits] \label{LEMbiasedorbitscontain}
If $ O^{\opsymbol} $ is a biased $ \whiteperm^{\opsymbol} \blackperm^{\opsymbol} $-orbit then $ O^{\opsymbol} $ contains at least one of $ a_{\white}, a_{\black}, b $.
\end{lem}

The style of argument here will be repeated many times throughout the rest of the paper. Note that $ a_{\white}, a_{\black}, b $ are distinct by definition, but that $ a, \whiteperm ( a ), b $ may not be. The possibility that $ \whiteperm ( a ) \in \{ a, b \} $ will require special cases to be treated in most of the proofs below. The first example of such a proof is this one.

\begin{proof}
Let $ O^{\opsymbol} $ be a biased $ \whiteperm^{\opsymbol} \blackperm^{\opsymbol} $-orbit and let $ x \in O^{\opsymbol} $ be arbitrary. Note that $ x \neq a $ because $ a \notin E^{\opsymbol} $. By nature of the claim, I can also assume that $ x \neq a_{\white}, a_{\black}, b $. Finally, I can assume that $ x \neq \whiteperm ( a ) $: if $ x = \whiteperm ( a ) $ then necessarily $ \whiteperm ( a ) \neq a $ and so the definitions (\ref{op1}) (\ref{op5}) of $ \whiteperm^{\opsymbol}, \blackperm^{\opsymbol} $ imply that $ \whiteperm^{\opsymbol} \blackperm^{\opsymbol} ( a_{\white} ) = \whiteperm ( a ) = x $, as desired. In summary, I can assume $ x \in E $ and $ x \neq a, \whiteperm ( a ), b $.

By definition of ``biased'' and Lemma \ref{LEMunbiasedorbitsalso}, the $ \whiteperm \blackperm $-orbit of $ x $ contains at least one of $ a, \whiteperm ( a ), b $. Among all $ \whiteperm \blackperm $-sequences from $ x $ to one of $ a, \whiteperm ( a ), b $, let $ x_0, x_1, \ldots, x_n $ be the one with minimal length. By the previous paragraph, this sequence is not a singleton ($ n \geq 1 $). By minimality, $ x_i \neq a, \whiteperm ( a ), b $ for all $ 0 < i < n $.

By using Lemma \ref{LEMunbiasedorbitsalso} repeatedly, $ x_0, x_1, \ldots, x_{n-1} $ is a $ \whiteperm^{\opsymbol} \blackperm^{\opsymbol} $-sequence. To complete the proof, I show that $ \whiteperm^{\opsymbol} \blackperm^{\opsymbol} ( x_{n-1} ) \in \{ a_{\white}, a_{\black}, b \} $. This will be done for each of the three possibilities for $ x_n $.

Suppose first that $ x_n = \whiteperm ( a ) $. Since $ \whiteperm \blackperm ( x_{n-1} ) = x_n $, it follows that $ \blackperm ( x_{n-1} ) = a $. Since $ x_{n-1} \neq a $ by the first/second paragraph, definition (\ref{op6}) of $ \blackperm^{\opsymbol} $ says that $ \blackperm^{\opsymbol} ( x_{n-1} ) = a_{\black} $. Definition (\ref{op3}) of $ \whiteperm^{\opsymbol} $ then says that $ \whiteperm^{\opsymbol} \blackperm^{\opsymbol} ( x_{n-1} ) = b $, as desired.

Before treating the other two cases, it will be efficient to make a comment. By the previous paragraph, I can assume that $ x_n \neq \whiteperm ( a ) $. Since $ \whiteperm \blackperm ( x_{n-1} ) = x_n $, it follows that $ \blackperm ( x_{n-1} ) \neq a $. Since it is known from the first/second paragraph that $ x_{n-1} \neq a $, definition (\ref{op7}) of $ \blackperm^{\opsymbol} $ says that $ \blackperm^{\opsymbol} ( x_{n-1} ) = \blackperm ( x_{n-1} ) $. Thus, to prove the claim for the remaining two cases it suffices merely to show that $ \whiteperm^{\opsymbol} ( \blackperm ( x_{n-1} ) ) \in \{ a_{\white}, a_{\black}, b \} $.

Suppose now that $ x_n = b $. Since $ \whiteperm ( \blackperm ( x_{n-1} ) ) = x_n $ and $ \whiteperm ( a ) \neq x_n = b $, definition (\ref{op2}) of $ \whiteperm^{\opsymbol} $ says that $ \whiteperm^{\opsymbol} ( \blackperm ( x_{n-1} ) ) = a_{\black} $.

Suppose finally that $ x_n = a $. Since $ \whiteperm ( \blackperm ( x_{n-1} ) ) = x_n $ and $ \whiteperm ( a ) \neq x_n = a $, definition (\ref{op1}) of $ \whiteperm^{\opsymbol} $ says that $ \whiteperm^{\opsymbol} ( \blackperm ( x_{n-1} ) ) = a_{\white} $.
\end{proof}

The relevance of all this to the determination of the relationship between $ \numcycles ( \whiteperm \blackperm ) $ and $ \numcycles ( \whiteperm^{\opsymbol} \blackperm^{\opsymbol} ) $ is clear:

\begin{prp}[Orbit Counting] \label{PRPorbitcounting}
Let $ U $ be the number of unbiased $ \whiteperm \blackperm $-orbits, let $ B $ be the number of $ \whiteperm \blackperm $-orbits containing at least one of $ a, \whiteperm ( a ), b $, and let $ B^{\opsymbol} $ be the number of $ \whiteperm^{\opsymbol} \blackperm^{\opsymbol} $-orbits containing at least one of $ a_{\white}, a_{\black}, b $. \Assertion $ \numcycles ( \whiteperm \blackperm ) = U + B $ and $ \numcycles ( \whiteperm^{\opsymbol} \blackperm^{\opsymbol} ) = U + B^{\opsymbol} $.
\end{prp}

\begin{proof}
This is immediate from Definitions \ref{DFNunbiasedorbits}/\ref{DFNbiasedorbits} and Lemmas \ref{LEMunbiasedorbitsalso}/\ref{LEMbiasedorbitscontain}.
\end{proof}

Thus, the problem is to calculate the difference $ B^{\opsymbol} - B $. The following can be used to calculate the difference $ B^{\opsymbol} - B $, and more:

\begin{orbittransfer} \label{LEMorbittransfer}
Let $ ( \whiteperm^{\opsymbol}, \blackperm^{\opsymbol} ) $ be the reroute of $ ( \whiteperm, \blackperm ) $ relative to $ ( a, b ) $.

A $ \whiteperm \blackperm $-sequence $ x_0, \ldots, x_n $ is called \emph{strict} relative to $ ( a, b ) $ iff it contains at least two terms ($ n \geq 1 $) and $ x_i \neq a, \whiteperm ( a ), b $ for all $ 0 < i < n $. \emph{It is allowed that $ x_0 = x_n $.}

\Assertions

\begin{enumerate}
\setlength{\itemsep}{5pt}

\item \label{LorbittransferP1P2} If $ x_0, \ldots x_i \ldots, x_n \in E $ is a strict $ \whiteperm \blackperm $-sequence from $ a $ to $ \whiteperm ( a ) $ then $ a_{\black}, \ldots x_i \ldots, b $ is a $ \whiteperm^{\opsymbol} \blackperm^{\opsymbol} $-sequence. In particular, $ a_{\black} $ and $ b $ represent the same $ \whiteperm^{\opsymbol} \blackperm^{\opsymbol} $-orbit. \emph{It is allowed that the $ \whiteperm \blackperm $-sequence has no interior terms: if $ a, \whiteperm ( a ) $ is a $ \whiteperm \blackperm $-sequence then $ a_{\black}, b $ is a $ \whiteperm^{\opsymbol} \blackperm^{\opsymbol} $-sequence.}

\item \label{LorbittransferP1P4} Assume that $ \whiteperm ( a ) \neq a, b $. If $ x_0, \ldots x_i \ldots, x_n \in E $ is a strict $ \whiteperm \blackperm $-sequence from $ \whiteperm ( a ) $ to $ b $ then $ a_{\white}, \whiteperm ( a ), \ldots x_i \ldots, a_{\black} $ is a $ \whiteperm^{\opsymbol} \blackperm^{\opsymbol} $-sequence. In particular, $ a_{\white} $ and $ a_{\black} $ represent the same $ \whiteperm^{\opsymbol} \blackperm^{\opsymbol} $-orbit. \emph{It is allowed that the $ \whiteperm \blackperm $-sequence has no interior terms: if $ \whiteperm ( a ), b $ is a $ \whiteperm \blackperm $-sequence then $ a_{\white}, \whiteperm ( a ), a_{\black} $ is a $ \whiteperm^{\opsymbol} \blackperm^{\opsymbol} $-sequence.} If $ \whiteperm ( a ) = b $ then $ \whiteperm^{\opsymbol} \blackperm^{\opsymbol} ( a_{\white} ) = a_{\black} $. If $ \whiteperm ( a ) = a $ then the conclusion is definitely false.

\item \label{LorbittransferP1P3} Assume that $ \whiteperm ( a ) \neq a $. If $ x_0, \ldots x_i \ldots, x_n \in E $ is a strict $ \whiteperm \blackperm $-sequence from $ b $ to $ a $ then $ b, \ldots x_i \ldots, a_{\white} $ is a $ \whiteperm^{\opsymbol} \blackperm^{\opsymbol} $-sequence. In particular, $ b $ and $ a_{\white} $ represent the same $ \whiteperm^{\opsymbol} \blackperm^{\opsymbol} $-orbit. \emph{It is allowed that the $ \whiteperm \blackperm $-sequence has no interior terms: if $ b, a $ is a $ \whiteperm \blackperm $-sequence then $ b, a_{\white} $ is a $ \whiteperm^{\opsymbol} \blackperm^{\opsymbol} $-sequence.} \emph{If $ \whiteperm ( a ) = a $ then the conclusion is definitely false.}

\item \label{LorbittransferP4N} If $ x_0, \ldots x_i \ldots, x_n \in E $ is a strict $ \whiteperm \blackperm $-sequence from $ b $ to $ \whiteperm ( a ) $ then $ b, \ldots x_i \ldots, b $ is a $ \whiteperm^{\opsymbol} \blackperm^{\opsymbol} $-sequence. In particular, $ b $ represents a different $ \whiteperm^{\opsymbol} \blackperm^{\opsymbol} $-orbit than both $ a_{\white} $ and $ a_{\black} $. \emph{It is allowed that the $ \whiteperm \blackperm $-sequence has no interior terms: if $ b, \whiteperm ( a ) $ is a $ \whiteperm \blackperm $-sequence then $ b, b $ is a $ \whiteperm^{\opsymbol} \blackperm^{\opsymbol} $-sequence, i.e. $ b $ is fixed by $ \whiteperm^{\opsymbol} \blackperm^{\opsymbol} $.}

\item \label{LorbittransferP2N} Assume $ \whiteperm ( a ) \neq a, b $. If $ x_0, \ldots x_i \ldots, x_n \in E $ is a strict $ \whiteperm \blackperm $-sequence from $ \whiteperm ( a ) $ to $ a $ then $ a_{\white}, \whiteperm ( a ), \ldots x_i \ldots, a_{\white} $ is a $ \whiteperm^{\opsymbol} \blackperm^{\opsymbol} $-sequence. In particular, $ a_{\white} $ represents a different $ \whiteperm^{\opsymbol} \blackperm^{\opsymbol} $-orbit than both $ a_{\black} $ and $ b $. \emph{It is allowed that the $ \whiteperm \blackperm $-sequence has no interior terms: if $ \whiteperm ( a ), a $ is a $ \whiteperm \blackperm $-sequence then $ a_{\white}, \whiteperm ( a ), a_{\white} $ is a $ \whiteperm^{\opsymbol} \blackperm^{\opsymbol} $-sequence.} If $ \whiteperm ( a ) = a $ then $ \whiteperm^{\opsymbol} \blackperm^{\opsymbol} ( a_{\white} ) = ( a_{\white} ) $. If $ \whiteperm ( a ) = b $ then the conclusion is definitely false.

\item \label{LorbittransferP3N} Assume that $ \whiteperm ( a ) \neq b $. If $ x_0, \ldots x_i \ldots, x_n \in E $ is a strict $ \whiteperm \blackperm $-sequence from $ a $ to $ b $ then $ a_{\black}, \ldots x_i \ldots, a_{\black} $ is a $ \whiteperm^{\opsymbol} \blackperm^{\opsymbol} $-sequence. In particular, $ a_{\black} $ represents a different $ \whiteperm^{\opsymbol} \blackperm^{\opsymbol} $-orbit than both $ a_{\white} $ and $ b $. \emph{It is allowed that the $ \whiteperm \blackperm $-sequence has no interior terms: if $ a, b $ is a $ \whiteperm \blackperm $-sequence then $ a_{\black}, a_{\black} $ is a $ \whiteperm^{\opsymbol} \blackperm^{\opsymbol} $-sequence, i.e. $ a_{\black} $ is fixed by $ \whiteperm^{\opsymbol} \blackperm^{\opsymbol} $.} If $ \whiteperm ( a ) = b $ then the conclusion is false.

\item \label{LorbittransferC} If $ x_0, \ldots x_i \ldots, x_n \in E $ is a strict $ \whiteperm \blackperm $-sequence from $ a $ to $ a $ then $ a_{\black}, \ldots x_i \ldots, a_{\white} $ is a $ \whiteperm^{\opsymbol} \blackperm^{\opsymbol} $-sequence. In particular, $ a_{\black} $ and $ a_{\white} $ represent the same $ \whiteperm^{\opsymbol} \blackperm^{\opsymbol} $-orbit. \emph{It is allowed that the $ \whiteperm \blackperm $-sequence has no interior terms: if $ a, a $ is a $ \whiteperm \blackperm $-sequence then $ a_{\black}, a_{\white} $ is a $ \whiteperm^{\opsymbol} \blackperm^{\opsymbol} $-sequence.}

\item \label{LorbittransferB} If $ x_0, \ldots x_i \ldots, x_n \in E $ is a strict $ \whiteperm \blackperm $-sequence from $ \whiteperm ( a ) $ to $ \whiteperm ( a ) $ then $ a_{\white}, \whiteperm ( a ), \ldots x_i \ldots, b $ is a $ \whiteperm^{\opsymbol} \blackperm^{\opsymbol} $-sequence. In particular, $ a_{\white} $ and $ b $ represent the same $ \whiteperm^{\opsymbol} \blackperm^{\opsymbol} $-orbit. \emph{It is allowed that the $ \whiteperm \blackperm $-sequence has no interior terms: if $ \whiteperm ( a ), \whiteperm ( a ) $ is a $ \whiteperm \blackperm $-sequence then $ a_{\white}, \whiteperm ( a ), b $ is a $ \whiteperm^{\opsymbol} \blackperm^{\opsymbol} $-sequence.}

\item \label{LorbittransferD} If $ x_0, \ldots x_i \ldots, x_n \in E $ is a strict $ \whiteperm \blackperm $-sequence from $ b $ to $ b $ then $ b, \ldots x_i \ldots, a_{\black} $ is a $ \whiteperm^{\opsymbol} \blackperm^{\opsymbol} $-sequence. In particular, $ b $ and $ a_{\black} $ represent the same $ \whiteperm^{\opsymbol} \blackperm^{\opsymbol} $-orbit. \emph{It is allowed that the $ \whiteperm \blackperm $-sequence has no interior terms: if $ b, b $ is a $ \whiteperm \blackperm $-sequence then $ b, a_{\black} $ is a $ \whiteperm^{\opsymbol} \blackperm^{\opsymbol} $-sequence.}
\end{enumerate}
\end{orbittransfer}

\begin{proof} \proofbox{assertion (\ref{LorbittransferP1P2})} If the $ \whiteperm \blackperm $-sequence has only two terms then necessarily $ \blackperm ( a ) = a $ and so it is immediate from the definitions (\ref{op3}) (\ref{op6}) of $ \whiteperm^{\opsymbol}, \blackperm^{\opsymbol} $ that $ \whiteperm^{\opsymbol} \blackperm^{\opsymbol} ( a_{\black} ) = b $. So, I can assume from now on that $ n \geq 2 $. It follows that $ \blackperm ( x_0 ) \neq a $ and $ \whiteperm ( \blackperm ( x_0 ) ) \neq a, b $ since otherwise $ x_1 \in \{ a, \whiteperm ( a ), b \} $ and the assumption ``strict'' would be contradicted. Thus, $ \blackperm^{\opsymbol} ( a_{\black} ) = \blackperm ( x_0 ) $ by the definition (\ref{op6}) of $ \blackperm^{\opsymbol} $, and $ \whiteperm^{\opsymbol} ( \blackperm ( x_0 ) ) = \whiteperm ( \blackperm ( x_0 ) ) $ by the definition (\ref{op4}) of $ \whiteperm^{\opsymbol} $. Combined, $ a_{\black}, x_1 $ is a $ \whiteperm^{\opsymbol} \blackperm^{\opsymbol} $-sequence. Using ``strict'' again and Lemma \ref{LEMunbiasedorbitsalso} repeatedly, $ a_{\black}, x_1, \ldots, x_{n-1} $ is a $ \whiteperm^{\opsymbol} \blackperm^{\opsymbol} $-sequence. It remains to show that $ \whiteperm^{\opsymbol} \blackperm^{\opsymbol} ( x_{n-1} ) = b $. Since $ \whiteperm \blackperm ( x_{n-1} ) = x_n = \whiteperm ( a ) $, we have $ \blackperm ( x_{n-1} ) = a $. Since $ x_{n-1} \neq a $ by ``strict'' and $ n \geq 2 $, the definition (\ref{op6}) of $ \blackperm^{\opsymbol} $ implies that $ \blackperm^{\opsymbol} ( x_{n-1} ) = a_{\black} $. The definition (\ref{op3}) of $ \whiteperm^{\opsymbol} $ then implies that $ \whiteperm^{\opsymbol} \blackperm^{\opsymbol} ( x_{n-1} ) = b $.

\proofbox{assertion (\ref{LorbittransferP1P4})} Because of the assumption $ \whiteperm ( a ) \neq a $, the definitions (\ref{op1}) (\ref{op5}) of $ \whiteperm^{\opsymbol}, \blackperm^{\opsymbol} $ imply that $ \whiteperm^{\opsymbol} \blackperm^{\opsymbol} ( a_{\white} ) = \whiteperm ( a ) $. Thus, it suffices to show that $ \whiteperm ( a ), \ldots x_i \ldots, a_{\black} $ is a $ \whiteperm^{\opsymbol} \blackperm^{\opsymbol} $-sequence. Note that $ \blackperm ( \whiteperm ( a ) ) \neq a $ because otherwise $ \whiteperm ( a ) $ would be a fixed point of $ \whiteperm \blackperm $, contradicting the fact that $ \whiteperm ( a ) $ represents the same $ \whiteperm \blackperm $-orbit as $ b \neq \whiteperm ( a ) $. Together with the assumption $ x_0 = \whiteperm ( a ) \neq a $, the definition (\ref{op7}) of $ \blackperm^{\opsymbol} $ says that $ \blackperm^{\opsymbol} ( x_0 ) = \blackperm ( x_0 ) $. Suppose that the $ \whiteperm \blackperm $-sequence contains only two terms, i.e. $ \whiteperm \blackperm ( x_0 ) = b $. Using the assumption $ \whiteperm ( a ) \neq b $, the definition (\ref{op2}) of $ \whiteperm^{\opsymbol} $ says that $ \whiteperm^{\opsymbol} ( \blackperm ( x_0 ) ) = a_{\black} $. Combining with the known $ \blackperm^{\opsymbol} ( x_0 ) = \blackperm ( x_0 ) $ shows that $ \whiteperm ( a ), a_{\black} $ is a $ \whiteperm^{\opsymbol} \blackperm^{\opsymbol} $-sequence. Suppose now that the $ \whiteperm \blackperm $-sequence contains at least three terms ($ n \geq 2 $). Since $ x_0 = \whiteperm ( a ) \neq a $ by assumption and $ x_i \neq a, \whiteperm ( a ), b $ for all $ 0 < i < n $ by ``strict'', Lemma \ref{LEMunbiasedorbitsalso} says that $ x_0, x_1, \ldots, x_{n-1} $ is a $ \whiteperm^{\opsymbol} \blackperm^{\opsymbol} $-sequence. It remains to show that $ \whiteperm^{\opsymbol} \blackperm^{\opsymbol} ( x_{n-1} ) = a_{\black} $. It must be true that $ \blackperm ( x_{n-1} ) \neq a $, because otherwise $ b = x_n = \whiteperm \blackperm ( x_{n-1} ) = \whiteperm ( a ) $, contradicting the assumption $ \whiteperm ( a ) \neq b $. Since $ x_{n-1} \neq a $ by ``strict'' and $ n \geq 2 $, the definition (\ref{op7}) of $ \blackperm^{\opsymbol} $ implies that $ \blackperm^{\opsymbol} ( x_{n-1} ) = \blackperm ( x_{n-1} ) $. Using the assumption $ \whiteperm ( a ) \neq b $ and the fact that $ \whiteperm ( \blackperm ( x_{n-1} ) ) = b $, the definition (\ref{op2}) of $ \whiteperm^{\opsymbol} $ implies that $ \whiteperm^{\opsymbol} ( \blackperm ( x_{n-1} ) ) = a_{\black} $. Combining the two equalities yields $ \whiteperm^{\opsymbol} \blackperm^{\opsymbol} ( x_{n-1} ) = a_{\black} $. The last claim is easy: if $ \whiteperm ( a ) = b $ then it is immediate from the definitions (\ref{op2}) (\ref{op5}) of $ \whiteperm^{\opsymbol}, \blackperm^{\opsymbol} $ that $ \whiteperm^{\opsymbol} \blackperm^{\opsymbol} ( a_{\white} ) = a_{\black} $. If $ \whiteperm ( a ) = a $ then it is immediate from the definitions (\ref{op1}) (\ref{op5}) of $ \whiteperm^{\opsymbol}, \blackperm^{\opsymbol} $ that $ \whiteperm^{\opsymbol} \blackperm^{\opsymbol} ( a_{\white} ) = a_{\white} $, so the conclusion cannot possibly be true.

\proofbox{assertion (\ref{LorbittransferP1P3})} It must be true that $ \blackperm ( b ) \neq a $, because otherwise $ \whiteperm \blackperm ( b ) = \whiteperm ( a ) $, contradicting either $ \whiteperm ( a ) \neq a $ (if $ n = 1 $) or the ``strict'' assumption (if $ n \geq 2 $). It follows from this and the definition (\ref{op7}) of $ \blackperm^{\opsymbol} $ that $ \blackperm^{\opsymbol} ( b ) = \blackperm ( b ) $. Suppose that the $ \whiteperm \blackperm $-sequence contains only two terms, i.e. that $ \whiteperm \blackperm ( b ) = a $. Since $ \whiteperm ( a ) \neq a $ is assumed, the definition (\ref{op1}) of $ \whiteperm^{\opsymbol} $ implies that $ \whiteperm^{\opsymbol} ( \blackperm ( b ) ) = a_{\white} $. Combining with the known $ \blackperm^{\opsymbol} ( b ) = \blackperm ( b ) $ yields $ \whiteperm^{\opsymbol} \blackperm^{\opsymbol} ( b ) = a_{\white} $. Suppose now that the $ \whiteperm \blackperm $-sequence contains at least three terms ($ n \geq 2 $). Since $ x_0 = b \neq a $ and $ x_i \neq a, \whiteperm ( a ), b $ by ``strict'', Lemma \ref{LEMunbiasedorbitsalso} says that $ x_0, x_1, \ldots, x_{n-1} $ is a $ \whiteperm^{\opsymbol} \blackperm^{\opsymbol} $-sequence. It remains to show that $ \whiteperm^{\opsymbol} \blackperm^{\opsymbol} ( x_{n-1} ) = a_{\white} $. By ``strict'' and $ n \geq 2 $, $ x_{n-1} \neq a $. Also, $ \blackperm ( x_{n-1} ) \neq a $ since otherwise $ a = x_n = \whiteperm \blackperm ( x_{n-1} ) = \whiteperm ( a ) $, contradicting the assumption. These two facts and the definition (\ref{op7}) of $ \blackperm^{\opsymbol} $ imply that $ \blackperm^{\opsymbol} ( x_{n-1} ) = \blackperm ( x_{n-1} ) $. Since $ \whiteperm ( a ) \neq a $ by assumption, the fact that $ \whiteperm \blackperm ( x_{n-1} ) = x_n = a $ and the definition (\ref{op1}) of $ \whiteperm^{\opsymbol} $ imply that $ \whiteperm^{\opsymbol} ( \blackperm ( x_{n-1} ) ) = a_{\white} $. Combining these two equalities yields $ \whiteperm^{\opsymbol} \blackperm^{\opsymbol} ( x_{n-1} ) = a_{\white} $. If $ \whiteperm ( a ) = a $ then it is immediate from the definitions (\ref{op1}) (\ref{op5}) of $ \whiteperm^{\opsymbol}, \blackperm^{\opsymbol} $ that $ \whiteperm^{\opsymbol} \blackperm^{\opsymbol} ( a_{\white} ) = a_{\white} $, so the conclusion cannot possibly be true.

\proofbox{assertion (\ref{LorbittransferP4N})} If the $ \whiteperm \blackperm $-sequence contains only two terms then necessarily $ \blackperm ( b ) = a $ and so it is immediate from the definitions (\ref{op3}) (\ref{op6}) of $ \whiteperm^{\opsymbol}, \blackperm^{\opsymbol} $ that $ \whiteperm^{\opsymbol} \blackperm^{\opsymbol} ( b ) = b $. So, I can assume that $ n \geq 2 $. Necessarily $ \blackperm ( b ) \neq a $ since otherwise $ x_1 = \whiteperm ( a ) $, contradicting ``strict'' and $ n \geq 2 $. Since $ x_i \neq a, \whiteperm ( a ), b $ for all $ 0 < i < n $ by ``strict'', Lemma \ref{LEMunbiasedorbitsalso} says that $ b, x_1, \ldots, x_{n-1} $ is a $ \whiteperm^{\opsymbol} \blackperm^{\opsymbol} $-sequence. It remains to show $ \whiteperm^{\opsymbol} \blackperm^{\opsymbol} ( x_{n-1} ) = b $. Since $ x_{n-1} \neq a $ by ``strict'' and $ n \geq 2 $, and since $ \whiteperm \blackperm ( x_{n-1} ) = x_n = \whiteperm ( a ) $ implies $ \blackperm ( x_{n-1} ) = a $, the definition (\ref{op6}) of $ \blackperm^{\opsymbol} $ says that $ \blackperm^{\opsymbol} ( x_{n-1} ) = a_{\black} $. Definition (\ref{op3}) of $ \whiteperm^{\opsymbol} $ says that $ \whiteperm^{\opsymbol} \blackperm^{\opsymbol} ( x_{n-1} ) = b $.

\proofbox{assertion (\ref{LorbittransferP2N})} Since $ \whiteperm ( a ) \neq a $ by assumption, the definitions (\ref{op1}) (\ref{op5}) of $ \whiteperm^{\opsymbol}, \blackperm^{\opsymbol} $ say that $ \whiteperm^{\opsymbol} \blackperm^{\opsymbol} ( a_{\white} ) = \whiteperm ( a ) $. So, it suffices to show that $ \whiteperm ( a ), \ldots x_i \ldots a_{\white} $ is a $ \whiteperm^{\opsymbol} \blackperm^{\opsymbol} $-sequence. It must be true that $ \blackperm ( x_0 ) \neq a $ since otherwise $ x_0 = \whiteperm ( a ) $ would be a fixed point of $ \whiteperm \blackperm $, contradicting the assumption that $ a $ represents the same $ \whiteperm \blackperm $-orbit as $ \whiteperm ( a ) \neq a $. Combined with the assumption $ x_0 = \whiteperm ( a ) \neq a $, the definition (\ref{op7}) of $ \blackperm^{\opsymbol} $ implies that $ \blackperm^{\opsymbol} ( x_0 ) = \blackperm ( x_0 ) $. Suppose that the $ \whiteperm \blackperm $-sequence contains only two terms, i.e. that $ \whiteperm \blackperm ( \whiteperm ( a ) ) = a $. Since $ \whiteperm ( a ) \neq a $ by assumption, the definition (\ref{op1}) of $ \whiteperm^{\opsymbol} $ says that $ \whiteperm^{\opsymbol} ( \blackperm ( x_0 ) ) = a_{\white} $. Since $ \blackperm^{\opsymbol} ( x_0 ) = \blackperm ( x_0 ) $ is known, $ \whiteperm^{\opsymbol} \blackperm^{\opsymbol} ( x_0 ) = a_{\white} $. Suppose now that the $ \whiteperm \blackperm $-sequence contains at least three terms ($ n \geq 2 $). Since $ x_0 = \whiteperm ( a ) \neq a $ by assumption, and $ x_i \neq a, \whiteperm ( a ), b $ for all $ 0 < i < n $ by ``strict'', Lemma \ref{LEMunbiasedorbitsalso} says that $ \whiteperm ( a ), x_1, \ldots, x_{n-1} $ is a $ \whiteperm^{\opsymbol} \blackperm^{\opsymbol} $-sequence. It remains to show that $ \whiteperm^{\opsymbol} \blackperm^{\opsymbol} ( x_{n-1} ) = a_{\white} $. It must be true that $ \blackperm ( x_{n-1} ) \neq a $, since otherwise $ a = x_n = \whiteperm \blackperm ( x_{n-1} ) = \whiteperm ( a ) $, contradicting the assumption $ \whiteperm ( a ) \neq a $. Since $ x_{n-1} \neq a $ by ``strict'' and $ n \geq 2 $, the definition (\ref{op7}) of $ \blackperm^{\opsymbol} $ says that $ \blackperm^{\opsymbol} ( x_{n-1} ) = \blackperm ( x_{n-1} ) $. Since $ \whiteperm ( a ) \neq a $ by assumption, the fact that $ \whiteperm \blackperm ( x_{n-1} ) = a $ and the definition (\ref{op1}) of $ \whiteperm^{\opsymbol} $ imply that $ \whiteperm^{\opsymbol} ( \blackperm ( x_{n-1} ) ) = a_{\white} $. Combining the two equalities yields $ \whiteperm^{\opsymbol} \blackperm^{\opsymbol} ( x_{n-1} ) = a_{\white} $. The last claim is easy: if $ \whiteperm ( a ) = a $ then it is immediate from the definitions (\ref{op1}) (\ref{op5}) of $ \whiteperm^{\opsymbol}, \blackperm^{\opsymbol} $ that $ \whiteperm^{\opsymbol} \blackperm^{\opsymbol} ( a_{\white} ) = a_{\white} $. If $ \whiteperm ( a ) = b $ then it is immediate from the definitions (\ref{op2}) (\ref{op5}) of $ \whiteperm^{\opsymbol}, \blackperm^{\opsymbol} $ that $ \whiteperm^{\opsymbol} \blackperm^{\opsymbol} ( a_{\white} ) = a_{\black} $, so the conclusion cannot possibly be true.

\proofbox{assertion (\ref{LorbittransferP3N})} It must be true that $ \blackperm ( a ) \neq a $, since otherwise $ \whiteperm \blackperm ( a ) = \whiteperm ( a ) $, contradicting either the assumption $ \whiteperm ( a ) \neq b $ (if $ n = 1 $) or ``strict'' (if $ n \geq 2 $). The definition (\ref{op6}) of $ \blackperm^{\opsymbol} $ then implies that $ \blackperm^{\opsymbol} ( a_{\black} ) = \blackperm ( x_0 ) $. Suppose that the $ \whiteperm \blackperm $-sequence contains only two terms, i.e. that $ \whiteperm \blackperm ( x_0 ) = b $. Since $ \whiteperm ( a ) \neq b $ by assumption, the fact that $ \whiteperm \blackperm ( x_0 ) = b $ and the definition (\ref{op2}) of $ \whiteperm^{\opsymbol} $ imply that $ \whiteperm^{\opsymbol} ( \blackperm ( x_0 ) ) = a_{\black} $. Combined with the known $ \blackperm^{\opsymbol} ( a_{\black} ) = \blackperm ( x_0 ) $ yields $ \whiteperm^{\opsymbol} \blackperm^{\opsymbol} ( a_{\black} ) = a_{\black} $. Suppose now that the $ \whiteperm \blackperm $-sequence contains at least three terms ($ n \geq 2 $). It is known already that $ \blackperm ( a ) \neq a $, and $ \whiteperm ( \blackperm ( x_0 ) ) = x_1 \neq a, b $ by ``strict'', so definition (\ref{op4}) of $ \whiteperm^{\opsymbol} $ implies that $ \whiteperm^{\opsymbol} ( \blackperm ( x_0 ) ) = x_1 $. Combined with the known $ \blackperm^{\opsymbol} ( a_{\black} ) = \blackperm ( x_0 ) $ yields $ \whiteperm^{\opsymbol} \blackperm^{\opsymbol} ( a_{\black} ) = x_1 $. By ``strict'', Lemma \ref{LEMunbiasedorbitsalso} says that $ a_{\black}, x_1, \ldots, x_{n-1} $ is a $ \whiteperm^{\opsymbol} \blackperm^{\opsymbol} $-sequence. It remains to show that $ \whiteperm^{\opsymbol} \blackperm^{\opsymbol} ( x_{n-1} ) = a_{\black} $. It must be true that $ \blackperm ( x_{n-1} ) \neq a $, since otherwise implies $ b = x_n = \whiteperm \blackperm ( x_{n-1} ) = \whiteperm ( a ) $, contradicting the assumption $ \whiteperm ( a ) \neq b $. Since also $ x_{n-1} \neq a $ is known by ``strict'' and $ n \geq 2 $, the definition (\ref{op7}) of $ \blackperm^{\opsymbol} $ says $ \blackperm^{\opsymbol} ( x_{n-1} ) = \blackperm ( x_{n-1} ) $. Since $ \whiteperm ( a ) \neq b $ by assumption, the fact that $ \whiteperm \blackperm ( x_{n-1} ) = b $ and the definition (\ref{op2}) of $ \whiteperm^{\opsymbol} $ imply that $ \whiteperm^{\opsymbol} ( \blackperm ( x_{n-1} ) ) = a_{\black} $. Combining the two equalities yields $ \whiteperm^{\opsymbol} \blackperm^{\opsymbol} ( x_{n-1} ) = a_{\black} $. If $ \whiteperm ( a ) = b $ then it is immediate from the definitions (\ref{op2}) (\ref{op5}) of $ \whiteperm^{\opsymbol}, \blackperm^{\opsymbol} $ that $ \whiteperm^{\opsymbol} \blackperm^{\opsymbol} ( a_{\white} ) = a_{\black} $, so the conclusion cannot possibly be true.

\proofbox{assertion (\ref{LorbittransferC})} Note that the assumption implies $ \whiteperm ( a ) \neq a $. Further, $ \blackperm ( a ) \neq a $ because otherwise $ \whiteperm \blackperm ( a ) = \whiteperm ( a ) $, contradicting the assumption. Since $ \blackperm ( a ) \neq a $, the definition (\ref{op6}) of $ \blackperm^{\opsymbol} $ implies that $ \blackperm^{\opsymbol} ( a_{\black} ) = \blackperm ( x_0 ) $. Suppose that the $ \whiteperm \blackperm $-sequence contains only two terms, i.e. that $ \whiteperm \blackperm ( a ) = a $. Since $ \whiteperm ( a ) \neq a $ by assumption, the definition (\ref{op1}) of $ \whiteperm^{\opsymbol} $ implies that $ \whiteperm^{\opsymbol} ( \blackperm ( a ) ) = a_{\white} $. Combined with the known $ \blackperm^{\opsymbol} ( a_{\black} ) = \blackperm ( x_0 ) $ yields $ \whiteperm^{\opsymbol} \blackperm^{\opsymbol} ( a_{\black} ) = a_{\white} $. Suppose now that the $ \whiteperm \blackperm $-sequence contains at least three terms ($ n \geq 2 $). Since $ \blackperm ( x_0 ) \neq a $ is known and $ \whiteperm ( \blackperm ( x_0 ) ) = x_1 \neq a, b $ by ``strict'', the definition (\ref{op4}) of $ \whiteperm^{\opsymbol} $ says that $ \whiteperm^{\opsymbol} ( \blackperm ( x_0 ) ) = x_1 $. Combined with the known $ \blackperm^{\opsymbol} ( a_{\black} ) = \blackperm ( x_0 ) $ yields $ \whiteperm^{\opsymbol} \blackperm^{\opsymbol} ( a_{\black} ) = x_1 $. Since $ x_i \neq a, \whiteperm ( a ), b $ for all $ 0 < i < n $ by ``strict'', Lemma \ref{LEMunbiasedorbitsalso} says that $ a_{\black}, x_1, \ldots, x_{n-1} $ is a $ \whiteperm^{\opsymbol} \blackperm^{\opsymbol} $-sequence. It remains to show that $ \whiteperm^{\opsymbol} \blackperm^{\opsymbol} ( x_{n-1} ) = a_{\white} $. It must be true that $ \blackperm ( x_{n-1} ) \neq a $, since otherwise $ a = x_n = \whiteperm \blackperm ( x_{n-1} ) = \whiteperm ( a ) $, contrary to the assumption. Since $ x_{n-1} \neq a $ by ``strict'' and $ n \geq 2 $, the definition (\ref{op7}) of $ \blackperm^{\opsymbol} $ implies that $ \blackperm^{\opsymbol} ( x_{n-1} ) = \blackperm ( x_{n-1} ) $. Since $ \whiteperm ( a ) \neq a $ by assumption, the fact that $ \whiteperm \blackperm ( x_{n-1} ) = a $ and the definition (\ref{op1}) of $ \whiteperm^{\opsymbol} $ imply that $ \whiteperm^{\opsymbol} ( \blackperm ( x_{n-1} ) ) = a_{\white} $. Combined, the two equalities yield $ \whiteperm^{\opsymbol} \blackperm^{\opsymbol} ( x_{n-1} ) = a_{\white} $.

\proofbox{assertion (\ref{LorbittransferB})} By assumption, $ \whiteperm ( a ) \neq a, b $. In particular, the definitions (\ref{op1}) (\ref{op5}) of $ \whiteperm^{\opsymbol}, \blackperm^{\opsymbol} $ imply that $ \whiteperm^{\opsymbol} \blackperm^{\opsymbol} ( a_{\white} ) = \whiteperm ( a ) $. Therefore, it suffices merely to show that $ \whiteperm ( a ), \ldots x_i \ldots, b $ is a $ \whiteperm^{\opsymbol} \blackperm^{\opsymbol} $-sequence. Suppose that the $ \whiteperm \blackperm $-sequence contains only two terms, i.e. that $ \whiteperm \blackperm ( \whiteperm ( a ) ) = \whiteperm ( a ) $. Then $ \blackperm ( \whiteperm ( a ) ) = a $ and since $ \whiteperm ( a ) \neq a $ is known the definition (\ref{op6}) says $ \blackperm^{\opsymbol} ( \whiteperm ( a ) ) = a_{\black} $. Definition (\ref{op3}) of $ \whiteperm^{\opsymbol} $ implies that $ \whiteperm^{\opsymbol} \blackperm^{\opsymbol} ( \whiteperm ( a ) ) = b $. Suppose now that the $ \whiteperm \blackperm $-sequence contains at least three terms ($ n \geq 2 $). Since $ x_0 = \whiteperm ( a ) \neq a $ is known, and since $ x_i \neq a, \whiteperm ( a ), b $ for all $ 0 < i < n $ by ``strict'', Lemma \ref{LEMunbiasedorbitsalso} says that $ x_0, x_1, \ldots, x_{n-1} $ is a $ \whiteperm^{\opsymbol} \blackperm^{\opsymbol} $-sequence. It remains to show that $ \whiteperm^{\opsymbol} \blackperm^{\opsymbol} ( x_{n-1} ) = b $. Since $ \blackperm ( x_{n-1} ) = a $, because $ \whiteperm ( a ) = x_n = \whiteperm \blackperm ( x_{n-1} ) $, and since $ x_{n-1} \neq a $ by ``strict'' and $ n \geq 2 $, definition (\ref{op6}) of $ \blackperm^{\opsymbol} $ says $ \blackperm^{\opsymbol} ( x_{n-1} ) = a_{\black} $. Definition (\ref{op3}) of $ \whiteperm^{\opsymbol} $ then says $ \whiteperm^{\opsymbol} \blackperm^{\opsymbol} ( x_{n-1} ) = b $.

\proofbox{assertion (\ref{LorbittransferD})} Note that the assumption implies $ \whiteperm ( a ) \neq b $. Further, $ \blackperm ( b ) \neq a $, since otherwise $ \whiteperm \blackperm ( b ) = \whiteperm ( a ) $, contradicting the assumption. Thus, $ x_0, \blackperm ( x_0 ) \neq a $ and so the definition (\ref{op7}) of $ \blackperm^{\opsymbol} $ implies that $ \blackperm^{\opsymbol} ( x_0 ) = \blackperm ( x_0 ) $. Suppose that the $ \whiteperm \blackperm $-sequence contains only two terms, i.e. that $ \whiteperm \blackperm ( b ) = b $. Since $ \whiteperm ( a ) \neq b $ is known, the fact that $ \whiteperm \blackperm ( b ) = b $ and the definition (\ref{op2}) of $ \whiteperm^{\opsymbol} $ imply that $ \whiteperm^{\opsymbol} ( \blackperm ( b ) ) = a_{\black} $. Combining with the known $ \blackperm^{\opsymbol} ( x_0 ) = \blackperm ( x_0 ) $ yields $ \whiteperm^{\opsymbol} \blackperm^{\opsymbol} ( b ) = a_{\black} $. Suppose now that the arc contains at least three terms ($ n \geq 2 $). Since $ x_i \neq a, \whiteperm ( a ), b $ for all $ 0 < i < n $ by ``strict'', Lemma \ref{LEMunbiasedorbitsalso} says that $ b, x_1, \ldots, x_{n-1} $ is a $ \whiteperm^{\opsymbol} \blackperm^{\opsymbol} $-sequence. It remains to show that $ \whiteperm^{\opsymbol} \blackperm^{\opsymbol} ( x_{n-1} ) = a_{\black} $. It must be true that $ \blackperm ( x_{n-1} ) \neq a $ since otherwise $ b = x_n = \whiteperm \blackperm ( x_{n-1} ) = \whiteperm ( a ) $, contradicting the assumption. Since $ x_{n-1} \neq a $ by ``strict'' and $ n \geq 2 $, the definition (\ref{op7}) of $ \blackperm^{\opsymbol} $ says that $ \blackperm^{\opsymbol} ( x_{n-1} ) = \blackperm ( x_{n-1} ) $. Since $ \whiteperm ( a ) \neq b $ is known, the fact that $ \whiteperm \blackperm ( x_{n-1} ) = b $ and the definition (\ref{op2}) of $ \whiteperm^{\opsymbol} $ imply that $ \whiteperm^{\opsymbol} ( \blackperm ( x_{n-1} ) ) = a_{\black} $. Combining the two yields $ \whiteperm^{\opsymbol} \blackperm^{\opsymbol} ( x_{n-1} ) = a_{\black} $.
\end{proof}

\begin{rmk}
The reader may detect some redundancy among the many statements in Orbit Transfer Lemma \ref{LEMorbittransfer}. However, they will eventually all be needed in full detail.
\end{rmk}

\subsection{Types and Theorem}

The most important classification in the paper is the following:

\begin{dfn}[Type] \label{DFNtypes}
Recall, from \S\ref{Sarcs}, the notion of ``arc''. The pair $ ( \whiteperm, \blackperm ) $ is
\begin{itemize}
\setlength{\itemsep}{7pt}
\item[-] \emph{Type U (Unoriented)} relative to $ ( a, b ) $ iff $ a, \whiteperm ( a ), b $ are distinct and no two of them represent the same $ \whiteperm \blackperm $-orbit.

\item[-] \emph{Type N (Negatively Oriented)} relative to $ ( a, b ) $ iff $ a, \whiteperm ( a ), b $ represent the same $ \whiteperm \blackperm $-orbit and the $ \whiteperm \blackperm $-arc from $ a $ to $ b $ does not contain $ \whiteperm ( a ) $.

\item[-] \emph{Type P (Positively Oriented)} relative to $ ( a, b ) $ iff it is neither Type U nor Type N.
\end{itemize}
\end{dfn}

It is easy to see that this is a partition of $ S_E \times S_E $. Before stating and proving the Reroute Theorem \ref{THMgenus}, I subdivide Type P:

\begin{dfn}[P-Subtypes] \label{DFNsubtypes}
$ ( \whiteperm, \blackperm ) $ is
\begin{itemize}
\setlength{\itemsep}{7pt}
\item[-] \emph{Type P1} relative to $ ( a, b ) $ iff $ a, \whiteperm ( a ), b $ represent the same $ \whiteperm \blackperm $-orbit and the $ \whiteperm \blackperm $-arc from $ a $ to $ b $ contains $ \whiteperm ( a ) $.

\item[-] \emph{Type P2} relative to $ ( a, b ) $ iff $ a, \whiteperm ( a ) $ represent the same $ \whiteperm \blackperm $-orbit, different from that of $ b $.

\item[-] \emph{Type P3} relative to $ ( a, b ) $ iff $ a, b $ represent the same $ \whiteperm \blackperm $-orbit, different from that of $ \whiteperm ( a ) $.

\item[-] \emph{Type P4} relative to $ ( a, b ) $ iff $ \whiteperm ( a ), b $ represent the same $ \whiteperm \blackperm $-orbit, different from that of $ a $.
\end{itemize}
\end{dfn}

It is easy to see that this is a partition of Type P. The definition of ``Type'' is mysterious but thoroughly justified by:

\begin{genustheorem} \label{THMgenus}
Let $ ( \whiteperm^{\opsymbol}, \blackperm^{\opsymbol} ) $ be the reroute of $ ( \whiteperm, \blackperm ) $ relative to $ ( a, b ) $. Let $ g $ be the synthetic genus of $ ( \whiteperm, \blackperm ) $ and $ g^{\opsymbol} $ that of $ ( \whiteperm^{\opsymbol}, \blackperm^{\opsymbol} ) $. \Assertions

\begin{enumerate}
\setlength{\itemsep}{5pt}
\item If $ ( \whiteperm, \blackperm ) $ is Type U relative to $ ( a, b ) $ then $ g^{\opsymbol} = g + 1 $.

\item If $ ( \whiteperm, \blackperm ) $ is Type N relative to $ ( a, b ) $ then $ g^{\opsymbol} = g - 1 $.

\item If $ ( \whiteperm, \blackperm ) $ is Type P relative to $ ( a, b ) $ then $ g^{\opsymbol} = g $.
\end{enumerate}
\end{genustheorem}

The reader will recall that if a permutation pair is not transitive then its synthetic genus does not quite have the ``expected'' topological meaning; see Example \ref{EXMsyntheticgenus}.

\begin{proof}
By Lemma \ref{LEMchardiff}, it suffices to show that $ \numcycles ( \whiteperm^{\opsymbol} \blackperm^{\opsymbol} ) - \numcycles ( \whiteperm \blackperm ) $ is equal to $ -2 $ if $ ( \whiteperm, \blackperm ) $ is Type U, is equal to $ 2 $ if $ ( \whiteperm, \blackperm ) $ is Type N, and is equal to $ 0 $ if $ ( \whiteperm, \blackperm ) $ is Type P. By Proposition \ref{PRPorbitcounting}, the desired difference is $ B^{\opsymbol} - B $ where $ B $ is the number of $ \whiteperm \blackperm $-orbits represented by $ a, \whiteperm ( a ), b $ and $ B^{\opsymbol} $ is the number of $ \whiteperm^{\opsymbol} \blackperm^{\opsymbol} $-orbits represented by $ a_{\white}, a_{\black}, b $.

Suppose $ ( \whiteperm, \blackperm ) $ is Type U. It is immediate from the definition of ``Type U'' that $ B = 3 $. Any two of Orbit Transfer Lemma \ref{LEMorbittransfer} (\ref{LorbittransferC}) (\ref{LorbittransferB}) (\ref{LorbittransferD}) imply $ B^{\opsymbol} = 1 $.

Suppose $ ( \whiteperm, \blackperm ) $ is Type N. It is immediate from the definition of ``Type N'' that $ B = 1 $. The definition of ``Type N'' also supplies the hypotheses of Orbit Transfer Lemma \ref{LEMorbittransfer} (\ref{LorbittransferP4N}) (\ref{LorbittransferP3N}), which then imply $ B^{\opsymbol} = 3 $.

Suppose $ ( \whiteperm, \blackperm ) $ is Type P1. It is immediate from the definition of ``Type P1'' that $ B = 1 $. The definition of ``Type P1'' also supplies the hypotheses of Orbit Transfer Lemma \ref{LEMorbittransfer} (\ref{LorbittransferP1P2}) (\ref{LorbittransferP1P3}), which then imply $ B^{\opsymbol} = 1 $.

Finally, suppose that $ ( \whiteperm, \blackperm ) $ is Type P2 or Type P3 or Type P4. In all three cases, $ B = 2 $.

If $ ( \whiteperm, \blackperm ) $ is Type P2 then Orbit Transfer Lemma \ref{LEMorbittransfer} (\ref{LorbittransferD}) implies that $ a_{\black} $ and $ b $ represent the same $ \whiteperm^{\opsymbol} \blackperm^{\opsymbol} $-orbit, while Orbit Transfer Lemma \ref{LEMorbittransfer} (\ref{LorbittransferP2N}) implies that $ a_{\white} $ represents a different $ \whiteperm^{\opsymbol} \blackperm^{\opsymbol} $-orbit.

If $ ( \whiteperm, \blackperm ) $ is Type P3 then Orbit Transfer Lemma \ref{LEMorbittransfer} (\ref{LorbittransferB}) implies that $ a_{\white} $ and $ b $ represent the same $ \whiteperm^{\opsymbol} \blackperm^{\opsymbol} $-orbit, while Orbit Transfer Lemma \ref{LEMorbittransfer} (\ref{LorbittransferP3N}) implies that $ a_{\black} $ represents a different $ \whiteperm^{\opsymbol} \blackperm^{\opsymbol} $-orbit.

If $ ( \whiteperm, \blackperm ) $ is Type P4 then Orbit Transfer Lemma \ref{LEMorbittransfer} (\ref{LorbittransferC}) implies that $ a_{\white} $ and $ a_{\black} $ represent the same $ \whiteperm^{\opsymbol} \blackperm^{\opsymbol} $-orbit, while Orbit Transfer Lemma \ref{LEMorbittransfer} (\ref{LorbittransferP4N}) implies that $ b $ represents a different $ \whiteperm^{\opsymbol} \blackperm^{\opsymbol} $-orbit. In all three cases, $ B^{\opsymbol} = 2 $.
\end{proof}

For fun, here are some examples:

\begin{exm}[``Theta''] \label{EXMthetatypeU}
In the picture here, a dessin d'enfant in $ \sphere $ is shown on the left, with edges labeled $ a $ and $ b $.
\begin{center}
\includegraphics[scale=0.4]{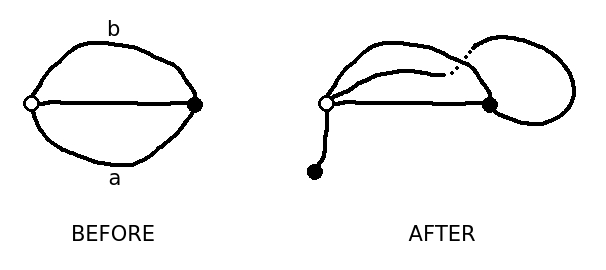}
\end{center}
Observe that the monodromy pair of this dessin d'enfant is Type U relative to $ ( a, b ) $. By applying the reroute $ \opsymbol $ relative to $ ( a, b ) $, one obtains a new monodromy pair, and a model for it is shown on the right. As predicted by the Reroute Theorem \ref{THMgenus}, the synthetic Euler characteristic of the model is $ ( 1 + 2 ) - 4 + 1 = 0 $, reflecting the fact that the ``true'' surface of the new dessin d'famille is $ \torus $.
\end{exm}

\begin{exm}[Tree \#1] \label{EXMtreetypeN}
In the picture here, a tree dessin d'enfant in $ \sphere $ is shown on the left, with edges labeled $ a $ and $ b $.
\begin{center}
\includegraphics[scale=0.4]{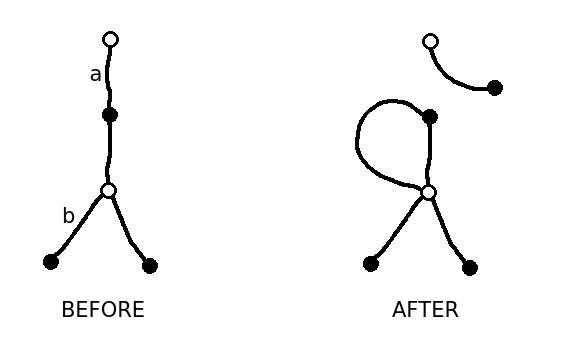}
\end{center}
Observe that the monodromy pair of this dessin d'enfant is Type N relative to $ ( a, b ) $. By applying the reroute $ \opsymbol $ relative to $ ( a, b ) $, one obtains a new monodromy pair, and a model for it is shown on the right. As predicted by the Reroute Theorem \ref{THMgenus}, the synthetic Euler characteristic of the model is $ ( 2 + 4 ) - 5 + 3 = 4 $, higher by $ 2 $ than the original Euler characteristic $ ( 2 + 3 ) - 4 + 1 = 2 $.
\end{exm}

\begin{exm}[Tree \#2] \label{EXMtreetypeP1}
In the picture here, the same tree is used as in the previous Example \ref{EXMtreetypeN}, but now a different edge $ a $ is chosen.
\begin{center}
\includegraphics[scale=0.4]{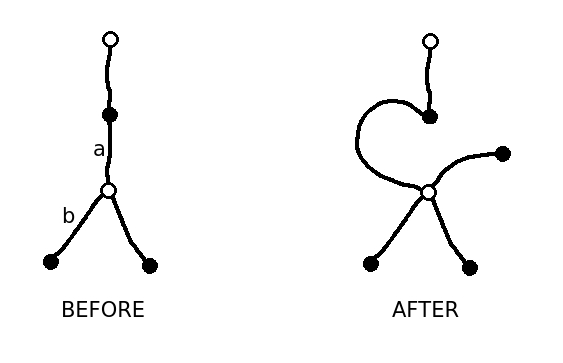}
\end{center}
Observe that the monodromy pair of this dessin d'enfant is Type P1 (since tree dessins d'enfants have only one face, the monodromy pair is always either Type N or Type P1 relative to any pair of edges) relative to $ ( a, b ) $. By applying the reroute $ \opsymbol $ relative to $ ( a, b ) $, one obtains a new monodromy pair, and a model for it is shown on the right. As predicted by the Reroute Theorem \ref{THMgenus}, the synthetic Euler characteristic of the model is $ ( 2 + 4 ) - 5 + 1 = 2 $, the same as the Euler characteristic of the original.
\end{exm}

\begin{rmk}
One can see the basic idea of Proposition \ref{PRPtreecase} in Examples \ref{EXMtreetypeN} and \ref{EXMtreetypeP1}: disconnection cannot occur without creating additional circuits.
\end{rmk}

\begin{exm}
In the picture here, a dessin d'enfant in $ \sphere $ is shown on the left, with edges labeled $ a $ and $ b $.
\begin{center}
\includegraphics[scale=0.4]{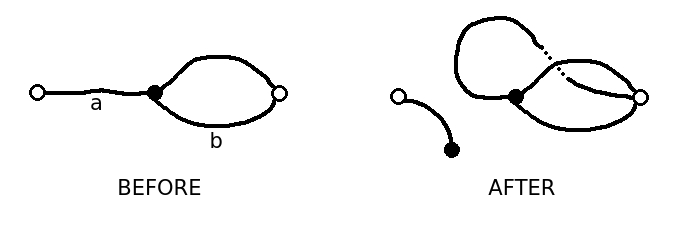}
\end{center}
Observe that the monodromy pair of this dessin d'enfant is Type P2 relative to $ ( a, b ) $. By applying the reroute $ \opsymbol $ relative to $ ( a, b ) $, one obtains a new monodromy pair, and a model for it is shown on the right. As predicted by the Reroute Theorem \ref{THMgenus}, the synthetic Euler characteristic of the model is $ ( 2 + 2 ) - 4 + 2 = 2 $, the same as the original Euler characteristic $ ( 2 + 1 ) - 3 + 2 = 2 $.
\end{exm}

\begin{exm}[Circuit \#1] \label{EXMcircuittypeP3}
In the picture here, a circuit dessin d'enfant in $ \sphere $ is shown on the left, with edges labeled $ a $ and $ b $.
\begin{center}
\includegraphics[scale=0.4]{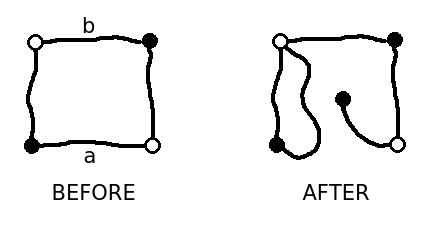}
\end{center}
Observe that the monodromy pair of this dessin d'enfant is Type P3 relative to $ ( a, b ) $. By applying the reroute $ \opsymbol $ relative to $ ( a, b ) $, one obtains a new monodromy pair, and a model for it is shown on the right. As predicted by the Reroute Theorem \ref{THMgenus}, the synthetic Euler characteristic of the model is $ ( 2 + 3 ) - 5 + 2 = 2 $, the same as the Euler characteristic of the original.
\end{exm}

\begin{exm}[Circuit \#2] \label{EXMcircuittypeP4}
In the picture here, the same circuit as in the previous Example \ref{EXMcircuittypeP3} is used, but now a different edge $ b $ is chosen.
\begin{center}
\includegraphics[scale=0.4]{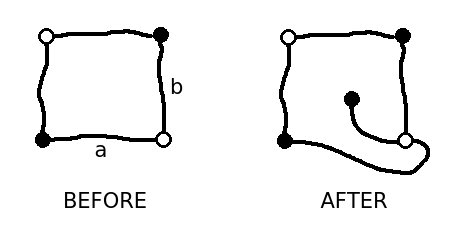}
\end{center}
Observe that the monodromy pair of this dessin d'enfant is Type P4 relative to $ ( a, b ) $. By applying the reroute $ \opsymbol $ relative to $ ( a, b ) $, one obtains a new monodromy pair, and a model for it is shown on the right. As predicted by the Reroute Theorem \ref{THMgenus}, the synthetic Euler characteristic of the model is $ ( 2 + 3 ) - 5 + 2 = 2 $, the same as the Euler characteristic of the original.
\end{exm}

\subsection{More about models}

As usual, let $ ( \whiteperm^{\opsymbol}, \blackperm^{\opsymbol} ) $ be the reroute of $ ( \whiteperm, \blackperm ) $ relative to $ ( a, b ) $.

It will be helpful in \S\ref{Sclassificationbytransitivity} to know a bit about models for $ ( \whiteperm^{\opsymbol}, \blackperm^{\opsymbol} ) $, which I record here:

\begin{lem}[Model Operation] \label{LEMmodelsforoperation}
Let $ ( G, X ) $ be a model for $ ( \whiteperm, \blackperm ) $. \Assertions For any model $ ( G^{\opsymbol}, X^{\opsymbol} ) $ of $ ( \whiteperm^{\opsymbol}, \blackperm^{\opsymbol} ) $, the underlying graph $ G^{\opsymbol} $ is obtained from the graph $ G \setminus a $ by introducing one new $ \black $-vertex $ v $ and introducing two new edges: edge $ a_{\white} $ between vertices $ \white_a $ and $ v $ and edge $ a_{\black} $ between $ \black_a $ and $ \white_b $. In particular, if $ G $ is connected then: $ G^{\opsymbol} $ is connected if and only if there is a walk in $ G^{\opsymbol} $ from $ \white_a $ to $ \black_a $.
\end{lem}

Of course, this is just a formalization of Figure \ref{FIGoperation}. Specific information related to embeddings could also be included, but models will only be used in \S\ref{Sclassificationbytransitivity} to argue about connectivity via the convenient language of walks, so Lemma \ref{LEMmodelsforoperation} need not concern itself with embeddings.

\begin{proof}
By the construction of dessins d'familles, Proposition \ref{PRPconstructionofmodels}, the $ \white $-vertices (resp. $ \black $-vertices) of the underlying graph correspond to $ \whiteperm $-orbits (resp. $ \blackperm $-orbits), edges are elements of $ E $, and an edge $ e \in E $ is incident to vertices $ x $ and $ y $ if and only if $ e $ is contained in both of the corresponding orbits.

I work directly from Definition \ref{DFNfundamentaloperation} of $ \opsymbol $. Let $ \vertices $ be the vertex set of $ G $ and $ \vertices^{\opsymbol} $ that of $ G^{\opsymbol} $. Definition \ref{DFNfundamentaloperation} already defines the edges $ E^{\opsymbol} $ and its simple relationship to $ E $.

First, define a natural injection $ \phi : \vertices \hookrightarrow \vertices^{\opsymbol} $. Define $ \phi $ on $ \white $-vertices of $ G $ as follows: if $ O $ is an $ \whiteperm $-orbit containing neither $ a $ nor $ b $ then $ O $ is also an $ \whiteperm^{\opsymbol} $-orbit and $ \phi ( O ) $ is defined to be $ O $ again, if $ O $ is the $ \whiteperm $-orbit containing $ a $ then $ \phi ( O ) $ is defined to be the $ \whiteperm^{\opsymbol} $-orbit containing $ a_{\white} $, and if $ O $ is the $ \whiteperm $-orbit containing $ b $ then $ \phi ( O ) $ is defined to be the $ \whiteperm^{\opsymbol} $-orbit containing $ a_{\black} $. \emph{Note that there is no conflict if $ a, b $ represent the same $ \whiteperm $-orbit.} Similarly, define $ \phi $ on $ \black $-vertices of $ G $ as follows: if $ O $ is an $ \blackperm $-orbit not containing $ a $ then $ O $ is also an $ \blackperm^{\opsymbol} $-orbit and $ \phi ( O ) $ is defined to be $ O $ again, and if $ O $ is the $ \blackperm $-orbit containing $ a $ then $ \phi ( O ) $ is defined to be the $ \blackperm^{\opsymbol} $-orbit containing $ a_{\black} $. Note, by Definition \ref{DFNfundamentaloperation}, that $ G^{\opsymbol} $ contains only one more vertex: the $ \black $-vertex of $ a_{\white} $, i.e. the vertex corresponding to the $ \blackperm^{\opsymbol} $-orbit containing $ a_{\white} $, a singleton.

Now, let $ e \subset G \setminus a $ be an edge and let $ O_{\white} $ and $ O_{\black} $ be the $ \whiteperm $-orbit and $ \blackperm $-orbit corresponding to its vertices $ \white_e $ and $ \black_e $. By Definition \ref{DFNfundamentaloperation} and the previous paragraph, $ e $ is also contained within the $ \whiteperm^{\opsymbol} $-orbit $ \phi ( O_{\white} ) $ and the $ \blackperm^{\opsymbol} $-orbit $ \phi ( O_{\black} ) $.

The previous two paragraphs show that, via $ \phi : \vertices \hookrightarrow \vertices^{\opsymbol} $ and $ E \setminus \{ a \} \hookrightarrow E^{\opsymbol} $, $ G \setminus a $ is a subgraph of $ G^{\opsymbol} $.

It is immediate from Definition \ref{DFNfundamentaloperation} that $ G^{\opsymbol} $ contains only one more vertex than $ G $, the vertex $ v $ in the statement of this Lemma: $ \vertices^{\opsymbol} \setminus \phi ( \vertices ) $ consists of the vertex corresponding to the $ \blackperm^{\opsymbol} $-orbit containing $ a_{\white} $, a singleton. It is also immediate from Definition \ref{DFNfundamentaloperation} that $ G^{\opsymbol} $ contains only two more edges than $ G \setminus a $: the edges $ a_{\white} $ and $ a_{\black} $. It is immediate from the Definition \ref{DFNfundamentaloperation} and the first paragraph of this proof that $ a_{\white} $ and $ a_{\black} $ are incident to vertices as described in the statement of this Lemma. This concludes the proof of the first statement.

Now, assume that $ G $ is connected. It is trivial that if $ G^{\opsymbol} $ is connected then there is such a walk. Conversely, suppose that there is such a walk. If $ G \setminus a $ is connected then this is obvious from what was already proved: $ G^{\opsymbol} $ is constructed from $ G \setminus a $ by attaching edges. So, I can assume that $ G \setminus a $ is disconnected. Necessarily, $ G \setminus a = G_{\white} \sqcup G_{\black} $, where $ G_{\white}, G_{\black} $ are connected and $ \white_a \in G_{\white} $, $ \black_a \in G_{\black} $. To prove that $ G^{\opsymbol} $ is connected, it is equivalent to prove that if $ x, y \in G^{\opsymbol} $ are vertices then there is a walk in $ G^{\opsymbol} $ from $ x $ to $ y $. Since $ G^{\opsymbol} $ contains only one new vertex, the $ \black $-vertex of edge $ a_{\white} $, I can assume that $ x, y $ are vertices of $ G $. Since $ G_{\white}, G_{\black} $ are connected, I can also assume that $ x \in G_{\white} $ and $ y \in G_{\black} $. But the claim is now obvious: concatenate walks from $ x $ to $ \white_a $ and from $ \black_a $ to $ y $ with the assumed walk from $ \white_a $ to $ \black_a $.
\end{proof}

\section{Iteration of the Operation} \label{Siteration}

\begin{center}
\emph{Throughout this section, $ E $ is a finite set and $ ( \whiteperm, \blackperm ) $ is an arbitrary pair in the symmetric group $ S_E $. Fix distinct $ a, b \in E $ and let $ ( \whiteperm^{\opsymbol}, \blackperm^{\opsymbol} ) $ be the reroute of $ ( \whiteperm, \blackperm ) $ relative to $ ( a, b ) $.}
\end{center}

\begin{dfn} \label{DFNdoublereroute}
$ ( \whiteperm^{\opsymbol \opsymbol}, \blackperm^{\opsymbol \opsymbol} ) $ is defined to be the reroute of $ ( \whiteperm^{\opsymbol}, \blackperm^{\opsymbol} ) $ relative to $ ( b, a_{\white} ) $. The natural set $ E^{\opsymbol \opsymbol} $ is, By Definition \ref{DFNfundamentaloperation}, $ E^{\opsymbol \opsymbol} \defeq \{ a_{\white}, a_{\black}, b_{\white}, b_{\black} \} \sqcup E \setminus \{ a, b \} $.
\end{dfn}

Recall that if $ \perm, s \in S_E $ then $ \perm^s \defeq s \cdot \perm \cdot s^{-1} $. As promised in \S\ref{SSdefofoperation}, the following is the relationship of the operation $ \opsymbol $ to conjugation:

\begin{prp} \label{PRPconjugationviaoperations}
Let $ t \in S_E $ be the transposition exchanging $ a $ and $ b $. Let $ ( \whiteperm^{\opsymbol \opsymbol}, \blackperm^{\opsymbol \opsymbol} ) $ be as in Definition \ref{DFNdoublereroute}. \Assertions $ ( \whiteperm^t, \blackperm ) $ is transitive if and only if $ ( \whiteperm^{\opsymbol \opsymbol}, \blackperm^{\opsymbol \opsymbol} ) $ is transitive and the synthetic genus of $ ( \whiteperm^t, \blackperm ) $ is the same as that of $ ( \whiteperm^{\opsymbol \opsymbol}, \blackperm^{\opsymbol \opsymbol} ) $.
\end{prp}

Graphically, this is fairly intuitive, but that argument is unrigorous.

\begin{proof}
By Definition \ref{DFNfundamentaloperation}, $ \whiteperm^{\opsymbol \opsymbol} $ is constructed from $ \whiteperm $ as follows: \emph{Insert the symbol $ a_{\black} $ immediately before $ b $, insert the symbol $ b_{\black} $ immediately before $ a $, replace the symbols $ a, b $ by $ a_{\white}, b_{\white} $.} Similarly, $ \blackperm^{\opsymbol \opsymbol} $ is constructed from $ \blackperm $ as follows: \emph{Replace $ a $ by the symbol $ a_{\black} $, replace $ b $ by the symbol $ b_{\black} $, introduce the trivial cycles $ ( a_{\white} ) $ and $ ( b_{\white} ) $.}

Now, let $ \phi : S_E \hookrightarrow S_{E^{\opsymbol \opsymbol}} $ be induced by the injection $ E \hookrightarrow E^{\opsymbol \opsymbol} $ defined by $ a \mapsto a_{\black} $, $ b \mapsto b_{\black} $, $ x \mapsto x $ for all $ x \in E \setminus \{ a, b \} $. It is immediate that $ \phi ( \blackperm ) = \blackperm^{\opsymbol \opsymbol} $. Set $ \whiteperm^{\prime \prime} \defeq \phi ( t \cdot \whiteperm \cdot t^{-1} ) $, and note that this is the same as the conjugate of $ \phi ( \whiteperm ) $ by the transposition $ ( a_{\black}, b_{\black} ) $. Set $ E^{\prime \prime} \defeq E^{\opsymbol \opsymbol} \setminus \{ a_{\white}, b_{\white} \} $. Clearly, $ ( \whiteperm^{\prime \prime}, \blackperm^{\opsymbol \opsymbol} ) $ permutes $ E^{\prime \prime} $ and $ ( \whiteperm^t, \blackperm ) $ is transitive if and only if $ ( \whiteperm^{\prime \prime}, \blackperm^{\opsymbol \opsymbol} ) $ is transitive on $ E^{\prime \prime} $. Thus, it is equivalent to show that $ ( \whiteperm^{\prime \prime}, \blackperm^{\opsymbol \opsymbol} ) $ is transitive if and only if $ ( \whiteperm^{\opsymbol \opsymbol}, \blackperm^{\opsymbol \opsymbol} ) $ is transitive.

It is immediate from the cycle structures of $ \whiteperm^{\opsymbol \opsymbol} $ and $ \whiteperm^{\prime \prime} $ that if $ O^{\prime \prime} \subset E^{\prime \prime} $ is an $ \whiteperm^{\prime \prime} $-orbit then there is a unique $ \whiteperm^{\opsymbol \opsymbol} $-orbit $ O^{\opsymbol \opsymbol} $ satisfying $ O^{\prime \prime} \subset O^{\opsymbol \opsymbol} \subset O^{\prime \prime} \cup \{ a_{\white}, b_{\white} \} $, and all $ \whiteperm^{\opsymbol \opsymbol} $-orbits are produced in this way. Since $ a_{\white}, b_{\white} $ are fixed points of $ \blackperm^{\opsymbol \opsymbol} $, it is clear that $ O^{\prime \prime} $ is $ \blackperm^{\opsymbol \opsymbol} $-stable if and only if $ O^{\opsymbol \opsymbol} $ is $ \blackperm^{\opsymbol \opsymbol} $-stable.

The first claim now follows from what is possibly the most trivial observation ever to appear in print: \emph{A pair $ ( A, B ) $ is transitive if and only if no union of $ A $-orbits is $ B $-stable except $ \emptyset $ and $ E $.}

The second claim is immediate from the information above: $ \vert E^{\opsymbol \opsymbol} \vert = \vert E \vert + 2 $, $ \numcycles ( \blackperm^{\opsymbol \opsymbol} ) = \numcycles ( \blackperm ) + 2 $, $ \numcycles ( \whiteperm^{\opsymbol \opsymbol} ) = \numcycles ( \whiteperm ) = \numcycles ( \whiteperm^t ) $, $ \numcycles ( \whiteperm^{\opsymbol \opsymbol} \blackperm^{\opsymbol \opsymbol} ) = \numcycles ( \whiteperm^t \blackperm ) $ \comment{EXPLAIN MORE LAST EQUALITY}.
\end{proof}

Therefore, I need to understand the Type of $ ( \whiteperm^{\opsymbol}, \blackperm^{\opsymbol} ) $ relative to $ ( b, a_{\white} ) $ in terms of the initial data $ ( \whiteperm, \blackperm ) $ and $ ( a, b ) $. I refer to this goal below as ``branching''. This obviously requires knowledge of $ \whiteperm^{\opsymbol} ( b ) $. The following simple observations, which are immediate from Definition \ref{DFNfundamentaloperation} of $ \whiteperm^{\opsymbol} $, will be useful:

\begin{itemize}
\item[-] If $ \whiteperm ( b ) = b $ then $ \whiteperm^{\opsymbol} ( b ) = a_{\black} $.

\item[-] If $ \whiteperm ( b ) = a $ then $ \whiteperm^{\opsymbol} ( b ) = a_{\white} $.

\item[-] If $ \whiteperm ( b ) \neq a, b $ then $ \whiteperm^{\opsymbol} ( b ) = \whiteperm ( b ) $.

\item[-] Since $ \whiteperm^{\opsymbol} ( a_{\black} ) = b $ always, $ \whiteperm^{\opsymbol} ( b ) \neq b $ always.
\end{itemize}

The reader will observe that situations $ \whiteperm ( b ) = b $ or $ \whiteperm ( b ) = a $ are, after exchanging $ a $ and $ b $, precisely those that required special treatment by the proofs in \S\ref{Sfundamentaloperation}. This is expected.

\subsection{Type U Branching}

By the proof of the Reroute Theorem \ref{THMgenus}, all of $ a_{\white}, a_{\black}, b $ represent the same $ \whiteperm^{\opsymbol} \blackperm^{\opsymbol} $-orbit, and the $ \whiteperm^{\opsymbol} \blackperm^{\opsymbol} $-cycle containing them is of the form $ ( \ldots a_{\black} \ldots a_{\white} \ldots b \ldots ) $. As usual, I check the position of $ \whiteperm^{\opsymbol} ( b ) $ relative to $ b, a_{\white} $.

If $ \whiteperm ( b ) = b $ then $ \whiteperm^{\opsymbol} ( b ) = a_{\black} $, so $ ( \whiteperm^{\opsymbol}, \blackperm^{\opsymbol} ) $ is Type P1. If $ \whiteperm ( b ) = a $ then $ \whiteperm^{\opsymbol} ( b ) = a_{\white} $, so $ ( \whiteperm^{\opsymbol}, \blackperm^{\opsymbol} ) $ is Type P1. So, I can assume that $ \whiteperm ( b ) \neq b, a $, in which case $ \whiteperm^{\opsymbol} ( b ) = \whiteperm ( b ) $. Note that $ \whiteperm ( b ) \neq \whiteperm ( a ) $ also, due to distinctness of $ a, b $.

If $ \whiteperm ( b ) $ represents the same $ \whiteperm \blackperm $-orbit as $ a $ then Orbit Transfer Lemma \ref{LEMorbittransfer} (\ref{LorbittransferC}) says that $ \whiteperm^{\opsymbol} ( b ) $ is contained in the $ \whiteperm^{\opsymbol} \blackperm^{\opsymbol} $-arc from $ a_{\black} $ to $ a_{\white} $. Therefore, $ ( \whiteperm^{\opsymbol}, \blackperm^{\opsymbol} ) $ is Type P1.

If $ \whiteperm ( b ) $ represents the same $ \whiteperm \blackperm $-orbit as $ \whiteperm ( a ) $ then Orbit Transfer Lemma \ref{LEMorbittransfer} (\ref{LorbittransferB}) says that $ \whiteperm^{\opsymbol} ( b ) $ is contained in the $ \whiteperm^{\opsymbol} \blackperm^{\opsymbol} $-arc from $ a_{\white} $ to $ b $. Therefore, $ ( \whiteperm^{\opsymbol}, \blackperm^{\opsymbol} ) $ is Type N.

If $ \whiteperm ( b ) $ represents the same $ \whiteperm \blackperm $-orbit as $ b $ then Orbit Transfer Lemma \ref{LEMorbittransfer} (\ref{LorbittransferD}) says that $ \whiteperm^{\opsymbol} ( b ) $ is contained in the $ \whiteperm^{\opsymbol} \blackperm^{\opsymbol} $-arc from $ b $ to $ a_{\black} $. Therefore, $ ( \whiteperm^{\opsymbol}, \blackperm^{\opsymbol} ) $ is Type P1.

If $ \whiteperm ( b ) $ represents a different $ \whiteperm \blackperm $-orbit than all of $ a, \whiteperm ( a ), b $ then it is clear from Lemma \ref{LEMunbiasedorbitsalso} that $ \whiteperm^{\opsymbol} ( b ) $ represents a different $ \whiteperm^{\opsymbol} \blackperm^{\opsymbol} $-orbit than all of $ a_{\white}, a_{\black}, b $, so $ ( \whiteperm^{\opsymbol}, \blackperm^{\opsymbol} ) $ is Type P3.

These facts can be summarized:

\begin{enumerate}
\setlength{\itemsep}{5pt}

\item \label{TYPEU1} If $ \whiteperm ( b ) $ represents the same $ \whiteperm \blackperm $-orbit as $ a $ then $ ( \whiteperm^{\opsymbol}, \blackperm^{\opsymbol} ) $ is Type P1 relative to $ ( b, a_{\white} ) $.

\item \label{TYPEU2} If $ \whiteperm ( b ) $ represents the same $ \whiteperm \blackperm $-orbit as $ \whiteperm ( a ) $ then $ ( \whiteperm^{\opsymbol}, \blackperm^{\opsymbol} ) $ is Type N relative to $ ( b, a_{\white} ) $.

\item \label{TYPEU3} If $ \whiteperm ( b ) $ represents the same $ \whiteperm \blackperm $-orbit as $ b $ then $ ( \whiteperm^{\opsymbol}, \blackperm^{\opsymbol} ) $ is Type P1 relative to $ ( b, a_{\white} ) $.

\item \label{TYPEU4} If $ \whiteperm ( b ) $ represents a different $ \whiteperm \blackperm $-orbit than all of $ a, \whiteperm ( a ), b $ then $ ( \whiteperm^{\opsymbol}, \blackperm^{\opsymbol} ) $ is Type P3 relative to $ ( b, a_{\white} ) $.
\end{enumerate}

\subsection{Type N Branching}

By the proof of the Reroute Theorem \ref{THMgenus}, no two of $ a_{\white}, a_{\black}, b $ represent the same $ \whiteperm^{\opsymbol} \blackperm^{\opsymbol} $-orbit. As usual, I check the position of $ \whiteperm^{\opsymbol} ( b ) $ relative to $ b, a_{\white} $.

If $ \whiteperm ( b ) = b $ then $ \whiteperm^{\opsymbol} ( b ) = a_{\black} $, so $ ( \whiteperm^{\opsymbol}, \blackperm^{\opsymbol} ) $ is Type U. If $ \whiteperm ( b ) = a $ then $ \whiteperm^{\opsymbol} ( b ) = a_{\white} $, so $ ( \whiteperm^{\opsymbol}, \blackperm^{\opsymbol} ) $ is Type P4. So, I can assume that $ \whiteperm ( b ) \neq b, a $, in which case $ \whiteperm^{\opsymbol} ( b ) = \whiteperm ( b ) $. Note that $ \whiteperm ( b ) \neq \whiteperm ( a ) $ also, due to distinctness of $ a, b $.

If $ \whiteperm ( b ) $ is contained in the $ \whiteperm \blackperm $-arc from $ a $ to $ b $ then Orbit Transfer Lemma \ref{LEMorbittransfer} (\ref{LorbittransferP3N}) says that $ \whiteperm^{\opsymbol} ( b ) $ represents the same $ \whiteperm^{\opsymbol} \blackperm^{\opsymbol} $-orbit as $ a_{\black} $. Therefore, $ ( \whiteperm^{\opsymbol}, \blackperm^{\opsymbol} ) $ is Type U.

If $ \whiteperm ( b ) $ is contained in the $ \whiteperm \blackperm $-arc from $ b $ to $ \whiteperm ( a ) $ then Orbit Transfer Lemma \ref{LEMorbittransfer} (\ref{LorbittransferP4N}) says that $ \whiteperm^{\opsymbol} ( b ) $ represents the same $ \whiteperm^{\opsymbol} \blackperm^{\opsymbol} $-orbit as $ b $. Therefore, $ ( \whiteperm^{\opsymbol}, \blackperm^{\opsymbol} ) $ is Type P2.

If $ \whiteperm ( b ) $ is contained in the $ \whiteperm \blackperm $-arc from $ \whiteperm ( a ) $ to $ a $ then Orbit Transfer Lemma \ref{LEMorbittransfer} (\ref{LorbittransferP2N}) says that $ \whiteperm^{\opsymbol} ( b ) $ represents the same $ \whiteperm^{\opsymbol} \blackperm^{\opsymbol} $-orbit as $ a_{\white} $. Therefore, so $ ( \whiteperm^{\opsymbol}, \blackperm^{\opsymbol} ) $ is Type P4.

If $ \whiteperm ( b ) $ represents a different $ \whiteperm \blackperm $-orbit than all of $ a, \whiteperm ( a ), b $ then it is clear from Lemma \ref{LEMunbiasedorbitsalso} that $ \whiteperm^{\opsymbol} ( b ) $ represents a different $ \whiteperm^{\opsymbol} \blackperm^{\opsymbol} $-orbit than all of $ a_{\white}, a_{\black}, b $, so $ ( \whiteperm^{\opsymbol}, \blackperm^{\opsymbol} ) $ is Type U.

These facts can be summarized:

\begin{enumerate}
\setlength{\itemsep}{5pt}
\item \label{TYPEN1} If $ \whiteperm ( b ) $ is contained in the $ \whiteperm \blackperm $-arc from $ a $ to $ b $ then $ ( \whiteperm^{\opsymbol}, \blackperm^{\opsymbol} ) $ is Type U relative to $ ( b, a_{\white} ) $.

\item \label{tameexceptionalN} If $ \whiteperm ( b ) $ is contained in the $ \whiteperm \blackperm $-arc from $ b $ to $ \whiteperm ( a ) $ then $ ( \whiteperm^{\opsymbol}, \blackperm^{\opsymbol} ) $ is Type P2 relative to $ ( b, a_{\white} ) $.

\item \label{TYPEN3} If $ \whiteperm ( b ) $ is contained in the $ \whiteperm \blackperm $-arc from $ \whiteperm ( a ) $ to $ a $ then $ ( \whiteperm^{\opsymbol}, \blackperm^{\opsymbol} ) $ is Type P4 relative to $ ( b, a_{\white} ) $.

\item \label{TYPEN4} If $ \whiteperm ( b ) $ represents a different $ \whiteperm \blackperm $-orbit than all of $ a, \whiteperm ( a ), b $ then $ ( \whiteperm^{\opsymbol}, \blackperm^{\opsymbol} ) $ is Type U relative to $ ( b, a_{\white} ) $.
\end{enumerate}

For reasons that will be explained in \S\ref{Sclassificationbytransitivity}, one situation must be separated:

\begin{dfn}[Tame Exceptional \#1A] \label{DFNtameexceptional1A}
$ ( \whiteperm, \blackperm ) $ is \emph{Tame Exceptional} if it is Type N relative to $ ( a, b ) $ and situation (\ref{tameexceptionalN}) occurs.
\end{dfn}

This is the first of three ``tame exceptional'' cases that need to be separated. A more concrete description of this case is: \emph{$ a, \whiteperm ( a ), b, \whiteperm ( b ) $ represent the same $ \whiteperm \blackperm $-orbit, their $ \whiteperm \blackperm $-cycle is of the form $ ( \ldots \whiteperm ( b ) \ldots \whiteperm ( a ) \ldots a \ldots b \ldots ) $, and $ \whiteperm ( b ) \neq b $ (it is allowed that $ \whiteperm ( a ) = a $).}

\subsection{Type P1 Branching}

By the proof of the Reroute Theorem \ref{THMgenus}, all of $ a_{\white}, a_{\black}, b $ represent the same $ \whiteperm^{\opsymbol} \blackperm^{\opsymbol} $-orbit, and the $ \whiteperm^{\opsymbol} \blackperm^{\opsymbol} $-cycle containing them is of the form $ ( \ldots a_{\white} \ldots a_{\black} \ldots b \ldots ) $. As usual, I check the position of $ \whiteperm^{\opsymbol} ( b ) $ relative to $ b, a_{\white} $.

If $ \whiteperm ( b ) = b $ then $ \whiteperm^{\opsymbol} ( b ) = a_{\black} $, so $ ( \whiteperm^{\opsymbol}, \blackperm^{\opsymbol} ) $ is Type N. If $ \whiteperm ( b ) = a $ then $ \whiteperm^{\opsymbol} ( b ) = a_{\white} $, so $ ( \whiteperm^{\opsymbol}, \blackperm^{\opsymbol} ) $ is Type P1. So, I can assume that $ \whiteperm ( b ) \neq b, a $, in which case $ \whiteperm^{\opsymbol} ( b ) = \whiteperm ( b ) $. Note that $ \whiteperm ( b ) \neq \whiteperm ( a ) $ also, due to distinctness of $ a, b $.

If $ \whiteperm ( b ) $ is contained in the $ \whiteperm \blackperm $-arc from $ a $ to $ \whiteperm ( a ) $ then Orbit Transfer Lemma \ref{LEMorbittransfer} (\ref{LorbittransferP1P2}) says that $ \whiteperm^{\opsymbol} ( b ) $ is contained in the $ \whiteperm^{\opsymbol} \blackperm^{\opsymbol} $-arc from $ a_{\black} $ to $ b $. Therefore, $ ( \whiteperm^{\opsymbol}, \blackperm^{\opsymbol} ) $ is Type N.

If $ \whiteperm ( b ) $ is contained in the $ \whiteperm \blackperm $-arc from $ \whiteperm ( a ) $ to $ b $ then Orbit Transfer Lemma \ref{LEMorbittransfer} (\ref{LorbittransferP1P4}) says that $ \whiteperm^{\opsymbol} ( b ) $ is contained in the $ \whiteperm^{\opsymbol} \blackperm^{\opsymbol} $-arc from $ a_{\white} $ to $ a_{\black} $. Therefore, $ ( \whiteperm^{\opsymbol}, \blackperm^{\opsymbol} ) $ is Type N.

If $ \whiteperm ( b ) $ is contained in the $ \whiteperm \blackperm $-arc from $ b $ to $ a $ then Orbit Transfer Lemma \ref{LEMorbittransfer} (\ref{LorbittransferP1P3}) says that $ \whiteperm^{\opsymbol} ( b ) $ is contained in the $ \whiteperm^{\opsymbol} \blackperm^{\opsymbol} $-arc from $ b $ to $ a_{\white} $. Therefore, $ ( \whiteperm^{\opsymbol}, \blackperm^{\opsymbol} ) $ is Type P1.

If $ \whiteperm ( b ) $ represents a different $ \whiteperm \blackperm $-orbit than all of $ a, \whiteperm ( a ), b $ then it is clear from Lemma \ref{LEMunbiasedorbitsalso} that $ \whiteperm^{\opsymbol} ( b ) $ represents a different $ \whiteperm^{\opsymbol} \blackperm^{\opsymbol} $-orbit than all of $ a_{\white}, a_{\black}, b $, so $ ( \whiteperm^{\opsymbol}, \blackperm^{\opsymbol} ) $ is Type P3.

These facts can be summarized:

\begin{enumerate}
\setlength{\itemsep}{5pt}
\item \label{neverspherealwaystrans} If $ \whiteperm ( b ) $ is contained in the $ \whiteperm \blackperm $-arc from $ a $ to $ \whiteperm ( a ) $ then $ ( \whiteperm^{\opsymbol}, \blackperm^{\opsymbol} ) $ is Type N relative to $ ( b, a_{\white} ) $.

\item \label{tameexceptionalP1} If $ \whiteperm ( b ) $ is contained in the $ \whiteperm \blackperm $-arc from $ \whiteperm ( a ) $ to $ b $ then $ ( \whiteperm^{\opsymbol}, \blackperm^{\opsymbol} ) $ is Type N relative to $ ( b, a_{\white} ) $.

\item \label{TYPEP13} If $ \whiteperm ( b ) $ is contained in the $ \whiteperm \blackperm $-arc from $ b $ to $ a $ then $ ( \whiteperm^{\opsymbol}, \blackperm^{\opsymbol} ) $ is Type P1 relative to $ ( b, a_{\white} ) $.

\item \label{TYPEP14} If $ \whiteperm ( b ) $ represents a different $ \whiteperm \blackperm $-orbit than all of $ a, \whiteperm ( a ), b $ then $ ( \whiteperm^{\opsymbol}, \blackperm^{\opsymbol} ) $ is Type P3 relative to $ ( b, a_{\white} ) $.
\end{enumerate}

It is tempting to merge (\ref{neverspherealwaystrans}) and (\ref{tameexceptionalP1}), but this should not be done since they are fundamentally different. In fact, (\ref{tameexceptionalP1}) is the second ``tame exceptional'' case that must be separated:

\begin{dfn}[Tame Exceptional \#1B] \label{DFNtameexceptional1B}
$ ( \whiteperm, \blackperm ) $ is \emph{Tame Exceptional} if it is Type P1 relative to $ ( a, b ) $ and situation (\ref{tameexceptionalP1}) occurs.
\end{dfn}

A more concrete description of this case is: \emph{$ a, \whiteperm ( a ), b, \whiteperm ( b ) $ represent the same $ \whiteperm \blackperm $-orbit, their $ \whiteperm \blackperm $-cycle is of the form $ ( \ldots \whiteperm ( a ) \ldots \whiteperm ( b ) \ldots b \ldots a \ldots ) $, and $ \whiteperm ( a ) \neq a $ (it is allowed that $ \whiteperm ( b ) = b $).} Note that the two Tame Exceptional cases known so far are exchanged when $ a $ and $ b $ are exchanged.

\subsection{Type P2 Branching}

By the proof of the Reroute Theorem \ref{THMgenus}, $ a_{\black} $ and $ b $ represent the same $ \whiteperm^{\opsymbol} \blackperm^{\opsymbol} $-orbit and $ a_{\white} $ represents a different $ \whiteperm^{\opsymbol} \blackperm^{\opsymbol} $-orbit. As usual, I check the position of $ \whiteperm^{\opsymbol} ( b ) $ relative to $ b, a_{\white} $.

If $ \whiteperm ( b ) = b $ then $ \whiteperm^{\opsymbol} ( b ) = a_{\black} $, so $ ( \whiteperm^{\opsymbol}, \blackperm^{\opsymbol} ) $ is Type P2. If $ \whiteperm ( b ) = a $ then $ \whiteperm^{\opsymbol} ( b ) = a_{\white} $, so $ ( \whiteperm^{\opsymbol}, \blackperm^{\opsymbol} ) $ is Type P4. So, I can assume that $ \whiteperm ( b ) \neq b, a $, in which case $ \whiteperm^{\opsymbol} ( b ) = \whiteperm ( b ) $. Note that $ \whiteperm ( b ) \neq \whiteperm ( a ) $ also, due to distinctness of $ a, b $.

If $ \whiteperm ( b ) $ represents the same $ \whiteperm \blackperm $-orbit as $ b $ then Orbit Transfer Lemma \ref{LEMorbittransfer} (\ref{LorbittransferD}) says that $ \whiteperm^{\opsymbol} ( b ) $ is contained in the $ \whiteperm^{\opsymbol} \blackperm^{\opsymbol} $-arc from $ b $ to $ a_{\black} $. Therefore, $ ( \whiteperm^{\opsymbol}, \blackperm^{\opsymbol} ) $ is Type P2.

If $ \whiteperm ( b ) $ represents the same $ \whiteperm \blackperm $-orbit as $ a $ and $ \whiteperm ( b ) $ is contained in the $ \whiteperm \blackperm $-arc from $ \whiteperm ( a ) $ to $ a $ then Orbit Transfer Lemma \ref{LEMorbittransfer} (\ref{LorbittransferP2N}) says that $ \whiteperm^{\opsymbol} ( b ) $ is contained in the $ \whiteperm^{\opsymbol} \blackperm^{\opsymbol} $-orbit of $ a_{\white} $. Therefore, $ ( \whiteperm^{\opsymbol}, \blackperm^{\opsymbol} ) $ is Type P4.

If $ \whiteperm ( b ) $ represents the same $ \whiteperm \blackperm $-orbit as $ a $ but $ \whiteperm ( b ) $ is not contained in the $ \whiteperm \blackperm $-arc from $ \whiteperm ( a ) $ to $ a $ then Orbit Transfer Lemma \ref{LEMorbittransfer} (\ref{LorbittransferP1P2}) says that $ \whiteperm^{\opsymbol} ( b ) $ is contained in the $ \whiteperm^{\opsymbol} \blackperm^{\opsymbol} $-arc from $ a_{\black} $ to $ b $. Therefore, $ ( \whiteperm^{\opsymbol}, \blackperm^{\opsymbol} ) $ is Type P2.

If $ \whiteperm ( b ) $ represents a different $ \whiteperm \blackperm $-orbit than all of $ a, \whiteperm ( a ), b $ then it is clear from Lemma \ref{LEMunbiasedorbitsalso} that $ \whiteperm^{\opsymbol} ( b ) $ represents a different $ \whiteperm^{\opsymbol} \blackperm^{\opsymbol} $-orbit than all of $ a_{\white}, a_{\black}, b $. Therefore, $ ( \whiteperm^{\opsymbol}, \blackperm^{\opsymbol} ) $ is Type U.

These facts can be summarized:

\begin{enumerate}
\setlength{\itemsep}{5pt}
\item \label{TYPEP21} If $ \whiteperm ( b ) $ represents the same $ \whiteperm \blackperm $-orbit as $ a $ and $ \whiteperm ( b ) $ is contained in the $ \whiteperm \blackperm $-arc from $ \whiteperm ( a ) $ to $ a $ then $ ( \whiteperm^{\opsymbol}, \blackperm^{\opsymbol} ) $ is Type P4 relative to $ ( b, a_{\white} ) $.

\item \label{TYPEP22} If $ \whiteperm ( b ) $ represents the same $ \whiteperm \blackperm $-orbit as $ a $ but $ \whiteperm ( b ) $ is not contained in the $ \whiteperm \blackperm $-arc from $ \whiteperm ( a ) $ to $ a $ then $ ( \whiteperm^{\opsymbol}, \blackperm^{\opsymbol} ) $ is Type P2 relative to $ ( b, a_{\white} ) $.

\item \label{wildexceptional} If $ \whiteperm ( b ) $ represents the same $ \whiteperm \blackperm $-orbit as $ b $ then $ ( \whiteperm^{\opsymbol}, \blackperm^{\opsymbol} ) $ is Type P2 relative to $ ( b, a_{\white} ) $.

\item \label{TYPEP24} If $ \whiteperm ( b ) $ represents a different $ \whiteperm \blackperm $-orbit than all of $ a, \whiteperm ( a ), b $ then $ ( \whiteperm^{\opsymbol}, \blackperm^{\opsymbol} ) $ is Type U relative to $ ( b, a_{\white} ) $.
\end{enumerate}

For reasons that will be explained in \S\ref{Sclassificationbytransitivity}, one situation must be separated:

\begin{dfn}[Wild Exceptional] \label{DFNwildexceptional}
$ ( \whiteperm, \blackperm ) $ is \emph{Wild Exceptional} iff it is Type P2 relative to $ ( a, b ) $ and situation (\ref{wildexceptional}) occurs. In other words, iff $ ( \whiteperm, \blackperm ) $ is Type P2 relative to both $ ( a, b ) $ and $ ( b, a ) $.
\end{dfn}

A more concrete description of this case is: $ a $ and $ \whiteperm ( a ) $ represent the same $ \whiteperm \blackperm $-orbit, $ b $ and $ \whiteperm ( b ) $ represent the same $ \whiteperm \blackperm $-orbit, and the two orbits are different (it is allowed that $ \whiteperm ( a ) = a $ and/or $ \whiteperm ( b ) = b $).

\subsection{Type P3 Branching}

By the proof of the Reroute Theorem \ref{THMgenus}, $ a_{\white} $ and $ b $ represent the same $ \whiteperm^{\opsymbol} \blackperm^{\opsymbol} $-orbit and $ a_{\black} $ represents a different $ \whiteperm^{\opsymbol} \blackperm^{\opsymbol} $-orbit. As usual, I check the position of $ \whiteperm^{\opsymbol} ( b ) $ relative to $ b, a_{\white} $.

If $ \whiteperm ( b ) = b $ then $ \whiteperm^{\opsymbol} ( b ) = a_{\black} $, so $ ( \whiteperm^{\opsymbol}, \blackperm^{\opsymbol} ) $ is Type P3. If $ \whiteperm ( b ) = a $ then $ \whiteperm^{\opsymbol} ( b ) = a_{\white} $, so $ ( \whiteperm^{\opsymbol}, \blackperm^{\opsymbol} ) $ is Type P1. So, I can assume that $ \whiteperm ( b ) \neq b, a $, in which case $ \whiteperm^{\opsymbol} ( b ) = \whiteperm ( b ) $. Note that $ \whiteperm ( b ) \neq \whiteperm ( a ) $ also, due to distinctness of $ a, b $.

If $ \whiteperm ( b ) $ is contained in the $ \whiteperm \blackperm $-arc from $ a $ to $ b $ then Orbit Transfer Lemma \ref{LEMorbittransfer} (\ref{LorbittransferP3N}) says that $ \whiteperm^{\opsymbol} ( b ) $ represents the same $ \whiteperm^{\opsymbol} \blackperm^{\opsymbol} $-orbit as $ a_{\black} $. Therefore, $ ( \whiteperm^{\opsymbol}, \blackperm^{\opsymbol} ) $ is Type P3.

If $ \whiteperm ( b ) $ is contained in the $ \whiteperm \blackperm $-arc from $ b $ to $ a $ then Orbit Transfer Lemma \ref{LEMorbittransfer} (\ref{LorbittransferP1P3}) says that $ \whiteperm^{\opsymbol} ( b ) $ is contained in the $ \whiteperm^{\opsymbol} \blackperm^{\opsymbol} $-arc from $ b $ to $ a_{\white} $. Therefore, $ ( \whiteperm^{\opsymbol}, \blackperm^{\opsymbol} ) $ is Type P1.

If $ \whiteperm ( b ) $ represents the same $ \whiteperm \blackperm $-orbit as $ \whiteperm ( a ) $ then Orbit Transfer Lemma \ref{LEMorbittransfer} (\ref{LorbittransferB}) says that $ \whiteperm^{\opsymbol} ( b ) $ is contained in the $ \whiteperm^{\opsymbol} \blackperm^{\opsymbol} $-arc from $ a_{\white} $ to $ b $. Therefore, $ ( \whiteperm^{\opsymbol}, \blackperm^{\opsymbol} ) $ is Type N.

If $ \whiteperm ( b ) $ represents a different $ \whiteperm \blackperm $-orbit than all of $ a, \whiteperm ( a ), b $ then it is clear from Lemma \ref{LEMunbiasedorbitsalso} that $ \whiteperm^{\opsymbol} ( b ) $ represents a different $ \whiteperm^{\opsymbol} \blackperm^{\opsymbol} $-orbit than all of $ a_{\white}, a_{\black}, b $. Therefore, $ ( \whiteperm^{\opsymbol}, \blackperm^{\opsymbol} ) $ is Type P3.

These facts can be summarized:

\begin{enumerate}
\setlength{\itemsep}{5pt}

\item \label{TYPEP31} If $ \whiteperm ( b ) $ is contained in the $ \whiteperm \blackperm $-arc from $ a $ to $ b $ then $ ( \whiteperm^{\opsymbol}, \blackperm^{\opsymbol} ) $ is Type P3 relative to $ ( b, a_{\white} ) $.

\item \label{TYPEP32} If $ \whiteperm ( b ) $ is contained in the $ \whiteperm \blackperm $-arc from $ b $ to $ a $ then $ ( \whiteperm^{\opsymbol}, \blackperm^{\opsymbol} ) $ is Type P1 relative to $ ( b, a_{\white} ) $.

\item \label{tameexceptionalP3} If $ \whiteperm ( b ) $ represents the same $ \whiteperm \blackperm $-orbit as $ \whiteperm ( a ) $ then $ ( \whiteperm^{\opsymbol}, \blackperm^{\opsymbol} ) $ is Type N relative to $ ( b, a_{\white} ) $.

\item \label{TYPEP34} If $ \whiteperm ( b ) $ represents a different $ \whiteperm \blackperm $-orbit than all of $ a, \whiteperm ( a ), b $ then $ ( \whiteperm^{\opsymbol}, \blackperm^{\opsymbol} ) $ is Type P3 relative to $ ( b, a_{\white} ) $.
\end{enumerate}

The third and last ``tame exceptional'' situation that needs to be separated is:

\begin{dfn}[Tame Exceptional \#2] \label{DFNtameexceptional2}
$ ( \whiteperm, \blackperm ) $ is \emph{Tame Exceptional} if it is Type P3 relative to $ ( a, b ) $ and situation (\ref{tameexceptionalP3}) occurs.
\end{dfn}

A more concrete description of this case is: \emph{$ a, b $ represent the same $ \whiteperm \blackperm $-orbit, $ \whiteperm ( a ), \whiteperm ( b ) $ represent the same $ \whiteperm \blackperm $-orbit, and the two orbits are different.}

\subsection{Type P4 Branching}

By the proof of the Reroute Theorem \ref{THMgenus}, $ a_{\white} $ and $ a_{\black} $ represent the same $ \whiteperm^{\opsymbol} \blackperm^{\opsymbol} $-orbit and $ b $ represents a different $ \whiteperm^{\opsymbol} \blackperm^{\opsymbol} $-orbit. As usual, I check the position of $ \whiteperm^{\opsymbol} ( b ) $ relative to $ ( b, a_{\white} ) $.

If $ \whiteperm ( b ) = b $ then $ \whiteperm^{\opsymbol} ( b ) = a_{\black} $, so $ ( \whiteperm^{\opsymbol}, \blackperm^{\opsymbol} ) $ is Type P4. If $ \whiteperm ( b ) = a $ then $ \whiteperm^{\opsymbol} ( b ) = a_{\white} $, so $ ( \whiteperm^{\opsymbol}, \blackperm^{\opsymbol} ) $ is Type P4. So, I can assume that $ \whiteperm ( b ) \neq b, a $, in which case $ \whiteperm^{\opsymbol} ( b ) = \whiteperm ( b ) $. Note that $ \whiteperm ( b ) \neq \whiteperm ( a ) $ also, due to distinctness of $ a, b $.

If $ \whiteperm ( b ) $ represents the same $ \whiteperm \blackperm $-orbit as $ a $ then Orbit Transfer Lemma \ref{LEMorbittransfer} (\ref{LorbittransferC}) says that $ \whiteperm^{\opsymbol} ( b ) $ is contained in the $ \whiteperm^{\opsymbol} \blackperm^{\opsymbol} $-arc from $ a_{\black} $ to $ a_{\white} $. Therefore, $ ( \whiteperm^{\opsymbol}, \blackperm^{\opsymbol} ) $ is Type P4.

If $ \whiteperm ( b ) $ represents the same $ \whiteperm \blackperm $-orbit as $ b $ and is contained in the $ \whiteperm \blackperm $-arc from $ \whiteperm ( a ) $ to $ b $ then Orbit Transfer Lemma \ref{LEMorbittransfer} (\ref{LorbittransferP1P4}) says that $ \whiteperm^{\opsymbol} ( b ) $ is contained in the $ \whiteperm^{\opsymbol} \blackperm^{\opsymbol} $-arc from $ a_{\white} $ to $ a_{\black} $. Therefore, $ ( \whiteperm^{\opsymbol}, \blackperm^{\opsymbol} ) $ is Type P4.

If $ \whiteperm ( b ) $ represents the same $ \whiteperm \blackperm $-orbit as $ b $ and is not contained in the $ \whiteperm \blackperm $-arc from $ \whiteperm ( a ) $ to $ b $ then Orbit Transfer Lemma \ref{LEMorbittransfer} (\ref{LorbittransferP4N}) says that $ \whiteperm^{\opsymbol} ( b ) $ is contained in the $ \whiteperm^{\opsymbol} \blackperm^{\opsymbol} $-orbit of $ b $. Therefore, $ ( \whiteperm^{\opsymbol}, \blackperm^{\opsymbol} ) $ is Type P2.

If $ \whiteperm ( b ) $ represents a different $ \whiteperm \blackperm $-orbit than all of $ a, \whiteperm ( a ), b $ then it is clear from Lemma \ref{LEMunbiasedorbitsalso} that $ \whiteperm^{\opsymbol} ( b ) $ represents a different $ \whiteperm^{\opsymbol} \blackperm^{\opsymbol} $-orbit than all of $ a_{\white}, a_{\black}, b $. Therefore, $ ( \whiteperm^{\opsymbol}, \blackperm^{\opsymbol} ) $ is Type U.

These facts can be summarized:

\begin{enumerate}
\setlength{\itemsep}{5pt}

\item \label{TYPEP41} If $ \whiteperm ( b ) $ represents the same $ \whiteperm \blackperm $-orbit as $ a $ then $ ( \whiteperm^{\opsymbol}, \blackperm^{\opsymbol} ) $ is Type P4 relative to $ b, a_{\white} $.

\item \label{TYPEP43} If $ \whiteperm ( b ) $ represents the same $ \whiteperm \blackperm $-orbit as $ b $ and is contained in the $ \whiteperm \blackperm $-arc from $ \whiteperm ( a ) $ to $ b $ then $ ( \whiteperm^{\opsymbol}, \blackperm^{\opsymbol} ) $ is Type P4 relative to $ ( b, a_{\white} ) $.

\item \label{TYPEP42} If $ \whiteperm ( b ) $ represents the same $ \whiteperm \blackperm $-orbit as $ b $ and is \emph{not} contained in the $ \whiteperm \blackperm $-arc from $ \whiteperm ( a ) $ to $ b $ then $ ( \whiteperm^{\opsymbol}, \blackperm^{\opsymbol} ) $ is Type P2 relative to $ ( b, a_{\white} ) $.
 
\item \label{TYPEP44} If $ \whiteperm ( b ) $ represents a different $ \whiteperm \blackperm $-orbit than all of $ a, \whiteperm ( a ), b $ then $ ( \whiteperm^{\opsymbol}, \blackperm^{\opsymbol} ) $ is Type U relative to $ ( b, a_{\white} ) $.
\end{enumerate}

\section{Classification 1: by Transitivity} \label{Sclassificationbytransitivity}

\begin{center}
\emph{Throughout this section, $ ( D, X ) $ is a dessin d'enfant with edges $ \edges $ and monodromy pair $ ( \whiteperm, \blackperm ) $. Fix distinct $ a, b \in \edges $. Let $ ( \whiteperm^{\opsymbol}, \blackperm^{\opsymbol} ) $ be the reroute of $ ( \whiteperm, \blackperm ) $ relative to $ ( a, b ) $ and $ ( \whiteperm^{\opsymbol \opsymbol}, \blackperm^{\opsymbol \opsymbol} ) $ as in Definition \ref{DFNdoublereroute}. Finally, $ ( D^{\opsymbol}, X^{\opsymbol} ) $ is a nondegenerate model for $ ( \whiteperm^{\opsymbol}, \blackperm^{\opsymbol} ) $ and $ ( D^{\opsymbol \opsymbol}, X^{\opsymbol \opsymbol} ) $ is a nondegenerate model for $ ( \whiteperm^{\opsymbol \opsymbol}, \blackperm^{\opsymbol \opsymbol} ) $.}
\end{center}

Recall that if $ \perm, s \in S_{\edges} $ then $ \perm^s \defeq s \cdot \perm \cdot s^{-1} $. In the subsections below, I give an explicit and nearly complete answer to the following question:

\begin{center}
\emph{For which transpositions $ t \in S_{\edges} $ is $ ( \whiteperm^t, \blackperm ) $ transitive?}
\end{center}

I will also show that the exceptional cases in which I do not make any absolute assertion are genuinely ambiguous and equivalent to a more general question that I am currently unable to answer satisfactorily.

\subsection{The Non-Exceptional case} \label{SSnonexceptional}

For the convenience of the reader, I recall from \S\ref{Siteration} the concrete descriptions of the exceptional cases:

\begin{dfn}
The pair $ ( \whiteperm, \blackperm ) $ is 

\begin{itemize}
\setlength{\itemsep}{5pt}

\item[-] \emph{Tame Exceptional \#1} relative to $ ( a, b ) $ iff $ a, \whiteperm ( a ), b, \whiteperm ( b ) $ are distributed into $ \whiteperm \blackperm $-cycles as either $ ( \ldots \whiteperm ( b ) \ldots \whiteperm ( a ) \ldots a \ldots b \ldots ) $ with $ \whiteperm ( b ) \neq b $, or this after exchanging $ a $ and $ b $.

\item[-] \emph{Tame Exceptional \#2} relative to $ ( a, b ) $ iff $ a, \whiteperm ( a ), b, \whiteperm ( b ) $ are distributed into $ \whiteperm \blackperm $-cycles as $ ( \ldots a \ldots b \ldots ), ( \ldots \whiteperm ( a ) \ldots \whiteperm ( b ) \ldots ) $.

\item[-] \emph{Wild Exceptional} relative to $ ( a, b ) $ iff $ a, \whiteperm ( a ), b, \whiteperm ( b ) $ are distributed into $ \whiteperm \blackperm $-cycles as $ ( \ldots a \ldots \whiteperm ( a ) \ldots ), ( \ldots b \ldots \whiteperm ( b ) \ldots ) $.
\end{itemize}
\end{dfn}

Note that, among Exceptional cases, Wild Exceptional is the only one that preserves synthetic genus -- if $ ( \whiteperm, \blackperm ) $ is Tame Exceptional relative to $ ( a, b ) $ then $ g^{\opsymbol \opsymbol} = g - 1 $, for $ g $ the synthetic genus of $ ( \whiteperm, \blackperm ) $ and $ g^{\opsymbol \opsymbol} $ that of $ ( \whiteperm^{\opsymbol \opsymbol}, \blackperm^{\opsymbol \opsymbol} ) $.

\begin{transitivitytheorem} \label{THMtransitivity}
For $ t \in S_{\edges} $ the transposition exchanging $ a $ and $ b $, if $ ( \whiteperm, \blackperm ) $ is not Exceptional relative to $ ( a, b ) $ then $ ( \whiteperm^t, \blackperm ) $ is transitive.
\end{transitivitytheorem}

Analysis of the exceptional cases is given in subsections \S\ref{SSwildexceptional}, \S\ref{SStameexceptional} below; classification of $ t $ according to the genus of $ ( \whiteperm^t, \blackperm ) $ is given in \S\ref{Sclassificationbygenus}.

\begin{proof}
By Proposition \ref{PRPconjugationviaoperations}, it suffices to prove that $ ( \whiteperm^{\opsymbol \opsymbol}, \blackperm^{\opsymbol \opsymbol} ) $ is transitive. I heavily use notation of the form ``Type X(Y)'', in reference to the analysis of branching given in \S\ref{Siteration}. The proof is by cases.

Many cases can be eliminated by use of the following simple observation: \emph{If $ ( \whiteperm, \blackperm ) $ is Type U, Type P1, Type P3, or Type P4 relative to $ ( a, b ) $ then $ ( \whiteperm^{\opsymbol}, \blackperm^{\opsymbol} ) $ is transitive.} This is true by Lemma \ref{LEMmodelsforoperation} and the proof of the Reroute Theorem \ref{THMgenus}: $ a_{\white} $ and $ b $ represent the same $ \whiteperm^{\opsymbol} \blackperm^{\opsymbol} $-orbit, so there is a walk in $ D^{\opsymbol} $ from $ \white_a $ to $ \white_b $ (cf. \S\ref{Sarcs}), and then to $ \black_a $ along edge $ a_{\black} $. Applying this observation twice, it is seen that $ ( \whiteperm^{\opsymbol \opsymbol}, \blackperm^{\opsymbol \opsymbol} ) $ can only (in principle) be non-transitive if $ ( \whiteperm, \blackperm ) $ is Type U(\ref{TYPEU2}), Type N(\ref{TYPEN1}), Type N(\ref{TYPEN3}), Type N(\ref{TYPEN4}), Type P1(\ref{neverspherealwaystrans}), Type P2(\ref{TYPEP21}), Type P2(\ref{TYPEP22}), Type P2(\ref{TYPEP24}), Type P4(\ref{TYPEP42}). 

More cases can be eliminated by exploiting the symmetry of the claim with respect to $ a, b $: after exchanging $ a $ and $ b $, Type N(\ref{TYPEN3}) becomes Type P1(\ref{neverspherealwaystrans}), Type N(\ref{TYPEN4}) becomes Type P3(\ref{TYPEP31}) (known to be transitive by the previous paragraph), Type P2(\ref{TYPEP21}) becomes Type P4(\ref{TYPEP42}), Type P2(\ref{TYPEP22}) becomes Type P4(\ref{TYPEP43}) (also known), Type P2(\ref{TYPEP24}) becomes Type U(\ref{TYPEU3}) (also known).

Altogether, transitivity must be verified only for Type U(\ref{TYPEU2}), Type N(\ref{TYPEN1}), Type P1(\ref{neverspherealwaystrans}), Type P4(\ref{TYPEP42}). It is equivalent by Lemma \ref{LEMtransiffconn} to prove that $ D^{\opsymbol \opsymbol} $ is connected.

Suppose first that $ ( \whiteperm, \blackperm ) $ is Type U(\ref{TYPEU2}). In particular, $ ( \whiteperm^{\opsymbol}, \blackperm^{\opsymbol} ) $ is transitive, so Lemma \ref{LEMmodelsforoperation} says that it is equivalent to find a walk in $ D^{\opsymbol \opsymbol} $ from $ \white_b $ to $ \black_b $. Recall that $ a_{\black} \subset D^{\opsymbol} \setminus b \subset D^{\opsymbol \opsymbol} \supset b_{\black} $ where $ b_{\black} $ is the ``new'' edge from $ \black_b $ to $ \white_a $. Since $ a_{\black} $ connects $ \black_a $ to $ \white_b $, it is therefore sufficient to exhibit a walk in $ D^{\opsymbol} \setminus b $ from $ \white_a $ to $ \black_a $. By nature of Type U(\ref{TYPEU2}), $ \whiteperm^{\opsymbol} \blackperm^{\opsymbol} $ contains a cycle of the form $ ( \ldots \whiteperm^{\opsymbol} ( b ) \ldots b \ldots a_{\black} \ldots a_{\white} \ldots ) $. Let $ x_0, x_1, \ldots, x_n \in \edges^{\opsymbol} $ be the minimal $ \whiteperm^{\opsymbol} \blackperm^{\opsymbol} $-sequence from $ a_{\black} $ to $ a_{\white} $. As discussed in \S\ref{Sarcs}, this produces a walk, via the sequence of edges $ x_0, \blackperm^{\opsymbol} ( x_0 ), x_1, \ldots, \blackperm^{\opsymbol} ( x_{n-1} ), x_n $, from either vertex of $ x_0 $ to either vertex of $ x_n $. Thus, it remains only to check that this walk exists in $ D^{\opsymbol} \setminus b $, i.e. that $ b $ is not among the edges $ x_0, \blackperm^{\opsymbol} ( x_0 ), x_1, \ldots, \blackperm^{\opsymbol} ( x_{n-1} ), x_n $. Since $ \whiteperm^{\opsymbol} ( \blackperm^{\opsymbol} ( x_i ) ) = x_{i+1} $ for all $ 0 \leq i < n $, it suffices to show that $ x_i \neq b, \whiteperm^{\opsymbol} ( b ) $ for all $ 0 \leq i \leq n $. But this is obvious from the structure of the $ \whiteperm^{\opsymbol} \blackperm^{\opsymbol} $-cycle.

Suppose now that $ ( \whiteperm, \blackperm ) $ is Type P1(\ref{neverspherealwaystrans}). This proof is identical to that for Type U(\ref{TYPEU2}). Since $ ( \whiteperm^{\opsymbol}, \blackperm^{\opsymbol} ) $ is transitive, it is equivalent to have a walk in $ D^{\opsymbol \opsymbol} $ from $ \white_b $ to $ \black_b $. As before, it is sufficient to exhibit a walk in $ D^{\opsymbol} \setminus b $ from $ \white_a $ to $ \black_a $. This is done exactly as before: by nature of Type P1(\ref{neverspherealwaystrans}), $ \whiteperm^{\opsymbol} \blackperm^{\opsymbol} $ contains a cycle of the form $ ( \ldots \whiteperm^{\opsymbol} ( b ) \ldots b \ldots a_{\white} \ldots a_{\black} \ldots ) $, and the walk associated with the minimal $ \whiteperm^{\opsymbol} \blackperm^{\opsymbol} $-sequence from $ a_{\white} $ to $ a_{\black} $ is contained in $ D^{\opsymbol} \setminus b $.

Suppose that $ ( \whiteperm, \blackperm ) $ is Type P4(\ref{TYPEP42}). The proof is nearly identical to that for Type U(\ref{TYPEU2}) and Type P1(\ref{neverspherealwaystrans}): $ ( \whiteperm^{\opsymbol}, \blackperm^{\opsymbol} ) $ is transitive and $ \whiteperm^{\opsymbol} \blackperm^{\opsymbol} $ contains a cycle of the form $ ( \ldots a_{\white} \ldots a_{\black} \ldots ) $, with both $ b $ and $ \whiteperm^{\opsymbol} ( b ) $ contained in a different cycle, so there is a walk in $ D^{\opsymbol} \setminus b $ from $ \white_a $ to $ \black_a $.

Suppose finally that $ ( \whiteperm, \blackperm ) $ is Type N(\ref{TYPEN1}). The proof here is different, because it is possible that $ ( \whiteperm^{\opsymbol}, \blackperm^{\opsymbol} ) $ is \emph{not} transitive (so Lemma \ref{LEMmodelsforoperation} cannot be used). If $ ( \whiteperm^{\opsymbol}, \blackperm^{\opsymbol} ) $ is transitive then the claim is immediate by the second paragraph of this proof: by nature of Type N(\ref{TYPEN1}), $ ( \whiteperm^{\opsymbol}, \blackperm^{\opsymbol} ) $ is Type U relative to $ ( b, a_{\white} ) $. So, I can assume that $ D^{\opsymbol} $ is disconnected. Necessarily, $ D^{\opsymbol} = D^{\opsymbol}_{\white} \sqcup D^{\opsymbol}_{\black} $ with both $ D^{\opsymbol}_{\white}, D^{\opsymbol}_{\black} $ connected and $ \white_a \in D^{\opsymbol}_{\white} $, $ \black_a \in D^{\opsymbol}_{\black} $. Since the edge $ a_{\black} \subset D^{\opsymbol} $ connects $ \black_a $ to $ \white_b $, necessarily $ \white_b \in D^{\opsymbol}_{\black} $, which forces $ b \subset D^{\opsymbol}_{\black} $ and $ \black_b \in D^{\opsymbol}_{\black} $. Since the underlying graph of $ D^{\opsymbol \opsymbol} $ is formed from $ D^{\opsymbol} \setminus b $ by attaching a new edge $ b_{\black} $ between $ \black_b $ and $ \white_a $, it suffices to show that $ D^{\opsymbol}_{\black} \setminus b $ is connected. The monodromy pair $ ( \varsigma_{\white}, \varsigma_{\black} ) $ of the connected dessin d'famille $ ( D^{\opsymbol}_{\black}, X^{\opsymbol} ) $ is simply the restriction of $ ( \whiteperm^{\opsymbol}, \blackperm^{\opsymbol} ) $ to the edges of $ D^{\opsymbol}_{\black} $. In particular, the orbits of $ \varsigma_{\white} \varsigma_{\black} $ are a certain subset of the orbits of $ \whiteperm^{\opsymbol} \blackperm^{\opsymbol} $. By nature of Type N(\ref{TYPEN1}), $ b $ and $ \whiteperm^{\opsymbol} ( b ) $ represent different $ \whiteperm^{\opsymbol} \blackperm^{\opsymbol} $-orbits. Since, by Corollary \ref{CORedgedeletion}, $ D^{\opsymbol}_{\black} \setminus b $ disconnected implies that $ b $ and $ \varsigma_{\white} ( b ) $ represent the same $ \varsigma_{\white} \varsigma_{\black} $-orbit, $ D^{\opsymbol}_{\black} \setminus b $ must be connected.
\end{proof}

\begin{rmk}
Due to the fact that genus is never negative, the reader may wonder how it is possible that $ ( \whiteperm^t, \blackperm ) $ is \emph{always} transitive for certain genus-lowering $ t $, specifically Type N(\ref{TYPEN3}) and Type P1(\ref{neverspherealwaystrans}). The answer is that the surface of such a dessin d'enfant is never $ \sphere $. Something of the ``toral'' nature of these types can be seen from the $ \whiteperm \blackperm $-cycle containing $ a, \whiteperm ( a ), b, \whiteperm ( b ) $. For example, for Type N(\ref{TYPEN3}) there is a $ \whiteperm \blackperm $-cycle of the form $ ( \ldots a \ldots b \ldots \whiteperm ( a ) \ldots \whiteperm ( b ) \ldots ) $, which means that the boundary walk of the corresponding face is of the form $ \ldots \white, a, \black \ldots \white, b, \black \ldots \black, a, \white \ldots \black, b, \white \ldots $, an appearance of the ``square model'' of $ \torus $. The same is true for Type P1(\ref{neverspherealwaystrans}) after exchanging $ a $ and $ b $.
\end{rmk}

Here is a nice corollary in the spherical case:

\begin{cor} \label{CORspherecasetheorem}
Assume that $ X = \sphere $ and that $ ( \whiteperm, \blackperm ) $ is not Wild Exceptional relative to $ ( a, b ) $. \Assertion $ ( \whiteperm^t, \blackperm ) $ is transitive if and only if $ \numcycles ( \whiteperm^t \blackperm ) \leq \numcycles ( \whiteperm \blackperm ) $.
\end{cor}

\begin{proof}
If $ ( \whiteperm^t, \blackperm ) $ is transitive and $ \numcycles ( \whiteperm^t \blackperm ) > \numcycles ( \whiteperm \blackperm ) $ then the dessin d'enfant corresponding to $ ( \whiteperm^t, \blackperm ) $ has genus strictly lower than that of $ ( \whiteperm, \blackperm ) $, which is absurd. Conversely, let $ g^t $ be the synthetic genus of $ ( \whiteperm^t, \blackperm ) $. By the inner hypothesis, $ g^t \geq 0 $. This implies that $ ( \whiteperm, \blackperm ) $ is not Tame Exceptional relative to $ ( a, b ) $, since Tame Exceptional would imply $ g^t < 0 $. Combined with the outer hypothesis, $ ( \whiteperm, \blackperm ) $ is not Exceptional relative to $ ( a, b ) $. By the Transitivity Theorem \ref{THMtransitivity}, $ ( \whiteperm^t, \blackperm ) $ is transitive.
\end{proof}

Half of Corollary \ref{CORspherecasetheorem} is true generally:

\begin{cor} \label{CORhighergenusimpliestransitive}
If $ \numcycles ( \whiteperm^t \blackperm ) < \numcycles ( \whiteperm \blackperm ) $ then $ ( \whiteperm^t, \blackperm ) $ is transitive.
\end{cor}

\begin{proof}
By hypothesis, the synthetic genus of $ ( \whiteperm^t, \blackperm ) $ is strictly larger than that of $ ( \whiteperm, \blackperm ) $. This implies that $ ( \whiteperm, \blackperm ) $ is not Exceptional relative to $ ( a, b ) $ so, by the Transitivity Theorem \ref{THMtransitivity}, $ ( \whiteperm^t, \blackperm ) $ is transitive.
\end{proof}

\subsection{The Wild Exceptional case} \label{SSwildexceptional}

The following is fairly weak, but still necessary:

\begin{prp} \label{PRPgeneralconnvswalks}
\emph{Here, $ ( \whiteperm, \blackperm ) $ is not necessarily Wild Exceptional relative to $ ( a, b ) $.} Let $ t \in S_{\edges} $ be the transposition exchanging $ a $ and $ b $. Let $ ( D^t, X^t ) $ be a model for $ ( \whiteperm^t, \blackperm ) $. \Assertion $ D^t $ is connected if and only if at least one of the following exists:
\begin{enumerate}
\item a walk in the subgraph $ D \setminus ( a \cup b ) $ from $ \white_a $ to $ \black_a $

\begin{flushleft}
or
\end{flushleft}

\item a walk in the subgraph $ D \setminus ( a \cup b ) $ from $ \white_b $ to $ \black_b $

\begin{flushleft}
or
\end{flushleft}

\item a walk in the subgraph $ D \setminus ( a \cup b ) $ from $ \white_a $ to $ \white_b $

\begin{flushleft}
or
\end{flushleft}

\item a walk in the subgraph $ D \setminus ( a \cup b ) $ from $ \black_a $ to $ \black_b $
\end{enumerate}
\end{prp}

The relation of this to the Wild Exceptional situation is given after the proof.

\begin{proof}
First, note that the number of connected components of $ D \setminus ( a \cup b ) $ is at most three. Second, each connected component must contain at least one of the (not necessarily distinct) vertices $ \white_a, \black_a, \white_b, \black_b $. Third, the underlying graph $ D^t $ is constructed from $ D \setminus ( a \cup b ) $ by attaching two new edges, one between $ \black_a $ and $ \white_b $ and another between $ \black_b $ and $ \white_a $ (cf. Lemma \ref{LEMmodelsforoperation} and the proof of Proposition \ref{PRPconjugationviaoperations}). Both halves of the equivalence are obvious if $ D \setminus ( a \cup b ) $ is connected (use the third remark from the beginning of this paragraph). So, I can assume that $ D \setminus ( a \cup b ) $ is disconnected.

Suppose that $ D \setminus ( a \cup b ) $ has three connected components. In this case, the second remark can be sharpened: one component $ C $ must contain two of $ \white_a, \black_a, \white_b, \black_b $ and the other two components each contain one. Further, among the six possible pairs, $ C $ can only contain one of these four pairs: $ \white_a, \black_b $ or $ \white_a, \white_b $ or $ \black_a, \black_b $ or $ \black_a, \white_b $. It is clear that the cases in which the attachment of the two new edges unifies the three components into one are these: $ \white_a, \white_b \in C $ or $ \black_a, \black_b \in C $. It is also clear that these are precisely the cases among the four in which there is a walk from the list.

Suppose instead that $ D \setminus ( a \cup b ) $ has two connected components. If one of the two components contains only one of the vertices $ \white_a, \black_a, \white_b, \black_b $ then both halves of the equivalence are obvious (use the third remarks from the first paragraph). So, I can assume that each connected component contains two of the vertices $ \white_a, \black_a, \white_b, \black_b $. By the third remark in the first paragraph, it is clear that the attachment of the two new edges unifies the two components into one precisley when neither component contains both $ \white_a, \black_b $ (equivalently, neither component contains both $ \black_b, \white_a $). Conversely, this is precisely the situation in which there is a walk from the list.
\end{proof}

\begin{rmk}
If $ X = \sphere $ then $ \vert \pi_0 ( D \setminus ( a \cup b ) ) \vert = 3 $ by Corollary \ref{CORedgedeletion}.
\end{rmk}

Unfortunately, it seems that no better statement is possible for the Wild Exceptional case. What prohibits the possibility of a better statement, and is the reason for the word ``wild'', is the fact that no good relationship need exist between the edges $ a $ and $ b $: $ a $ and $ b $ need not border a common face and may be very far apart within the graph.

In short, understanding of the Wild Exceptional class as a whole requires understanding possibly very long walks from $ a $ to $ b $ that depend on the \emph{global} structure of the graph. To the best of my knowledge, not much can be said about this.

By contrast, Tame Exceptional requires that $ a, b $ represent the same $ \whiteperm \blackperm $-orbit; in particular, $ a $ and $ b $ border a common face and an explicit walk from $ a $ to $ b $ is easy to construct.

\subsection{The Tame Exceptional case} \label{SStameexceptional}

First, I provide simple examples to show the possibilities that can occur.

\begin{exm}[Tame Exceptional \#1, spherical / non-transitive] \label{EXMtypeP12nontransitive}
Define $ \whiteperm, \blackperm \in S_3 $ by specifying disjoint cycle decompositions: $ \whiteperm \defeq ( 1, 2 ) \cdot ( 3 ) $ and $ \blackperm \defeq ( 1 ) \cdot ( 2, 3 ) $. This pair $ ( \whiteperm, \blackperm ) $ is obviously transitive, $ \whiteperm \blackperm = ( 1, 2, 3 ) $, and the surface of the corresponding dessin d'enfant is $ \sphere $. \emph{This pair describes the linear dessin d'enfant with three edges.} It is evident that $ ( \whiteperm, \blackperm ) $ is Type P1(\ref{tameexceptionalP1}) relative to $ ( 1, 3 ) $, and therefore Tame Exceptional \#1. For $ t = ( 1, 3 ) $, it is clear that $ \whiteperm^t = \blackperm $ and so $ ( \whiteperm^t, \blackperm ) $ is not transitive. \emph{Of course, transitivity was expected to fail by considering synthetic genus.}
\end{exm}

\begin{exm}[Tame Exceptional \#1, toral / transitive] \label{EXMtypeP12transitive}
Define $ \whiteperm, \blackperm \in S_4 $ by specifying disjoint cycle decompositions: $ \whiteperm \defeq ( 1, 2, 3 ) \cdot ( 4 ) $ and $ \blackperm \defeq ( 1, 2, 4, 3 ) $. This pair $ ( \whiteperm, \blackperm ) $ is transitive. It is easy to compute that $ \whiteperm \blackperm = ( 1, 3, 2, 4 ) $. In particular, the formula for Euler Characteristic shows that the surface of this dessin d'enfant is $ \torus $. It is evident that $ ( \whiteperm, \blackperm ) $ is Type P1(\ref{tameexceptionalP1}) relative to $ ( 1, 4 ) $, and therefore Tame Exceptional \#1. The following depicts the dessin d'enfant corresponding to $ ( \whiteperm, \blackperm ) $ on the left and, for $ t = ( 1, 4 ) $, a model for $ ( \whiteperm^t, \blackperm ) $ on the right:
\begin{center}
\includegraphics[scale=0.4]{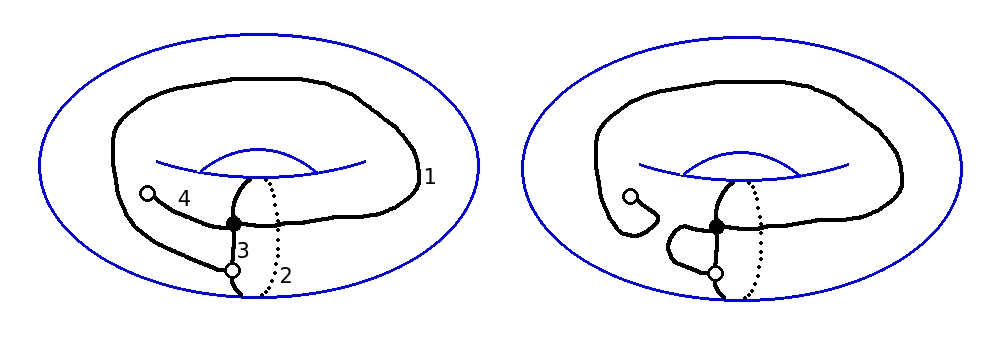}
\end{center}
Connectedness of the model implies that $ ( \whiteperm^t, \blackperm ) $ is transitive. Note that the ``true'' surface of the dessin d'enfant corresponding to $ ( \whiteperm^t, \blackperm ) $ is $ \sphere $.
\end{exm}

\begin{exm}[Tame Exceptional \#1, toral / non-transitive]
This example is created by gluing Example \ref{EXMtypeP12nontransitive} onto the trivial dessin d'enfant. Define $ \whiteperm, \blackperm \in S_6 $ by specifying disjoint cycle decompositions: $ \whiteperm \defeq ( 1, 2 ) \cdot ( 3 ) \cdot ( 4, 5, 6 ) $ and $ \blackperm \defeq ( 1, 4, 5, 6 ) \cdot ( 2, 3 ) $. The pair $ ( \whiteperm, \blackperm ) $ is clearly transitive, $ \whiteperm \blackperm = ( 1, 5, 4, 6, 2, 3 ) $, and the formula for Euler Characteristic shows that the surface of its dessin d'enfant is $ \torus $. The pair $ ( \whiteperm, \blackperm ) $ is Type P1(\ref{tameexceptionalP1}) relative to $ ( 1, 3 ) $ and therefore Tame Exceptional \#1. For $ t = ( 1, 3 ) $, $ \whiteperm^t = ( 3, 2 ) \cdot ( 1 ) \cdot ( 4, 5, 6 ) $ and so $ ( \whiteperm^t, \blackperm ) $ is clearly not transitive: both stabilize $ \{ 2, 3 \} $. The following depicts the dessin d'enfant corresponding to $ ( \whiteperm, \blackperm ) $ on the left and a (disconnected) model for $ ( \whiteperm^t, \blackperm ) $ on the right:
\begin{center}
\includegraphics[scale=0.4]{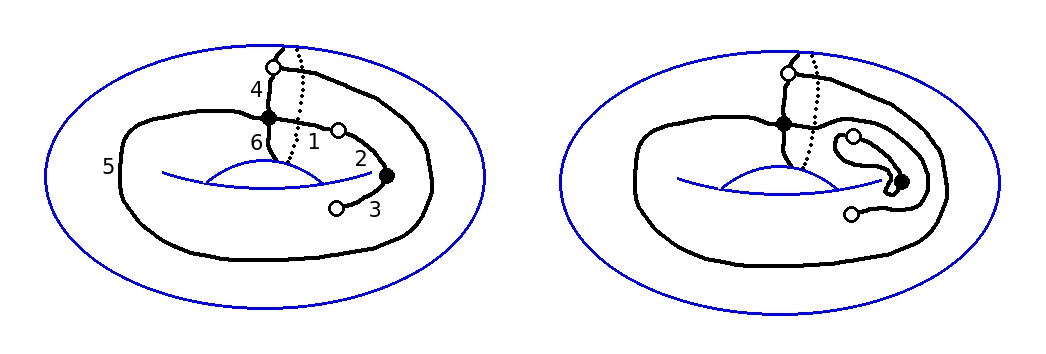}
\end{center}
\end{exm}

\begin{exm}[Tame Exceptional \#2, spherical / non-transitive] \label{EXMtypeP33spherenontransitive}
Define $ \whiteperm, \blackperm \in S_4 $ by specifying disjoint cycle decompositions: $ \whiteperm \defeq ( 1, 2 ) \cdot ( 3, 4 ) $ and $ \blackperm \defeq ( 1, 4 ) \cdot ( 2, 3 ) $. The pair $ ( \whiteperm, \blackperm ) $ is clearly transitive. It is easy to compute that $ \whiteperm \blackperm = ( 1, 3 ) \cdot ( 2, 4 ) $. In particular, the formula for Euler Characteristic shows that the surface of its dessin d'enfant is $ \sphere $. \emph{This is Example \ref{EXMcircuittypeP3}.} It is evident that $ ( \whiteperm, \blackperm ) $ is Type P3(\ref{tameexceptionalP3}) relative to $ ( 1, 3 ) $, and therefore Tame Exceptional \#2. But for $ t = ( 1, 3 ) $, it is clear that $ \whiteperm^t = \blackperm $ and $ \blackperm $ is not a maximal cycle so $ ( \whiteperm^t, \blackperm ) $ is not transitive. \emph{Of course, transitivity was expected to fail by considering synthetic genus.}
\end{exm}

\begin{exm}[Tame Exceptional \#2, toral / transitive] \label{EXMtypeP33transitive}
Define $ \whiteperm, \blackperm \in S_8 $ by specifying disjoint cycle decompositions: $ \whiteperm \defeq ( 1, 2, 3 ) \cdot ( 4, 5, 6 ) \cdot ( 7, 8 ) $ and $ \blackperm \defeq ( 1, 7, 5 ) \cdot ( 2, 6, 4 ) \cdot ( 3, 8 ) $. This pair $ ( \whiteperm, \blackperm ) $ is transitive. It is easy to compute that $ \whiteperm \blackperm = ( 1, 8 ) \cdot ( 2, 4, 3, 7, 6, 5 ) $. In particular, the formula for Euler Characteristic shows that the surface of this dessin d'enfant is $ \torus $. It is evident that $ ( \whiteperm, \blackperm ) $ is Type P3(\ref{tameexceptionalP3}) relative to $ ( 1, 8 ) $, and therefore Tame Exceptional \#2. The following depicts the dessin d'enfant corresponding to $ ( \whiteperm, \blackperm ) $ on the left and, for $ t = ( 1, 8 ) $, a model for $ ( \whiteperm^t, \blackperm ) $ on the right:
\begin{center}
\includegraphics[scale=0.4]{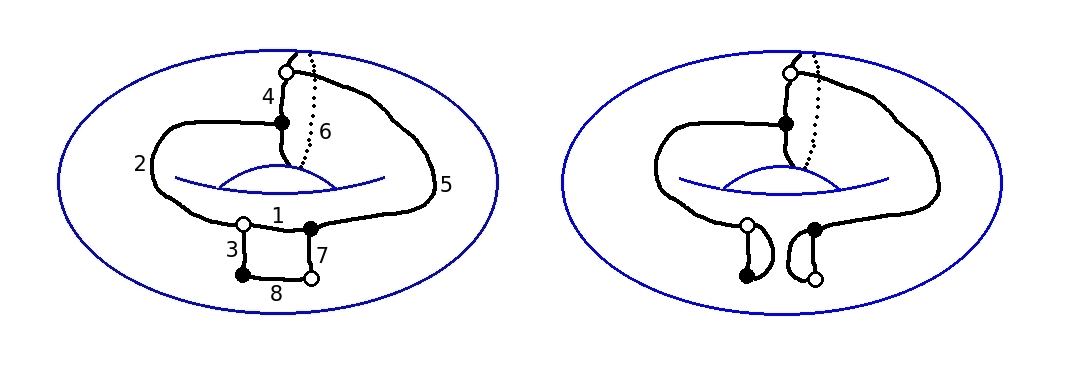}
\end{center}
Connectedness of the model implies that $ ( \whiteperm^t, \blackperm ) $ is transitive. Note that the ``true'' surface of the dessin d'enfant corresponding to $ ( \whiteperm^t, \blackperm ) $ is $ \sphere $.
\end{exm}

\begin{exm}[Tame Exceptional \#2, toral / non-transitive] \label{EXMtypeP33torusnontransitive}
Let $ ( \whiteperm, \blackperm ) $ be as in the previous Example \ref{EXMtypeP33transitive}, but choose instead the transposition $ t = ( 3, 7 ) $. It is again true that $ ( \whiteperm, \blackperm ) $ is Type P3(\ref{tameexceptionalP3}) relative to $ ( 3, 7 ) $, but $ ( \whiteperm^t, \blackperm ) $ is not transitive.
\end{exm}

The following shows that to decide transitivity in the Tame Exceptional situation is also equivalent to a seemingly subtle question about deletion vs. connectedness in graphs which, nonetheless, is simpler than that in Proposition \ref{PRPgeneralconnvswalks}:

\begin{prp} \label{PRPtameexctransiffwalks}
Assume that $ ( \whiteperm, \blackperm ) $ is Tame Exceptional relative to $ ( a, b ) $. Let $ t \in S_{\edges} $ be the transposition exchanging $ a $ and $ b $. Let $ ( D^t, X^t ) $ be a model for $ ( \whiteperm^t, \blackperm ) $. \Assertion $ D^t $ is connected if and only if at least one of the following exists:
\begin{enumerate}
\item a walk in the subgraph $ D \setminus ( a \cup b ) $ from $ \white_a $ to $ \black_a $

\begin{flushleft}
or
\end{flushleft}

\item a walk in the subgraph $ D \setminus ( a \cup b ) $ from $ \white_b $ to $ \black_b $
\end{enumerate}
\end{prp}

At the end of the subsection, a short discussion of the ``walks'' condition is provided.

\begin{proof}
Recall from \S\ref{SSnonexceptional} the disjoint cycle decompositions that are possible for Tame Exceptional. Since each half of the equivalence is symmetric in $ a $ and $ b $, I can exchange $ a $ and $ b $ if necessary and assume that the $ \whiteperm \blackperm $-arc from $ a $ to $ b $ contains neither $ \whiteperm ( a ) $ nor $ \whiteperm ( b ) $. If $ x_0, x_1, \ldots, x_n \in \edges $ is the minimal $ \whiteperm \blackperm $-sequence from $ a $ to $ b $ then $ x_i \neq a, b $ for all $ 0 < i < n $ and $ x_i \neq \whiteperm ( a ), \whiteperm ( b ) $ for all $ 0 < i \leq n $. This implies that the sequence $ \blackperm ( x_0 ), x_1, \ldots, x_{n-1}, \blackperm ( x_{n-1} ) $ does not contain $ a, b $. As explained in \S\ref{Sarcs}, the sequence $ \blackperm ( x_0 ), x_1, \ldots, x_{n-1}, \blackperm ( x_{n-1} ) $ defines a walk $ W $, necessarily in $ D \setminus ( a \cup b ) $, from $ \black_a $ to $ \white_b $. The equivalence is now immediate from Proposition \ref{PRPgeneralconnvswalks}: by concatenating with the walk $ W $ if necessary, the list of four walks in Proposition \ref{PRPgeneralconnvswalks} reduces to the list of two walks here.
\end{proof}

Observe that if $ D \setminus ( a \cup b ) $ is connected then the statement regarding walks in Proposition \ref{PRPtameexctransiffwalks} is true, but not conversely. Similarly, if the statement regarding walks in Proposition \ref{PRPtameexctransiffwalks} is true then at least one of $ D \setminus a $ or $ D \setminus b $ is connected, but not conversely -- see Example \ref{EXMtypeP33torusnontransitive}.

\begin{rmk}
It is intuitive and tempting to think that the equivalence in Proposition \ref{PRPtameexctransiffwalks} might generalize completely, but this is false. On the other hand, it is indeed true generally that if the statement about walks is true then $ D^t $ is connected.
\end{rmk}

The value of Proposition \ref{PRPtameexctransiffwalks} would be greatly increased if the following question could be answered:

\begin{deletionquestionII}
Is there a ``good'' characterization, in terms of the monodromy pair $ ( \whiteperm, \blackperm ) $, of those pairs $ a, b \in \edges $ such that there is a walk in $ D \setminus ( a \cup b ) $ from $ \white_a $ to $ \black_a $ or from $ \white_b $ to $ \black_b $?
\end{deletionquestionII}

Since $ D \setminus e $ is connected iff and only if there is a walk in $ D \setminus e $ from $ \white_e $ to $ \black_e $, this is some kind of ``second order'' analogue of Deletion Question 1 from \S\ref{Sdeletion}.

\section{Classification 2: by Genus} \label{Sclassificationbygenus}

\begin{center}
\emph{Throughout this section, $ ( G, X ) $ is a dessin d'famille with edges $ \edges $ and monodromy pair $ ( \whiteperm, \blackperm ) $, and $ g $ is the synthetic genus of $ ( \whiteperm, \blackperm ) $.}
\end{center}

Recall that if $ \perm, s \in S_{\edges} $ then $ \perm^s \defeq s \cdot \perm \cdot s^{-1} $. In this section, I give a complete and explicit answer to the following question:

\begin{center}
\emph{If $ t \in S_{\edges} $ is a transposition then what is the synthetic genus of $ ( \whiteperm^t, \blackperm ) $?}
\end{center}

This is mostly a matter of collating the facts from \S\ref{Siteration}. Fix $ a, b \in \edges $ and let $ t \in S_{\edges} $ be the transposition exchanging $ a $ and $ b $.

\subsection{Genus-Raising Transpositions}

\begin{prp} \label{PRPgenusraising}
Let $ g^t $ be the synthetic genus of $ ( \whiteperm^t, \blackperm ) $. \Assertions $ g^t > g $ if and only if $ t $ has the following property:
\begin{enumerate}
\item At least one of $ a, b $ represents a different $ \whiteperm \blackperm $-orbit than the other three of $ a, \whiteperm ( a ), b, \whiteperm ( b ) $

\begin{flushleft}
and 
\end{flushleft}

\item at least one of $ \whiteperm ( a ), \whiteperm ( b ) $ represents a different $ \whiteperm \blackperm $-orbit than the other three of $ a, \whiteperm ( a ), b, \whiteperm ( b ) $.
\end{enumerate}
In this case, $ g^t = g + 1 $. If $ ( \whiteperm, \blackperm ) $ is transitive then $ ( \whiteperm^t, \blackperm ) $ is also transitive. \emph{It is not assumed that $ a, \whiteperm ( a ), b $ are distinct.}
\end{prp}

\begin{proof}
According to \S\ref{Siteration} and the Reroute Theorem \ref{THMgenus}, the situations in which $ g^t > g $ are precisely these five: Type U(\ref{TYPEU1}), Type U(\ref{TYPEU3}), Type U(\ref{TYPEU4}), Type P2(\ref{TYPEP24}), Type P4(\ref{TYPEP44}). The fact that $ g^t = g + 1 $ is then immediate from the Reroute Theorem \ref{THMgenus}. This then implies, in the case that $ ( \whiteperm, \blackperm ) $ is transitive, that $ ( \whiteperm^t, \blackperm ) $ is also transitive, by Corollary \ref{CORhighergenusimpliestransitive}. Explicitly, the five situations above are described by saying that $ a, \whiteperm ( a ), b, \whiteperm ( b ) $ are distributed into $ \whiteperm \blackperm $-cycles in one of the following ways:
\begin{itemize}
\setlength{\itemsep}{5pt}

\item[-] $ ( \ldots a \ldots \whiteperm ( b ) \ldots ), ( \ldots \whiteperm ( a ) \ldots ), ( \ldots b \ldots ) $

\item[-] $ ( \ldots a \ldots ), ( \ldots \whiteperm ( a ) \ldots ), ( \ldots b \ldots \whiteperm ( b ) \ldots ) $

\item[-] $ ( \ldots a \ldots ), ( \ldots \whiteperm ( a ) \ldots ), ( \ldots b \ldots ), ( \ldots \whiteperm ( b ) \ldots ) $

\item[-] $ ( \ldots a \ldots \whiteperm ( a ) \ldots ), ( \ldots b \ldots ), ( \ldots \whiteperm ( b ) \ldots ) $

\item[-] $ ( \ldots a \ldots ), ( \ldots \whiteperm ( a ) \ldots b \ldots ), ( \ldots \whiteperm ( b ) \ldots ) $
\end{itemize}

An easy ``placing balls into boxes'' argument shows that the given property extracts precisely these five situations from among all possible.
\end{proof}

Proposition \ref{PRPgenusraising} answers a question asked to me by Brian Hwang, which I record:

\begin{cor}
Let $ ( D, \sphere ) $ be a spherical dessin d'enfant with edges $ \edges $ and monodromy $ ( \whiteperm, \blackperm ) $. \Assertion The transpositions $ t \in S_{\edges} $ for which $ ( \whiteperm^t, \blackperm ) $ defines a dessin d'enfant on $ \torus $ are precisely those satisfying the property in Proposition \ref{PRPgenusraising}.
\end{cor}

\begin{rmk}
There are many dessins d'enfants for which no genus-raising transpositions exist. For example, Proposition \ref{PRPgenusraising} implies that if a dessin d'enfant has only one face, what Adrianov-Shabat call ``unicellular'', then it is impossible for any transposition to be genus-raising.
\end{rmk}

\begin{rmk}
Unlike the genus-lowering and genus-preserving transpositions, the genus-raising transpositions can be described using only the concept of ``orbit'', rather than the more refined concept of ``cycle''.
\end{rmk}

\subsection{Genus-Lowering Transpositions}

\begin{prp} \label{PRPgenuslowering}
Let $ g^t $ be the synthetic genus of $ ( \whiteperm^t, \blackperm ) $. \Assertions $ g^t < g $ if and only if $ t $ has the following property:
\begin{enumerate}
\item Both $ a $ and $ b $ represent the same $ \whiteperm \blackperm $-orbit

\begin{flushleft}
and 
\end{flushleft}

\item both $ \whiteperm ( a ), \whiteperm ( b ) $ represent the same $ \whiteperm \blackperm $-orbit

\begin{flushleft}
and 
\end{flushleft}

\item either the $ \whiteperm \blackperm $-arc from $ a $ to $ b $ contains neither $ \whiteperm ( a ), \whiteperm ( b ) $ or the $ \whiteperm \blackperm $-arc from $ b $ to $ a $ contains neither.
\end{enumerate}
In this case, $ g^t = g - 1 $. Tame Exceptional situations may occur here, so $ ( \whiteperm^t, \blackperm ) $ may be transitive or not even if $ ( \whiteperm, \blackperm ) $ is transitive.
\end{prp}

\begin{proof}
According to \S\ref{Siteration} and the Reroute Theorem \ref{THMgenus}, the situations in which this can occur are precisely Type N(\ref{tameexceptionalN}), Type N(\ref{TYPEN3}), Type P1(\ref{neverspherealwaystrans}), Type P1(\ref{tameexceptionalP1}), Type P3(\ref{tameexceptionalP3}). The fact that $ g^t = g - 1 $ is then immediate from the Reroute Theorem \ref{THMgenus}. Explicitly, these five cases are described by saying that $ a, \whiteperm ( a ), b, \whiteperm ( b ) $ are distributed into $ \whiteperm \blackperm $-cycles in one of the following ways:
\begin{itemize}
\setlength{\itemsep}{5pt}

\item[-] $ ( \ldots \whiteperm ( b ) \ldots \whiteperm ( a ) \ldots a \ldots b \ldots ) $, with $ \whiteperm ( b ) \neq b $

\item[-] $ ( \ldots \whiteperm ( a ) \ldots \whiteperm ( b ) \ldots a \ldots b \ldots ) $, with $ \whiteperm ( a ) \neq b $

\item[-] $ ( \ldots \whiteperm ( b ) \ldots \whiteperm ( a ) \ldots b \ldots a \ldots ) $, with $ \whiteperm ( b ) \neq a $

\item[-] $ ( \ldots \whiteperm ( a ) \ldots \whiteperm ( b ) \ldots b \ldots a \ldots ) $, with $ \whiteperm ( a ) \neq a $

\item[-] $ ( \ldots a \ldots b \ldots ), ( \ldots \whiteperm ( a ) \ldots \whiteperm ( b ) \ldots ) $, with $ b \neq \whiteperm ( a ) $
\end{itemize}

An easy ``placing balls into boxes'' argument shows that the given property extracts precisely these five situations from among all possible.
\end{proof}

\begin{rmk}
For some dessins d'enfants, there are no genus-lowering transpositions. This is trivial in $ \sphere $, but examples exist in higher genus. For example, define $ \whiteperm = ( 1, 2, 3, 4, 5 ) $ and $ \blackperm = ( 1, 5, 3, 2, 4 ) $. This pair $ ( \whiteperm, \blackperm ) $ is transitive, and it is easy to compute that $ \whiteperm \blackperm = ( 1 ) \cdot ( 2, 5, 4 ) \cdot ( 3 ) $, so $ ( \whiteperm, \blackperm ) $ defines a dessin d'enfant with synthetic Euler characteristic $ ( 1 + 1 ) - 5 + 3 = 0 $. Since $ \whiteperm \blackperm $ contains only one non-singleton cycle, the only way that a transposition $ t = ( a, b ) $ can satisfy the first two requirements of Proposition \ref{PRPgenuslowering} is if $ a, \whiteperm ( a ), b, \whiteperm ( b ) $ all represent the same $ \whiteperm \blackperm $-orbit. But it is clear from $ \whiteperm, \whiteperm \blackperm $ that there is only one edge $ e $ for which $ e, \whiteperm ( e ) $ represent the same $ \whiteperm \blackperm $-orbit: $ e = 4 $. Therefore, no such $ a, b $ can exist.

In particular, some dessins d'enfants in $ \torus $ do not come from $ \sphere $ via conjugation by transpositions: If $ ( \whiteperm^{\prime}, \blackperm^{\prime} ) = ( \whiteperm^t, \blackperm^{\prime} ) $ with $ ( \whiteperm^{\prime}, \blackperm^{\prime} ) $ toral and $ ( \whiteperm, \blackperm^{\prime} ) $ spherical then $ t $ must be genus-lowering for $ ( \whiteperm^{\prime}, \blackperm^{\prime} ) $.
\end{rmk}

\subsection{Genus-Preserving Transpositions}

Although it is possible to describe these directly, it seems best to simply negate those properties from Propositions \ref{PRPgenusraising} and \ref{PRPgenuslowering}.

\section*{Appendix: \texttt{MAGMA}}

The following \texttt{MAGMA} functions were used to check the assertions in \S\ref{Sfundamentaloperation}, \S\ref{Siteration}, \S\ref{Sclassificationbytransitivity}. They are quite specific to the goals of this paper, except for one -- the function \texttt{MakeCycleCoercible} should be useful to anyone interested in permutations.

The functions were formatted so they can be processed by \texttt{MAGMA} without editing. In some cases, the definition of the function is preceded by a \texttt{forward} command. This command merely makes explicit that the function depends on some other function defined here.

\subsection*{Format a cycle for coercion}

The standard \texttt{MAGMA} function \texttt{Cycle} will, given a permutation and an element, return the part of the permutation's disjoint cycle decomposition containing the element. It is frequently desirable to use this cycle as a permutation. However, the object returned by \texttt{Cycle} is, from \texttt{MAGMA}'s perspective, not a permutation at all -- it is merely a sequence of positive integers.

There seems to be no easy or standard way for \texttt{MAGMA} to interpret this sequence as a permutation. For example, an error results if one attempts to coerce (typecast) the sequence into the original symmetric group. It seems that the only way to create a permutation at runtime is to coerce a sequence \texttt{S}, where \texttt{S[i]} indicates the image of \texttt{i} (``two row notation''). The following function performs the desired conversion.

\codebox{
MakeCycleCoercible := function( cycle, n )\\
tworowseq := [];\\
loopedcycle := Append( Setseq( cycle ), cycle[1] );\\
for i := 1 to n do\\
if i in loopedcycle then\\
Append( \twiddle tworowseq, loopedcycle[ Index(loopedcycle, i) + 1 ] );\\
else\\
Append( \twiddle tworowseq, i );\\
end if;\\
end for;\\
return tworowseq;\\
end function;
}

The input \texttt{cycle} is an object of type \texttt{SetIndx} (\emph{Indexed Set}), the same type of object returned by the \texttt{MAGMA} function \texttt{Cycle}. Input \texttt{n} is a positive integer at least as large as the integers in \texttt{cycle}. The returned object is a sequence of \texttt{n} positive integers, can be coerced via \texttt{Sym(n)!} or similar, and the result behaves exactly as \texttt{cycle} should. I deliberately do not coerce the returned object because the user may frequently want to further modify it (I do this myself below, in the body of \texttt{Reroute}).

\subsection*{Compute the synthetic genus of a permutation pair}

The following function implements Definition \ref{DFNsyntheticgenus} and is otherwise self-explanatory.

\codebox{
ComputeGenus := function( white, black )\\
return ( 1 - ( ( \#CycleDecomposition( white ) + \#CycleDecomposition( black ) - Degree( Parent( white ) ) + \#CycleDecomposition( black*white ) ) / 2 ) );\\
end function;
}

There are two other ways to count cycles besides the standard \texttt{MAGMA} function \texttt{CycleDecomposition}: by extracting the quantities of cycles of each length via the standard \texttt{MAGMA} function \texttt{CycleStructure}, or by counting the size of the set returned by the standard \texttt{MAGMA} function \texttt{Orbits}. It is difficult to believe that either of these alternatives is more efficient.

\subsection*{Implementation of the Reroute operation}

The following function implements Definition \ref{DFNfundamentaloperation}.

\codebox{
forward MakeCycleCoercible;\\
Reroute := function( white, black, a, b )\\
n := Degree( Parent( white ) );\\
Gext := Sym( n+1 );\\
W := Identity( Gext );\\
B := Identity( Gext );\\
cycles := CycleDecomposition( white );\\
for c in cycles do\\
if b in c then\\
cnew := MakeCycleCoercible( c, n+1 );\\
cnew[ Position( cnew, b ) ] := n+1;\\
cnew[n+1] := b;\\
W := W*( Gext ! cnew );\\
else\\
W := W*( Gext ! MakeCycleCoercible( c, n+1 ) );\\
end if;\\
end for;\\
cycles := CycleDecomposition( black );\\
for c in cycles do\\
if a in c then\\
cnew := MakeCycleCoercible( c, n+1 );\\
cnew[ Position( cnew, a ) ] := n+1;\\
cnew[n+1] := cnew[a];\\
cnew[a] := a;\\
B := B*( Gext ! cnew );\\
else\\
B := B*( Gext ! MakeCycleCoercible( c, n+1 ) );\\
end if;\\
end for;\\
return W, B;\\
end function;
}

The returned object needs to be explained. Let \texttt{1}, \texttt{2}, \ldots, \texttt{n} be the set permuted by the inputs \texttt{white} and \texttt{black}. This function \texttt{Reroute} returns a pair of permutations in the symmetric group on \texttt{1}, \texttt{2}, \ldots, \texttt{n+1}. The first one plays the role of $ \whiteperm^{\opsymbol} $ and the second one plays the role of $ \blackperm^{\opsymbol} $. In the new larger symmetric group, the ``new'' integer \texttt{n+1} plays the role of $ a_{\black} $. By abuse of notation, the ``old'' integer \texttt{a} plays the role of $ a_{\white} $.

\subsection*{Compute an arc from one element to another}

The following function implements Definition \ref{DFNarcs}.

\codebox{
ComputeArc := function( g, a, b )\\
if a eq b then\\
return [];\\
end if;\\
aorbit := Cycle( g, a );\\
arc := [];\\
for i := 2 to \#aorbit do /*a is first element of aorbit*/\\
Append( \twiddle arc, aorbit[i] );\\
if aorbit[i] eq b then\\
return arc;\\
end if;\\
end for;\\
end function;
}

The function \texttt{ComputeArc} does \emph{not} check whether \texttt{a} and \texttt{b} are in the same \texttt{g}-orbit -- this is the user's responsibility.

\subsection*{Check the Type of a permutation pair}

These functions implement Definitions \ref{DFNtypes}/\ref{DFNsubtypes} and are all self-explanatory. The implementation of each function reflects as closely as possible the actual mathematical definitions even when this is less economical than a logically equivalent method.

\codebox{
IsTypeU := function( white, black, a, b )\\
g := black*white; /*MAGMA acts on the right!*/\\
aorbit := Cycle( g, a );\\
wa := a\caret white;\\
if ( b in aorbit ) or ( wa in aorbit ) then\\
return false;\\
end if;\\
borbit := Cycle( g, b );\\
if wa in borbit then\\
return false;\\
end if;\\
return true;\\
end function;
}

\codebox{
forward ComputeArc;\\
IsTypeN := function( white, black, a, b )\\
aorbit := Cycle( black*white, a ); /*MAGMA acts on the right!*/\\
wa := a\caret white;\\
if ( b notin aorbit ) or ( wa notin aorbit ) then\\
return false;\\
end if;\\
return wa notin ComputeArc( black*white, a, b );\\
end function;
}

\codebox{
forward IsTypeU;\\
forward IsTypeN;\\
IsTypeP := function( white, black, a, b )\\
return not ( IsTypeU(white,black,a,b) or IsTypeN(white,black,a,b) );\\
end function;
}

\codebox{
forward ComputeArc;\\
IsTypeP1 := function( white, black, a, b )\\
aorbit := Cycle( black*white, a ); /*MAGMA acts on the right!*/\\
wa := a\caret white;\\
if ( b notin aorbit ) or ( wa notin aorbit ) then\\
return false;\\
end if;\\
return wa in ComputeArc( black*white, a, b );\\
end function;
}

\codebox{
IsTypeP2 := function( white, black, a, b )\\
aorbit := Cycle( black*white, a ); /*MAGMA acts on the right!*/\\
return ( ( a\caret white ) in aorbit ) and ( b notin aorbit );\\
end function;
}

\codebox{
IsTypeP3 := function( white, black, a, b )\\
aorbit := Cycle( black*white, a ); /*MAGMA acts on the right!*/\\
return ( b in aorbit ) and ( ( a\caret white ) notin aorbit );\\
end function;
}

\codebox{
IsTypeP4 := function( white, black, a, b )\\
borbit := Cycle( black*white, b ); /*MAGMA acts on the right!*/\\
return ( ( a\caret white ) in borbit ) and ( a notin borbit );\\
end function;
}

\subsection*{Check whether a permutation pair is Exceptional}

The following functions test, to varying degrees of specificity, whether a given permutation pair is Exceptional.

\codebox{
IsTameExceptional1B := function( white, black, a, b )\\
acycle := Cycle( black*white, a ); /*MAGMA acts on the right!*/\\
wa := a\caret white;\\
wb := b\caret white;\\
if (b notin acycle) or (wa notin acycle) or (wb notin acycle) then\\
return false;\\
end if;\\
x := Position( acycle, wa );\\
y := Position( acycle, wb );\\
z := Position( acycle, b );\\
return (1 lt x) and (x lt y) and (y le z); /*3rd inequality weak!*/\\
end function;
}

\codebox{
forward IsTameExceptional1B;\\
IsTameExceptional1A := function( white, black, a, b )\\
return IsTameExceptional1B( white, black, b, a );\\
end function;
}

\codebox{
IsTameExceptional2 := function( white, black, a, b )\\
g := black*white; /*MAGMA acts on the right!*/\\
wa := a\caret white;\\
wb := b\caret white;\\
aorbit := Cycle( g, a );\\
waorbit := Cycle( g, wa );\\
return (a notin waorbit) and (b in aorbit) and (wb in waorbit);\\
end function;
}

\codebox{
forward IsTameExceptional1B;\\
forward IsTameExceptional1A;\\
forward IsTameExceptional2;\\
IsTameExceptional := function( white, black, a, b )\\
return IsTameExceptional1A( white, black, a, b ) or IsTameExceptional1B( white, black, a, b ) or IsTameExceptional2( white, black, a, b );\\
end function;
}

\codebox{
forward IsTypeP2;\\
IsWildExceptional := function( white, black, a, b )\\
return IsTypeP2(white,black,a,b) and IsTypeP2(white,black,b,a);\\
end function;
}

\codebox{
forward IsTameExceptional;\\
forward IsWildExceptional;\\
IsExceptional := function( white, black, a, b )\\
return IsTameExceptional(white,black,a,b) or IsWildExceptional(white,black,b,a);\\
end function;
}

\subsection*{Check how conjugation will change genus}

The following functions implement Propositions \ref{PRPgenusraising}/\ref{PRPgenuslowering}.

\codebox{
IsGenusRaising := function( white, black, a, b )\\
g := black*white; /*MAGMA acts on the right!*/\\
aorbit := Cycle( g, a );\\
if b in aorbit then\\
return false;\\
end if;\\
wa := a\caret white;\\
wb := b\caret white;\\
borbit := Cycle( g, b );\\
if ( ( wa in aorbit ) or ( wb in aorbit ) ) and ( ( wa in borbit ) or ( wb in borbit ) ) then\\
return false;\\
end if;\\
waorbit := Cycle( g, wa );\\
if wb in waorbit then\\
return false;\\
end if;\\
wborbit := Cycle( g, wb );\\
if ( ( a in waorbit ) or ( b in waorbit ) ) and ( ( a in wborbit ) or ( b in wborbit ) ) then\\
return false;\\
end if;\\
return true;\\
end function;
}

\codebox{
forward ComputeArc;\\
IsGenusLowering := function( white, black, a, b )\\
g := black*white; /*MAGMA acts on the right!*/\\
aorbit := Cycle( g, a );\\
if b notin aorbit then\\
return false;\\
end if;\\
wa := a\caret white;\\
wb := b\caret white;\\
waorbit := Cycle( g, wa );\\
if wb notin waorbit then\\
return false;\\
end if;\\
arcab := ComputeArc( g, a, b );\\
arcba := ComputeArc( g, b, a );\\
return ( ( wa notin arcab ) and ( wb notin arcab ) ) or ( ( wa notin arcba ) and ( wb notin arcba ) );\\
end function;
}

\codebox{
forward IsGenusRaising;\\
forward IsGenusLowering;\\
IsGenusPreserving := function( white, black, a, b )\\
return not ( IsGenusRaising( white, black, a, b ) or IsGenusLowering( white, black, a, b ) );\\
end function;
}

\end{document}